\documentclass[11pt]{amsart}

\newif\iffinalrun
\finalruntrue

\newif\ifkuvio
\kuviotrue

\iffinalrun
  \kuviotrue
\fi

\usepackage{amssymb}
\iffinalrun
\else
  \usepackage[notref,notcite]{showkeys}
\fi
\usepackage{stmaryrd} 
\usepackage{epsfig}

\newcommand{\color}[2][1]{}

\ifkuvio
  \usepackage[forcekdg,arrsy]{kuvio}
\fi

\usepackage{hyperref}

\newcommand{\topspace}[1]{\vbox{\vspace*{#1}}}
\renewcommand{\v}{\vspace}

\newcommand{\ie}{i.\,e.}
\newcommand{\eg}{e.\,g.}

\newcommand{\tens}{\otimes}
\newcommand{\leg}[2]{\genfrac{(}{)}{}{}{#1}{#2}}

\newcommand{\A}{\mathbb A}
\renewcommand{\AA}{\mathcal A}
\newcommand{\I}{\ensuremath{\mathcal I}}
\renewcommand{\L}{\mathcal L}
\newcommand{\C}{\mathbb C}
\newcommand{\CC}{\mathcal C}
\newcommand{\Q}{\mathbb Q}
\newcommand{\qb}{\overline\Q}
\newcommand{\qp}{\Q_p}
\newcommand{\ql}{\Q_l}
\newcommand{\qpb}{\overline\Q_p}
\newcommand{\qpnr}{\Q_p^{nr}}
\newcommand{\R}{\mathbb R}
\newcommand{\T}{\mathbb T}
\newcommand{\Z}{\mathbb Z}

\newcommand{\zpb}{\overline\Z_p}
\newcommand{\znp}{\Z_{(Np)}}
\newcommand{\zn}{\Z_{(N)}}
\newcommand{\zl}{\Z_l}
\newcommand{\F}{\mathbb F}
\newcommand{\fp}{\F_p}
\newcommand{\fq}{\F_q}
\newcommand{\fpb}{\overline \F_p}
\newcommand{\fpn}[1]{\F_{p^{#1}}}
\newcommand{\fqn}[1]{\F_{q^{#1}}}
\newcommand{\G}{\mathbb G}

\newcommand{\ZZ}{\widehat{\mathbb Z}}
\newcommand{\HH}{\mathcal H}

\newcommand{\mm}{\mathfrak m}

\newcommand{\f}{\mathfrak f}
\newcommand{\pp}{\mathfrak p}
\newcommand{\qq}{\mathfrak q}

\renewcommand{\O}{{\mathcal O}}
\newcommand{\rO}{{\mathrm O}}
\newcommand{\RR}{\ensuremath{\mathcal R}}
\renewcommand{\SS}{\mathcal S}

\newcommand{\ok}{\O_K}
\renewcommand{\oe}{\O_E}
\newcommand{\oks}{\O_K\s}

\newcommand{\eq}{\Leftrightarrow}
\newcommand{\INTO}{\hookrightarrow}
\newcommand{\onto}{\twoheadrightarrow}
\newcommand{\congto}{\xrightarrow{\,\sim\,}}

\renewcommand{\r}{{}^r}
\newcommand{\s}{^\times}
\renewcommand{\d}{^\ast}
\renewcommand{\ss}{^{ss}}

\newcommand{\dual}{^\vee}
\newcommand{\subtf}{_{\mathit{tor}\textit{-}\mathit{free}}}
\newcommand{\subres}{_{\mathit{res}}}
\newcommand{\subreg}{_{\mathit{reg}}}
\newcommand{\reg}{\mathit{reg}}
\newcommand{\Ser}{\mathit{Ser}}
\newcommand{\ext}{\mathit{ext}}

\newcommand{\mmod}{\mathop{\mathrm{mod}}}

\newcommand{\End}{\mathop{\mathrm{End}}}

\newcommand{\Map}{\mathop{\mathrm{Map}}}
\newcommand{\Gal}{\mathop{\mathrm{Gal}}}
\newcommand{\Stab}{\mathop{\mathrm{Stab}}\nolimits}
\newcommand{\tr}{\operatorname{tr}}
\newcommand{\plim}{\varprojlim}  
\newcommand{\ilim}{\varinjlim}   

\newcommand{\res}{\operatorname{res}}
\newcommand{\BC}{\operatorname{BC}}
\newcommand{\AI}{\operatorname{AI}}
\newcommand{\cores}{\operatorname{cores}}

\newcommand{\ord}{\operatorname{ord}}
\newcommand{\sgn}{\operatorname{sgn}}

\newcommand{\fqrank}{\mathop{\hbox{$\fq$-rank}}}
\newcommand{\opt}{\mathit{opt}}

\DeclareMathOperator{\Det}{det}

\newcommand{\Ind}{\operatorname{Ind}}
\newcommand{\PInd}{\operatorname{PInd}}
\newcommand{\ind}{\operatorname{ind}}
\newcommand{\Mor}{\operatorname{Mor}}
\newcommand{\soc}{\operatorname{soc}}
\newcommand{\ch}{\operatorname{ch}}

\newcommand{\dia}{\operatorname{diag}}
\newcommand{\Frob}{\operatorname{Frob}}
\newcommand{\Sym}{\operatorname{Sym}}
\newcommand{\Supp}{\operatorname{Supp}}
\newcommand{\JH}{\operatorname{JH}}
\newcommand{\GL}{\operatorname{GL}}
\newcommand{\GSp}{\operatorname{GSp}}
\newcommand{\SL}{\operatorname{SL}}
\newcommand{\ssl}{\mathfrak{sl}}

\newcommand{\uo}{{U_1(N)}}

\newcommand{\go}{{\Gamma_1(N)}}
\newcommand{\gotilde}{{\widetilde\Gamma_1(N)}}

\newcommand{\gztilde}{{\widetilde\Gamma_0(N)}}
\newcommand{\gi}{{\Gamma_i(N)}}
\newcommand{\gitilde}{{\widetilde\Gamma_i(N)}}

\newcommand{\so}{{S_1(N)}}

\newcommand{\soptilde}{{\widetilde S_1'(N)}}

\newcommand{\szptilde}{{\widetilde S_0'(N)}}
\newcommand{\sigo}{{\Sigma_1(N)}}
\newcommand{\sip}{{S_i'(N)}}
\newcommand{\siptilde}{{\widetilde S_i'(N)}}

\newcommand{\ho}{\HH_1(N)}
\newcommand{\hop}{\HH'_1(N)}
\newcommand{\hoptilde}{\widetilde\HH'_1(N)}

\newcommand{\hzptilde}{\widetilde\HH'_0(N)}
\newcommand{\ha}{\HH_1^\A(N)}
\newcommand{\hip}{\HH'_i(N)}
\newcommand{\hiptilde}{\widetilde\HH'_i(N)}

\def\varddots{\mathinner{\raise0pt\vbox{\kern3pt\hbox{.}}\mkern1mu\smash{\raise-2pt\hbox{.}}
\mkern1mu\smash{\raise-4pt\hbox{.}}}}

\newcommand{\gln}{\ensuremath{{\GL_n(\fp)}}}
\newcommand{\glnfq}{\ensuremath{{\GL_n(\fq)}}}

\newcommand{\nivonematrix}{\bigg(\begin{smallmatrix}\omega^i \\ & \omega^j \\ && \omega^k
  \end{smallmatrix}\bigg)}

\newcommand{\nivtwomatrix}{\bigg(\begin{smallmatrix} \omega_2^m \\ & \omega_2^{pm} \\ && \omega^i
  \end{smallmatrix}\bigg)}

\newcommand{\nivthreematrix}{\bigg(\begin{smallmatrix} \omega_3^m \\ & \omega_3^{pm}\! \\[-4pt]
    && \omega_3^{p^2m} \end{smallmatrix}\bigg)}

\newcommand{\gpxtw}{\ensuremath{{X(T) \rtimes W}}}
\newcommand{\wxt}{\ensuremath{W \times X(T)}}
\newcommand{\rwmu}{\ensuremath{R_w(\mu)}}

\newcommand{\wmu}{\ensuremath{(w,\mu)}}
\newcommand{\wmup}{\ensuremath{(w',\mu')}}

\newcommand{\sd}{sufficiently deep}
\newcommand{\sdr}{sufficiently deep in a restricted alcove}

\newcommand{\ones}{1_{\mspace{-2mu}S}}

\iffinalrun
  \newcommand{\need}[1]{}
\else
  \newcommand{\need}[1]{{\tiny *** #1}}
\fi
\newcommand{\separator}{\medskip \centerline{$\ast\ \ \ast\ \ \ast$} \medskip}

\renewcommand{\ll}{\llbracket}
\newcommand{\rr}{\rrbracket}

\renewcommand{\(}{\textup{(}}
\renewcommand{\)}{\textup{)}}

\DeclareMathOperator*{\pprod}{{\textstyle\prod}}

\newcommand{\expbull}{{\mathop{\hspace{1.5pt}\begin{picture}(1,1)(0,-2)\circle*{2}\end{picture}\hspace{0.5pt}}}}

\newcommand{\Frn}{\text{$\mathrm{Fr}$-$n$}}
\newcommand{\St}{\operatorname{St}}

\theoremstyle{plain} 
\newtheorem{lm}[equation]{Lemma}
\newtheorem{sublm}[equation]{Sublemma}
\newtheorem{prop}[equation]{Proposition}
\newtheorem{thm}[equation]{Theorem}
\newtheorem{coroll}[equation]{Corollary}
\newtheorem{conj}[equation]{Conjecture}
\theoremstyle{definition}
\newtheorem{df}[equation]{Definition}
\newtheorem{rk}[equation]{Remark}

\newtheoremstyle{jantzen}{3pt}{3pt}{\itshape}{}{\scshape}{.}{.5em}{}
\theoremstyle{jantzen}
\newtheorem*{thmJ}{Theorem}
\newtheorem*{propJ}{Proposition}

\numberwithin{equation}{section}
\numberwithin{figure}{section}

\begin{document}

  \title[The weight in Serre's Conjecture for $\GL_n$]{The weight in a Serre-type conjecture for tame $n$-dimensional Galois representations}
  \author{Florian Herzig}
  \address{Department of Mathematics, 2033 Sheridan Road, Evanston, IL 60208-2730, USA}
  \email{herzig@math.northwestern.edu}

  \begin{abstract}
    We formulate a Serre-type conjecture for $n$-dimensional Galois representations that are tamely ramified at~$p$.
    The weights are predicted using a representation-theoretic recipe. 
    For $n = 3$ some of these weights were not predicted by the previous conjecture of Ash, Doud, Pollack, and Sinnott. Computational evidence
    for these extra weights is provided by calculations of Doud and Pollack. We obtain theoretical evidence for $n = 4$ using automorphic inductions of
    Hecke characters. 
  \end{abstract}

  \maketitle

\section{Introduction}

Serre conjectured in 1973 that every two-dimensional irreducible, odd Galois representation
  $\rho : \Gal(\qb/\Q) \to \GL_2(\fpb)$
  arises from a modular eigenform. He later predicted that some such eigenform occurs in level $\Gamma_1(N^?(\rho))$ and weight
  $k^?(\rho)$, where  $N^?(\rho)$ is a prime-to-$p$ integer measuring the ramification of~$\rho$ outside~$p$, whereas $k^?(\rho) \ge 2$
  was defined by Serre in terms of the restriction of ~$\rho$ to an inertia subgroup~$I_p$ at~$p$ using an essentially combinatorial
  recipe~\cite{bib:SC2}. After important results of Mazur, Ribet, Gross, Taylor, and many others, the conjecture was 
  finally proved by Khare and Wintenberger~\cite{bib:KW1}, \cite{bib:KW2} (and Kisin~\cite{bib:Kisin}).

In this paper we consider $n$-dimensional irreducible, odd Galois representations
\begin{equation*}
  \rho : \Gal(\qb/\Q) \to \GL_n(\fpb)
\end{equation*}
(for ``odd'' see def.~\ref{df:odd}).  Ash, Doud, Pollack, and Sinnott~\cite{bib:ASinn}, \cite{bib:ADP} conjectured that such~$\rho$ arise in
the mod~$p$ group cohomology of $\Gamma_1(N^?(\rho)) \le \SL_n(\Z)$, where $N^?(\rho)$ is the natural analogue of the above.  Eigenvectors in mod~$p$
cohomology under a natural Hecke action are the analogues of mod~$p$ modular eigenforms, with the coefficients playing the role of the
weight.  The basic set of (``coefficient'') weights, the so-called Serre weights, are the irreducible representations of $\gln$ over~$\fpb$
with $\Gamma_1(N^?(\rho))$ acting via reduction mod~$p$. It is thus desirable to describe the \emph{set} of Serre weights in which $\rho$
arises. This actually provides finer information than $k^?(\rho)$ when $n = 2$. For us it will be more convenient to let $W(\rho)$ be the
set of ``regular'' Serre weights (up to twisting this corresponds to excluding $p+1$ among weights $2 \le k \le p+1$ when $n = 2$)
in which $\rho$ arises in \emph{some} prime-to-$p$ level~$N$ (i.e., not just $N = N^?(\rho)$; this is not expected to yield any
  further weights, just as when $n = 2$).

To state our Serre-type conjecture for the weights $W(\rho)$ of~$\rho$, we define a (Deligne--Lusztig) representation
$V(\rho|_{I_p})$ of $\gln$ over~$\qpb$ and an operator~$\RR$ on the set of Serre weights. By $\overline{V(\rho|_{I_p})}$ we denote the
reduction of a $\gln$-stable $\zpb$-lattice inside $V(\rho|_{I_p})$ modulo the maximal ideal and let $\JH(-)$ denote the set of Jordan--H\"older
factors of a composition series.
\begin{conj}\label{conj:intro}
  Suppose that $\rho : \Gal(\qb/\Q) \to \GL_n(\fpb)$ is irreducible, odd, and tamely ramified at~$p$. Then
  $W(\rho) = \RR(\JH(\overline{V(\rho|_{I_p})}))$.
\end{conj}
Let us denote this conjectural weight set by $W^?(\rho|_{I_p})$, noting that it only depends on $\rho|_{I_p}$. When $\rho$ is no longer
tamely ramified at~$p$, i.e., $\rho|_{I_p}$ no longer semisimple, one expects that $\varnothing \ne W(\rho) \subseteq W^?(\rho|_{I_p}^{ss})$.

When $n = 3$ and $\rho|_{I_p}$ is tame, $W^?(\rho|_{I_p})$ contains the set of regular Serre weights specified
in~\cite{bib:ADP} (strictly in most cases); see thm.~\ref{thm:comparison-with-adps}. The set of all regular Serre weights is
essentially the disjoint union of two subsets (according to the ``alcoves'' in the representation theory of algebraic groups
in characteristic~$p$) that are interchanged by~$\RR$. If $\rho|_{I_p}$ is moreover generic, $W^?(\rho|_{I_p})$ consists of 9
weights, 3 lying in the lowest alcove and 6 lying in the other, regardless of what fundamental tame characters $\rho|_{I_p}$
involves (there are three possibilities).
The genericity assumption is a condition on the exponents of tame fundamental characters in $\rho|_{I_p}$ which guarantees
that the predicted weights do not get too close to alcove boundaries. For a precise definition of ``generic'', see
def.~\ref{df:generic_tau}; note that as $p$ tends to infinity the proportion of tame $\rho|_{I_p}$ that are generic tends to
1.  For any~$n$ we obtain an explicit description of $W^?(\rho|_{I_p})$ for generic tame $\rho|_{I_p}$ in terms of the
geometry of alcoves, using results of Jantzen on the decomposition of Deligne--Lusztig representations. Roughly,
$W^?(\rho|_{I_p})$ consists of $n!$~weights, $n$ to an alcove, together with certain higher translates. (The latter dominate
once $n \ge 4$.) See prop.~\ref{prop:generic_pred} and cor.~\ref{cor:no_wts_predicted} for precise statements.

The evidence we obtain for the conjecture is of two kinds. First, when $n = 3$, Doud and Pollack independently verified for us
computationally (up to convincing bounds) for several explicit, tame~$\rho$ (taken mostly from~\cite{bib:ADP}) that $W(\rho)$ contains
those weights predicted by conjecture~\ref{conj:intro} but missing in the predictions of~\cite{bib:ADP}. Doud has moreover verified for some
particular tame~$\rho$ that $\rho$ arises in no regular weights outside $W^?(\rho|_{I_p})$ (at least in level $N^?(\rho)$).

Second, when $n = 4$ we produce many 
odd, tame~$\rho$ and Serre weights~$F$ such
that $F \in W^?(\rho|_{I_p}) \cap W(\rho)$ (see thm.~\ref{thm:gl4_th_evid} and prop.~\ref{prop:GL4_theo_evid} for a precise description of which pairs
$(\rho|_{I_p},F)$ are obtained).  Our method is to obtain first Hecke
eigenvectors in group cohomology with complex coefficients from cohomological automorphic representations of ${\GL_4}_{/\Q}$ whose associated
$p$-adic Galois representation is known, and then to
``reduce mod~$p$.''  We use representations automorphically induced from carefully chosen Hecke characters over non-Galois quartic CM
fields. The main limitations of this method are that essentially only the Serre weights lying
in the lowest alcove can be lifted to characteristic zero (as representations of the ambient algebraic group ${\GL_n}_{/\Q}$)---although
weaker evidence for higher alcoves is obtained---and that the Serre weights have to satisfy a symmetry condition coming from a corresponding
condition on the infinity type of cuspidal, algebraic automorphic representations of ${\GL_n}_{/\Q}$~\cite[p.\ 144]{bib:Clozel}.  The argument also
goes through for $\GL_{2m}$ with $m > 2$ whenever the required automorphic inductions are known to exist. We remark that for $n=3$ a similar method was employed
in~\cite[\S 4]{bib:ADP} using symmetric square liftings of modular forms.

We also show that conjecture~\ref{conj:intro} is compatible with other conjectures. On the one hand we verify for generic tame
$\rho|_{I_p}$ the compatibility with a conjecture of Gee predicting a certain closure property of~$W^?(\rho|_{I_p})$ (see prop.~\ref{prop:gee_evidence}).
On the other hand we show that the predicted weight set in the Serre-type conjecture of Buzzard, Diamond, and Jarvis~\cite{bib:BDJ} (in many cases a theorem
of Gee) for two-dimensional, irreducible,
totally odd, mod~$p$ representations~$\rho$ of the Galois group of a totally real field that is unramified at~$p$ can be expressed completely analogously to
conjecture~\ref{conj:intro} in the tamely ramified case (restricting to regular weights). 
This contrasts with the result of Diamond~\cite{bib:Di} that in this case, the conjectural weight set itself (at a prime dividing~$p$) is
essentially equal to the Jordan--H\"older constituents of the reduction ``mod~$p$'' of an irreducible characteristic zero representation.
The possibility of relating the set of Serre weights of~$\rho$ to the reduction of characteristic zero representations in two ways (with or without~$\RR$) reflects
the fact for $n = 2$ there is just one relevant alcove. 
For $n > 2$ an operator like~$\RR$ is ``necessary,'' as $\RR$ interchanges alcoves with different numbers of predicted Serre weights.

Unfortunately we were unable to formulate a conjecture including the non-regular Serre weights of~$\rho$, but we expect more complicated
boundary phenomena based on considerations of local crystalline lifts. We were able to account for \emph{all} weights predicted by the conjecture
of Buzzard, Diamond, and Jarvis in the tame case by using a multi-valued extension of~$\RR$ (see thm.~\ref{thm:relation_to_bdj}).

Finally let us remark that we formulated many parts of this paper for groups more general than $\GL_n$ in the hope of its future usefulness.
We in fact apply some of the results in the case of $\GSp_4$ in recent work with Jacques Tilouine~\cite{bib:scgsp4}.

The paper is structured as follows. In sections~\ref{sec:glncharp}--\ref{sec:decomp_gln} we review the relevant
representation theory of $\GL_n(\fq)$ (and more general groups) and Jantzen's results on the decomposition ``mod~$p$'' of
Deligne--Lusztig representations. In section~\ref{sec:conj} we define $\RR$, $V(\rho|_{I_p})$, state the conjecture
in~\eqref{conj:serre} and discuss its generic behaviour. Section~\ref{sec:comp_adps} is devoted to a detailed comparison with the
conjecture of Ash, Doud, Pollack, and Sinnott when $n = 3$. We list the computations of Doud and Pollack providing numerical
evidence in section~\ref{sec:comp_evid}. The following section contains the generic compatibility result with the conjecture of
Gee. In section~\ref{sec:theo_evid} we obtain evidence for the conjecture from automorphic representations of~$\GL_4$, and in
section~\ref{sec:hilb_mod} we discuss the compatibility with the conjecture of Buzzard--Diamond--Jarvis. Finally, appendix~\ref{app:gener-jantz-form}
explains how Jantzen's theorem on the decomposition ``mod~$p$'' of
Deligne--Lusztig representations generalises to a larger class of reductive groups that includes $\GL_n$.

\subsection{Acknowledgements}

This paper grew out of my Harvard thesis~\cite{bib:thesis}. 
I am deeply indebted to my adviser, Richard Taylor, for
his invaluable guidance and unfailing support.  I am very grateful to Jens Carsten Jantzen for generalising his result about the decomposition of
Deligne--Lusztig representations and for answering other questions. The computations that Darrin Doud and David Pollack undertook for me were
crucial, and I thank them very much for their willingness to assist me in this way. 
I am grateful to Avner Ash and to Fred Diamond for very helpful discussions and for
their encouragement.  I also want to thank Christophe Breuil, Kevin Buzzard, Ga\"etan Chenevier, Matt Emerton, Toby Gee, Michael Harris, Guy Henniart,
Mark Reeder, Michael Schein, and Teruyoshi Yoshida for helpful comments and discussions. I thank the I.H.\'E.S. for 
the excellent working conditions it provided during my stay in 2006/07 when some of this work was carried out.
Finally I am grateful to the referees for their helpful comments.

  \tableofcontents

\section{Notation}

Throughout, $p$ denotes a prime number and $q = p^r$. Fix an algebraic closure $\qpb$ of~$\qp$ and denote by~$\fpb$ its
residue field. For all~$n$, let $\Q_{p^n} \subseteq \qpb$ denote the unique subfield which is unramified and of degree~$n$
over~$\qp$ and let $\F_{p^n} \subseteq \fpb$ denote the unique subfield of cardinality~$p^n$.

Fix an embedding $\qb \to \qpb$, and let~$G_p$ (resp.\ $I_p$) denote the corresponding decomposition group
(resp.\ inertia group) in $G_\Q = \Gal(\qb/\Q)$. A (choice of) \emph{geometric} Frobenius element at~$l$ will be denoted by $\Frob_l$.
We will normalise the local Artin map so that geometric Frobenius elements correspond to uniformisers.
Let $\widetilde\ : \fpb\s \to \qpb\s$ denote the Teichm\"uller lift.

All Galois representations we consider are assumed to be \emph{continuous}.

\subsection{Hecke pairs and Hecke algebras}\label{sub:hecke}

We will generally use the same terminology as Ash--Stevens \cite{bib:AStev}, but prefer left actions for our modules. Thus a \emph{Hecke
  pair} is a pair $(\Gamma,S)$ consisting of a subgroup~$\Gamma$ and a subsemigroup~$S$ of a fixed ambient group~$G$ such that
\begin{enumerate}
\item $\Gamma \subseteq S$.
\item $\Gamma$ and $s\Gamma s^{-1}$ are commensurable for all $s \in S$.
\end{enumerate}

The \emph{Hecke algebra} $\HH(\Gamma,S)$ is, as an abelian group, the subgroup of left $\Gamma$-invariant elements in the free
abelian group of left cosets $s \Gamma$ ($s \in S$). The multiplication is given by
\[ \sum a_i (s_i\Gamma) \sum b_j (t_j\Gamma) = \sum a_ib_j (s_i t_j\Gamma), \]
where $a_i$, $b_j \in \Z$, $s_i$, $t_j \in S$. In particular, any double coset $\Gamma s\Gamma = \coprod_i s_i\Gamma$ (a finite disjoint union) becomes a Hecke
operator in $\HH(\Gamma,S)$ in the natural way; it is denoted by $[\Gamma s\Gamma]$. If $M$ is a left $S$-module (over any ring),
the group cohomology modules $H^\expbull(\Gamma,M)$ inherit a natural linear action of $\HH(\Gamma,S)$. This action is $\delta$-functorial,
\ie, long exact sequences associated to short exact sequences of $S$-modules are $\HH(\Gamma,S)$-equivariant. It is thus determined by
demanding that
\begin{equation*}
  [\Gamma s\Gamma]m = \sum_i s_i m
\end{equation*}
for all $s \in S$, $m \in H^0(\Gamma, M)$. It is also possible to explicitly describe the action on cocyles in any degree (see~\cite{bib:AStev}, p.~194).

A Hecke pair $(\Gamma_0,S_0)$ is \emph{compatible} with $(\Gamma,S)$ if $\Gamma_0 \subseteq \Gamma$, $S_0 \subseteq S$, $S_0 \Gamma = S$,
and $\Gamma \cap S_0^{-1}S_0 = \Gamma_0$. In this case, it is easy to check that there is a natural injection
\begin{equation*}
  \HH(\Gamma,S) \INTO \HH(\Gamma_0,S_0)
\end{equation*}
induced by restriction from the map on left cosets sending $s_0 \Gamma$ to $s_0 \Gamma_0$ ($s_0 \in S_0$).

It will moreover be convenient to introduce a stronger relation. Let us say that two compatible Hecke algebras $(\Gamma_0,S_0)$, $(\Gamma,S)$ are \emph{strongly
compatible} if $\Gamma s_0 \Gamma = \Gamma_0 s_0 \Gamma$ for all $s_0 \in S_0$ (equivalently, $\Gamma = \Gamma_0(\Gamma \cap s_0 \Gamma s_0^{-1})$ for all $s_0 \in
S_0$). It is easy to see that this is precisely the condition to make the induced injection on Hecke algebras an isomorphism. Note that this isomorphism
identifies $[\Gamma s_0 \Gamma]$ with $[\Gamma_0 s_0 \Gamma_0]$ for all $s_0 \in S_0$.

\section{Representations of $\glnfq$ in characteristic $p$}\label{sec:glncharp}

In this section we will review some relevant results of the modular representation theory of groups like $\glnfq$. The main reference
is~\cite{bib:Jan-reps}.

\subsection{Generalities}

Let~$G_0$ be a connected, \emph{split} reductive group over~$\fp$ and let $G = G_0 \times_{\fp} \fpb$.
Let $T\subseteq G$ be a maximal torus defined and split over~$\fp$ with character group $X(T)$.  Let $R \subseteq X(T)$ be the set of roots of
$(G,T)$. For any $\alpha \in R$, $\alpha^\vee$ denotes the associated coroot.
Choose a set of positive roots~$R^+$ and let $\alpha_i$ denote the simple roots. By $B \supseteq T$ we denote the corresponding Borel subgroup and by~$B^-$ the opposite Borel.
Let $W = N(T)/T$ be the Weyl group of $(G,T)$ and $X(T)_+$ the monoid of dominant weights with respect to our choice of positive roots.

$W$ acts on $X(T)$ via $w\mu := \mu \circ w^{-1}$.  It will be useful in the following to also use a modified action. Choose $\rho'
\in \frac 12(\sum_{\alpha\in R^+} \alpha) + (X(T)\tens \Q)^W$ and define the ``dot action'' of~$W$ by
\begin{equation}\label{eq:dot_action}
  w \cdot \lambda := w(\lambda+\rho')-\rho'.
\end{equation}
Of course, this is independent of the choice of~$\rho'$. Note also that $\langle \rho',\alpha_i\dual\rangle = 1$ for all~$i$.
(In the literature, usually $\rho' = \frac 12\sum_{\alpha\in R^+} \alpha$
is used and denoted by~$\rho$. We prefer to reserve the letter ``$\rho$'' for a more convenient choice of~$\rho'$ in the case of $G=\GL_n$
\eqref{df:rho}.)

Any $\lambda\in X(T)$ can be considered as character of~$B^-$ via the natural map $B^- \onto T$.
For $\lambda\in X(T)_+$ the (dual) \emph{Weyl module} $W(\lambda)$ is defined as algebraic induced module:
\begin{align}
\label{eq:weylmod}  W(\lambda) &= \ind_{B^-}^G (\fpb(\lambda)) \\
\notag  &= \{ f \in \Mor(G,\G_a) : f(bg) = \lambda(b) f(g)\, \forall \,g\in G,\ b\in B^-\}.
\end{align}
(For non-dominant~$\lambda$, this induced module is zero.)
This is a finite-dimensional $\fpb$-vector space, which becomes a left $G$-module in the natural
way:
\[ (xf)(g) = f(gx) \quad \forall \,g, x\in G;\ f\in W(\lambda). \]

Let $F(\lambda) := \soc_G W(\lambda)$ (the socle of the Weyl module, as $G$-module).

\begin{thm}
  The set of simple $G$-modules is $\{F(\lambda) : \lambda\in X(T)_+\}$. If $F(\lambda) \cong F(\mu)$ \($\lambda$, $\mu \in X(T)_+$\)
  then $\lambda = \mu$.
\end{thm}

The formal character map
\begin{equation*}
  \ch : \{ \text{$G$-modules} \} \to \Z[X(T)]^W
\end{equation*}
induces an isomorphism between the Grothendieck group of $G$-modules and $\Z[X(T)]^W$ \cite[II.5.8]{bib:Jan-reps}.
Note that for all $\lambda \in X(T)$ a Weyl module $W(\lambda)$ can be defined 
in the Grothendieck group of $G$-modules \cite[II.5.7]{bib:Jan-reps}:
\begin{equation*}
  W(\lambda) = \sum_i (-1)^i (R^i\ind_{B^-}^G) (\fpb(\lambda)).
\end{equation*}
(If $\lambda$ is dominant, only the $i=0$ term is non-zero, so this agrees with the previous definition.)
The context should always make it clear whether $W(\lambda)$ refers to a genuine representation
(and $\lambda$ dominant) or to an element of the Grothendieck group.
The formal character is given by the Weyl character formula \cite[II.5.10]{bib:Jan-reps}:
\begin{equation}\label{eq:wcf}
  \ch W(\lambda) = \frac{\sum_{w\in W} \det w \cdot e(w(\lambda+\rho'))}{\sum_{w\in W} \det w \cdot e(w(\rho'))} \in \Z[X(T)]^W.
\end{equation}
Here $e(\lambda) \in \Z[X(T)]$ denotes the weight~$\lambda$ considered in the group algebra. In particular
it follows that
\begin{equation}\label{eq:change_of_weyl_chamber}
  W(w\cdot \lambda) = \det(w) W(\lambda),
\end{equation}
and in turn that $W(\lambda) = 0$ if and only if $\lambda+\rho'$ lies on the wall of a Weyl chamber, whereas in all other
cases, this formula allows to express $W(\lambda)$ as $\pm W(\lambda_+)$ with $\lambda_+$ dominant.

\begin{df}\label{df:restr}
  \begin{align*}
      X^0(T) &= \{\lambda \in X(T): \langle \lambda, \alpha^\vee\rangle = 0 \quad \forall\alpha\in R\}. \\
      \intertext{The set of $p^s$-restricted weights is defined to be:}
      X_s(T) &= \{\lambda \in X(T): 0 \le \langle \lambda, \alpha^\vee\rangle < p^s \quad
        \text{for all simple roots $\alpha$}\}.
  \end{align*}
\end{df}

\begin{rk}\label{rk:x0t}
  Note that $X^0(T) = X(T)^W$, by looking at the basic reflections $s_\alpha$ \($\alpha\in R$\) generating~$W$.
  If $\nu \in X^0(T)$ then $W(\nu) = F(\nu)$ is a one-dimensional representation with character $e(\nu)$
  by the Weyl character formula. From~\eqref{eq:weylmod} we get for $\mu \in X(T)_+$,
  \begin{equation*}
    W(\mu + \nu) \cong W(\mu) \tens W(\nu),\ F(\mu + \nu) \cong F(\mu) \tens F(\nu).
  \end{equation*}
\end{rk}

\begin{prop}[Brauer's formula]\label{prop:Brauer_formula}
  If $\sum_{\mu \in X(T)} a_\mu e(\mu) \in \Z[X(T)]^W$, then for all $\lambda \in X(T)$,
  \begin{equation*}
      \ch W(\lambda) \cdot \sum_{\mu \in X(T)} a_\mu e(\mu) = \sum_{\mu \in X(T)} a_\mu \ch W(\lambda+\mu).
  \end{equation*}
\end{prop}

For the simple proof, see for example \cite[\S2(1)]{bib:Jan-Dekompverhalten}.

Let $F_p : G \to G$ denote the $p$-power Frobenius morphism obtained as base change of the absolute Frobenius morphism
of~$G_0$.  For any $i \ge 0$ and any $G$-module~$V$, corresponding to a homomorphism $\rho: G\to \GL(V)$, define a new
$G$-module $V^{(i)}$ which equals~$V$ abstractly but whose $G$-action is obtained by composing $\rho$ with $F_p^i$.

\begin{thm}[Steinberg] \label{thm:Steinberg}
  Suppose $\lambda = \sum_{i=0}^s \lambda_i p^i$ with $\lambda_i \in X_1(T)$. Then
  \begin{equation*}
    F(\lambda) \cong F(\lambda_0) \tens F(\lambda_1)^{(1)} \tens \dots \tens F(\lambda_s)^{(s)}.
  \end{equation*}
\end{thm}

For a proof using the representation theory of Frobenius kernels see \cite[II.3.17]{bib:Jan-reps}.

Now we can state the classification theorem for irreducible modular representations of $G_0(\fq)$, at least under a mild condition on $G$.
The theorem is a slight extension of the one in~\cite[app.\,1]{bib:Jan-expo}, where in addition $G$ is assumed to be semisimple.
For the proof see prop.~1.3 in the appendix.

\begin{thm}\label{thm:mod_reps}
  Suppose that $G$ has simply connected derived group \(\eg, $G = \GL_n$\). Recall that $q = p^r$.
  
  \begin{enumerate}
  \item If $\lambda \in X_r(T)$, $F(\lambda)$ is irreducible as representation of $G_0(\fq)$. Any irreducible representation
    of $G_0(\fq)$ over~$\fpb$ arises in this way.
  \item $F(\lambda) \cong F(\mu)$ as representation of $G_0(\fq)$ if and only if $\lambda-\mu \in (q-1)X^0(T)$.
  \end{enumerate}
\end{thm}

\subsection{Alcoves and the decomposition of Weyl modules}

The \emph{affine Weyl group} $W_p := p\Z R \rtimes W$ and the extended affine Weyl group $\widetilde W_p := pX(T) \rtimes W$
are defined with respect to the natural action of $W$ on $\Z R \subseteq X(T)$. We 
identify them with their images in the group of affine linear automorphisms of $X(T)\tens \R$ as follows:
\begin{equation*}
  (p\nu,w) \cdot \lambda :=  w\cdot \lambda + p\nu
\end{equation*}
(using the dot-action~\eqref{eq:dot_action}). For any $\alpha \in R$ and any $n \in \Z$ there is an affine reflection on $X(T) \tens \R$,
\[ s_{\alpha,np}(\lambda) = \lambda - (\langle \lambda + \rho', \alpha\dual \rangle - np) \alpha. \]
Note that the $s_{\alpha,np}$ generate $W_p$. For $\alpha \in R$ and any $n \in \Z$ denote by
\begin{equation}
  \label{eq:hyperplanes}
  H_{\alpha,np} = \{ \lambda : \langle \lambda + \rho', \alpha\dual \rangle = np \}
\end{equation}
the affine hyperplane fixed by $s_{\alpha,np}$.

\begin{df}
  An \emph{alcove} is a connected component of the complement of these affine hyperplanes in $X(T)\tens\R$.
\end{df}

In particular there is the \emph{``lowest alcove''}
\begin{equation*}
  C_0 = \{ \lambda :  0 < \langle \lambda + \rho', \alpha\dual \rangle < p\quad\forall\alpha\in R^+\}.
\end{equation*}
It can easily be checked that $W_p$ and even $\widetilde W_p$
map alcoves to alcoves; in fact, $\overline{C_0}$ is a fundamental domain for the $W_p$-action.

\begin{df}
  An alcove $C$ is \emph{dominant} if it is contained in
  \begin{equation*}
      \{ \lambda : 0 < \langle \lambda + \rho', \alpha_i\dual \rangle \quad\forall i \}.
  \end{equation*}

  An alcove $C$ is \emph{restricted} if it is contained in the restricted region
  \begin{equation}\label{eq:A_res}
      A\subres = \{ \lambda : 0 < \langle \lambda + \rho', \alpha_i\dual \rangle < p \quad\forall i \}.
  \end{equation}
\end{df}
Recall that the $\alpha_i$ denote the simple roots. Note that the restricted region $A\subres$ is related to the set of
$p$-restricted weights~\eqref{df:restr} as follows:
\begin{equation*}
  X(T)\cap A\subres \subseteq X_1(T) \subseteq X(T)\cap \overline{A\subres}.
\end{equation*}
Also, it is clear from the definition that $\overline{A\subres}$ is a union of closures of alcoves.

\begin{df}\ \label{df:uparrow}

  \begin{enumerate}
  \item Suppose that $\lambda$, $\mu \in X(T)$. We will say $\lambda \uparrow \mu$ if there exist
    $s_i := s_{\alpha_i,pn_i} \in W_p$ with $\alpha_i \in R$, $n_i \in \Z$ for $1 \le i\le r$ such that
    \[ \lambda \le s_1 \cdot \lambda \le s_2s_1 \cdot \lambda \le \dots \le s_r\dots s_1 \cdot \lambda = \mu. \]
  \item Suppose that $C_0 \cap X(T) \ne \varnothing$. Given alcoves $C$, $C'$, pick $\lambda \in C$ and let
    $\lambda'$ be the unique element of $W_p \cdot \lambda \cap C'$. Then
    \[ C \uparrow C'\ :\Longleftrightarrow\ \lambda \uparrow \lambda'. \]
  \end{enumerate}
\end{df}

Note that
\begin{equation*}
  \lambda \uparrow \mu \ \Longrightarrow\ \lambda \le \mu \ \text{and}\ \lambda \in W_p \cdot \mu
\end{equation*}
but the converse does not hold in general. One verifies that the second part of the definition is independent of the choice of~$\lambda$.
There is a natural definition even if $C_0$ contains no weights \cite[II.6.5]{bib:Jan-reps}. In any case, $C_0$ is the lowest dominant
alcove with respect to~$\uparrow$. If $C \uparrow C'$ we will also say that $C$ lies \emph{below} alcove~$C'$ and $C'$ \emph{above}~$C$.

The following result, the so-called ``strong linkage principle'' of Jantzen and Andersen \cite[II.6.13]{bib:Jan-reps}, is crucial
in the representation theory of reductive groups in prime characteristic.

\begin{prop}\label{prop:stronglinking}
  Suppose that $\lambda$, $\mu \in X(T)_+$ and that $F(\lambda)$ is a constituent of $W(\mu)$. Then $\lambda \uparrow \mu$.
\end{prop}

\subsection{The case of $\GL_n$}\label{sub:case-gl_n}

To apply these results to $\GL_n$, let $T$ be the diagonal matrices and $B$ the
upper-triangular matrices. Denote by $\epsilon_i\in X(T)$ the character
\begin{equation*}
  \Biggl(\begin{smallmatrix} t_1 \\ & t_2 \\[-5pt] && \ddots \\ &&& t_n \end{smallmatrix}\Biggr) \mapsto t_i,
\end{equation*}
and we identify $X(T)$ with~$\Z^n$, also writing $(a_1,a_2,\dots,a_n)$ for $\sum a_i\epsilon_i$. Then $R = \{\epsilon_i-\epsilon_j : i\ne j\}$
and the simple roots are given by $\alpha_i = \epsilon_i-\epsilon_{i+1}$ for $1\le i\le n-1$. The coroot $(\epsilon_i-\epsilon_j)\dual$ for $i\ne j$
then sends $t$ to a diagonal matrix whose only entries are 1's except for a $t$ in the $(i,i)$-entry and a $t^{-1}$ in the $(j,j)$-entry. We will identify $W$
with~$S_n$ so that $w(\epsilon_i) = \epsilon_{w(i)}$.

Then $X^0(T) = (1,\dots,1)\Z$, $X_r(T) = \{(a_1,\dots,a_n):0\le a_i-a_{i+1}\le q-1\ \forall i\}$,
$(a_1,\dots,a_n)$ is dominant if and only if $a_1 \ge \dots \ge a_n$. We may choose $\rho' = (n-1,n-2,\dots,1,0)$.

\begin{coroll}\ \label{cor:gln_modules}

  \begin{enumerate}
  \item The irreducible $\GL_n$-modules over~$\fpb$ are the $F(a_1,\dots,a_n)$, $a_1 \ge \dots \ge a_n$.
  \item The irreducible representations of $\glnfq$ over~$\fpb$ are the $F(a_1,\dots,a_n)$, $0 \le a_i-a_{i+1} \le q-1\ \forall i$.
    $F(a_1,\dots,a_n) \cong F(a_1',\dots,a_n')$ if and only if $(a_1,\dots,a_n) - (a_1',\dots,a_n') \in (q-1,\dots,q-1)\Z$.
  \item Any irreducible representation of $\glnfq$ over~$\fpb$ can be written as
    \[ M_0 \tens_{\fpb} M_1^{(1)} \tens_{\fpb} \dots \tens_{\fpb} M_{r-1}^{(r-1)} \]
    for unique irreducible representations $M_i = F(\lambda_i)$ with $\lambda_i \in X_1(T)$.
  \end{enumerate}
\end{coroll}

The number of restricted alcoves is $(n-1)!$ (see \S\ref{sub:jantzens-theorem}).

\emph{Suppose that $n = 2$.} The only restricted alcove is $C_0 = \{(a,b) \in \R^2 : -1 < a-b < p-1 \}$. If $(a,b) \in X_1(T)$,
we claim that $F(a,b) \cong \Sym^{a-b} \fpb^2 \tens \det^b$. First note that $F(a,b) = W(a,b)$ by the strong linkage principle~\eqref{prop:stronglinking}.
For any homogeneous polynomial~$F$ of degree $a-b$,
$
\big(\begin{smallmatrix}
  x_1& x_2 \\ x_3 & x_4
\end{smallmatrix}\big) \mapsto (x_1x_4-x_2x_3)^b F(x_1,x_2)$ is in $W(a,b)$ and these elements form a
subrepresentation isomorphic to $\Sym^{a-b} \fpb^2 \tens \det^b$.
By irreducibility the claim follows.

\emph{Suppose that $n = 3$.} The two restricted alcoves
are the \emph{``lower alcove''}
\begin{equation*}
  C_0 = \{(a,b,c)-\rho'\in\R^3 : 0 < a-b,\; b-c\quad\text{and}\quad a-c < p \}
\end{equation*}
and the \emph{``upper alcove''}
\begin{equation*}
  C_1 := \{(a,b,c)-\rho'\in\R^3 : p < a-c \quad\text{and}\quad a-b,\; b-c < p \}.
\end{equation*}

\begin{prop}[Jantzen]\label{prop:weyl_simple} Suppose that $(x,y,z) \in X_1(T)$.

  \begin{enumerate}
  \item If $(x,y,z)$ is in the upper alcove
    then there is a \(non-split\) exact sequence
    \[ 0 \to F(x,y,z) \to W(x,y,z) \to F(z+p-2,y,x-p+2) \to 0. \]
  \item Otherwise, \ie, if $(a,b,c)$ is in the lower alcove or on the boundary of the upper alcove,
    $F(x,y,z) = W(x,y,z)$.
  \end{enumerate}

\end{prop}

\textbf{Notation:}\label{not:\r} $\r (x,y,z) = (z+p-2,y,x-p+2)$.

\need{really need that notation?}

\begin{proof} (ii) follows from the strong linkage principle (\ref{prop:stronglinking}). (i) is a consequence of prop.~II.7.11 and lemma~II.7.15 \cite{bib:Jan-reps}:
  let $\lambda = (x,y,z)$. Then $\r\lambda = (z+p-2,y,x-p+2)$ is the unique weight which is strictly smaller than~$\lambda$ in the
  $\uparrow$-ordering of $X(T)$. Pick a weight~$\mu$ in the upper closure of the lower alcove, but not in the lower alcove itself (\eg\ 
  $\mu = (p-2,0,0)$), and apply the translation functor $T_{\r\!\lambda}^\mu$ to the identity of formal characters
\[ \ch W(\lambda) = \ch F(\lambda) + m\ch F(\r\lambda), \]
which holds for some integer~$m$ by the strong linkage principle, to deduce that $m = 1$. \end{proof}

\emph{Suppose that $n = 4$.} 
Here is a list of all dominant alcoves below the top restricted one ($C_5$).
They consist of all $(a,b,c,d)-\rho' \in X(T)\tens \R = \R^4$ satisfying respectively:
\begin{align*}
  C_0:\quad & 0 < a-b,\;b-c,\;c-d;\ a-d < p, \\
  C_1:\quad & 0 < b-c;\ p < a-d;\ a-c,\;b-d < p, \\
  C_2:\quad & 0 < c-d;\ p < a-c;\ a-b,\;b-d < p, \\
  C_3:\quad & 0 < a-b;\ p < b-d;\ c-d,\;a-c < p, \\
  C_4:\quad & p < a-c,\;b-d;\ b-c < p;\ a-d < 2p, \\
  C_5:\quad & 2p < a-d;\ a-b,\;b-c,\;c-d < p, \\
  C_{0'}:\quad & 0 < b-c,\;c-d;\ p < a-b;\ a-d < 2p, \\
  C_{0''}:\quad & 0 < a-b,\;b-c;\ p < c-d;\ a-d < 2p.
\end{align*}
\need{???$C_i$ will also be called \emph{alcove~$i$}. As $C_0$ is the lowest alcove, the notation is compatible with the above.???}

The first six alcoves in this list are the restricted ones. The $\uparrow$-ordering on the above eight alcoves of~$\GL_4$ is 
generated by $0 \uparrow 1 \uparrow i \uparrow 4\uparrow 5$ ($i = 2$, 3), $2 \uparrow 0' \uparrow 5$, and
$3 \uparrow 0'' \uparrow 5$.

The constituents of $W(\lambda)$ for $\lambda \in X_1(T)$ are known by~\cite{bib:Jan-Charformel}.

\section{Representations of $\glnfq$ in characteristic zero}\label{sec:glnchar0}

The aim of this section is to recall relevant facts about the  ordinary representations theory of~$\glnfq$. Since it will be convenient for reduction later,
we will work over the field~$\qpb$.

\subsection{Deligne--Lusztig representations}

We allow $G$ to be slightly more general than in the previous section: $G = G_0 \times_{\fq} \fpb$ for a connected reductive group $G_0$
over~$\fq$.  We will identify a variety over~$\fpb$ with the set of its $\fpb$-rational points.  Let $F$ denote the ($q$-power)
Frobenius morphism, so $G^F = G_0(\fq)$. We assume that the maximal torus $T$ is $F$-stable (not necessarily split over~$\fq$).

To each pair $(\T,\theta)$ consisting of an $F$-stable maximal torus~$\T$ and a homomorphism $\theta : \T^F \to \qpb\s$, Deligne--Lusztig \cite{bib:DL} associate
a virtual representation $R_\T^\theta$ of~$G^F$ (defined in terms of the \'etale cohomology of a variety over~$\fpb$ having commuting
$\T^F$- and $G^F$-actions). We will recall the relevant facts, together with Jantzen's parameterisation \cite[3.1]{bib:Jan-DL}.

Given $w \in W$, by Lang's theorem there is a $g_w \in G$ such that $g_w^{-1} F(g_w)$ is a lift of~$w$ in $N(T)$. Then $T_w := {}^{g_w} T( = g_w T g_w^{-1})$ is
an $F$-stable maximal torus, well defined up to $G^F$-conjugacy. Two elements $w$, $w' \in W$ are said to be $F$-conjugate if $w =
\sigma^{-1} w' F(\sigma)$ for some $\sigma \in W$ (note that the natural $F$-action on~$W$ is trivial if $G$ is split over~$\fq$). The
map sending $w \in W$ to~$T_w$ induces a bijection between $F$-conjugacy classes in~$W$ and $G^F$-conjugacy classes of $F$-stable maximal
tori. We say that the \emph{type} of~$T_w$ is (the $F$-conjugacy class of)~$w$.

If $\mu \in X(T)$ let
\begin{align*}
  \theta_{w,\mu} : T_w^F &\to \qpb\s \\
               t_w &\mapsto \widetilde\mu(g_w^{-1}t_w g_w).
\end{align*}
(Recall that $\widetilde\ $ denotes the Teichm\"uller lift.) 
Form the semi-direct product \gpxtw\ where $w\in W$ acts on $\mu \in X(T)$ as $F(w)(\mu)$. The group \gpxtw\ acts on the set \wxt\ as follows:
\begin{equation*}
 {}^{(\nu,\sigma)}\wmu = (\sigma w F(\sigma)^{-1},\sigma\mu + F(\nu)-\sigma w F(\sigma)^{-1} \nu).
\end{equation*}
In particular if $G$ is split over~$\fq$, $W$ acts on $X(T)$ in the natural way and the \gpxtw-action becomes
\begin{equation}\label{eq:jantzen-action}
 {}^{(\nu,\sigma)}\wmu = (\sigma w \sigma^{-1},\sigma\mu + (q-\sigma w \sigma^{-1}) \nu).
\end{equation}
We will also use the notation $\wmu \sim \wmup$ for  elements of $\wxt$ in the same $\gpxtw$-orbit.

\begin{lm}\label{lm:dl_reps}\label{x}  
  \begin{alignat*}{2}
    \frac{\wxt}{\gpxtw} &\congto \frac{\{ \text{pairs $(\T,\theta)$} \}}{\text{$G^F$-conjugacy}} &&\INTO 
    \bigg\{\begin{array}[h]{@{}c@{}}
      \text{virtual representations} \\
      \text{of $G^F$ over $\qpb$}
    \end{array} \bigg\} / \cong \\
    (w,\mu)\quad &\mapsto (T_w,\theta_{w,\mu}); (\T,\theta) &&\mapsto \qquad \epsilon_G\epsilon_\T R_\T^\theta
  \end{alignat*}
  where $\epsilon_G = (-1)^{\text{\upshape $\fq$-rank(G)}}$ and $\epsilon_{\T} = (-1)^{\text{\upshape $\fq$-rank($\T$)}}$.
  If $(\T_i,\theta_i)$ are not $G^F$-conjugate, then $\langle
  R_{\T_1}^{\theta_1},R_{\T_2}^{\theta_2}\rangle = 0$.
\end{lm}

Following Jantzen, we denote the image of $\wmu$ under the composite of these maps by $\rwmu$. The choice of sign ensures that
the character value at~1 is positive.

\begin{proof}
  It is elementary to establish the bijection (the key point is \cite[13.7(i)]{bib:Digne-Michel}). \cite[6.8]{bib:DL} implies that the
  second arrow is well defined and the claim about orthogonality which, in turn, entails the injectivity.
\end{proof}

\subsection{The case of $\GL_n$}\label{sub:case-gl_n-char0}

We let $G = \GL_n$ and keep the notation of \S\ref{sub:case-gl_n}.

For any decomposition $n = \sum_{i=1}^r n_i$ with $n_i >0$ there is a corresponding ``parabolic'' subgroup $P_{\vec n}(\fq)$ in $G^F =
\glnfq$ consisting of matrices with $n_i \times n_i$ square blocks along the diagonal (in that order) with arbitrary entries above the
blocks and zeroes below.

\begin{df}
  Suppose that $n = \sum_{i=1}^r n_i$ and for all~$i$, $\sigma_i$ is a representation of $\GL_{n_i}(\fq)$ over~$\qpb$. 
  The \emph{parabolic induction} of the~$\sigma_i$ is defined by
  \[ \PInd(\sigma_1,\dots,\sigma_r) := \Ind_{P_{\vec n}(\fq)}^\glnfq (\sigma_1\tens \dots \tens \sigma_r). \]
\end{df}

It is independent of the order of the $(n_i,\sigma_i)$. An irreducible representation~$\pi$ of $\glnfq$ (over~$\qpb$) is
called \emph{cuspidal} if $\pi$ does not occur in any parabolic induction $\PInd(\sigma_1,\dots,\sigma_r)$ with $r > 1$. For
any~$\pi$ there is a set $\Supp(\pi) = \{\sigma_1,\dots,\sigma_r\}$ uniquely determined by demanding that each~$\sigma_i$ is
cuspidal and that $\pi$ occurs in $\PInd(\sigma_1,\dots,\sigma_r)$.  (See \eg\ \cite{bib:Bump}, ex.\ 4.1.17--20.)

If $l/k$ is an extension of finite fields and $A$ an abelian group, we say that a homomorphism $l\s \to A$
is \emph{$k$-primitive} if it does not factor
through the norm map $l\s \to k_0\s$ for any intermediate field $k \subseteq k_0 \subsetneq l$. More generally, for extensions
$l_i/k$ we say a homomorphism $\prod l_i\s \to A$ is \emph{$k$-primitive} if each component $l_i\s \to A$ is.

\begin{lm}\label{lm:cusp}
  Suppose that $w \in W$ is an $n$-cycle. Since $T_w^F \cong T^{wF}$ via~$g_w$, there is an identification
  $T_w^F \congto \F_{q^n}\s$, determined up to the action of the $q$-power map. Then
  \begin{align}\notag
    \Bigg\{\begin{array}[h]{@{}c@{}}
      \text{$\fq$-primitive} \\
      \text{$\F_{q^n}\s \xrightarrow{\theta} \qpb\s$}
    \end{array} \Bigg\}/ (\theta \sim \theta^q)
    &\overset{\,\sim\,}{\longrightarrow}
    \Bigg\{\begin{array}[h]{@{}c@{}}
      \text{cuspidal representations} \\[1mm]
      \text{of $\GL_n(\fq)$ over $\qpb$}
    \end{array} \Bigg\} / \cong \\ \label{eq:cusp}
     [\theta] \hspace{2cm} & \longmapsto \hspace{2.1cm} (-1)^{n-1} R_{T_w}^\theta.
  \end{align}
\end{lm}

\begin{proof}
  Note that as $w$ is an $n$-cycle, $T_w$ has $\fqrank$ one and hence is not contained in any proper $F$-stable parabolic subgroup.
  Also, no non-trivial element of $(N(T_w)/T_w)^F$ (a cyclic group of order~$n$) fixes~$\theta$,
  as $\theta$ is $\fq$-primitive. Then (5.15), (7.4) and (8.3) of~\cite{bib:DL} show that $R_w(\mu)$ is cuspidal.

  The map is well defined and injective by lemma~\ref{lm:dl_reps}, noting that the $G^F$ conjugacy class of the pair $(T_w,\theta)$
  determines~$\theta$ up to $(N(T_w)/T_w)^F$, i.e., up to $q$-power action.

  For surjectivity we use \cite{bib:Springer_special}. First note that a character is in the discrete series in Springer's nomenclature
  \cite[\S4.3]{bib:Springer_cusp} if and only if it is cuspidal \cite[9.1.2]{bib:Carter}.  Theorems~8.6 and~7.12
  in~\cite{bib:Springer_special} show that the cuspidal characters are precisely the ones denoted there by $\chi_n(\phi)$, for
  $\fq$-primitive characters $\phi : \fqn n\s \to \qpb\s$ (and $\fqn n\s$ naturally embedded in $\GL_n(\fq)$; the image is denoted by~$T_n$
  in Springer's notes), with $\chi_n(\phi) = \chi_n(\phi')$ if and only if $\phi$ is in the $q$-power orbit of~$\phi'$. As the two
  constructions yield the same number of cuspidal representations and Springer shows that he constructs them all, we are done. (It is true
  that $\chi_n(\phi) = (-1)^{n-1}R_{T_w}^\phi$ \cite[\S2.1]{bib:thesis}.)
\end{proof}

\begin{df}\label{df:cusp}
  Denote the cuspidal representation parameterised by~$\theta$ by $\kappa(\theta)$. It follows from lemma~\ref{lm:dl_reps} that it is independent of~$w$.
\end{df}

\begin{lm}\label{lm:para_ind}
  Suppose that $w \in W \cong S_n$. Write $\{1,\dots,n\} = \coprod S_i$ as disjoint union of orbits under the action of~$w$ and
  let $n_i := \# S_i$.
  Via~$g_w$ there is an identification $T_w^F \congto \prod \fqn {n_i}\s$, well defined up to the action of the $q$-power
  map on each component. Suppose that $\theta : T_w^F \to \qpb\s$ is $\fq$-primitive, and denote by $\theta_i : \fqn {n_i}\s \to \qpb\s$
  its $i$-th component. Then
  \[ R_{T_w}^\theta \cong \PInd(\kappa(\theta_1),\dots,\kappa(\theta_r)). \]
\end{lm}

\begin{proof}
  First let $P$ be the parabolic subgroup consisting of $x \in \GL_n$ with $x_{\alpha,\beta} = 0$ whenever $\alpha \in S_i$, $\beta \in S_j$ and $i > j$.
  Similarly let $L$ be the Levi subgroup of~$P$ defined by
  $x_{\alpha,\beta} = 0$ if $i \ne j$. Then $P_w = g_w P g_w^{-1}$ is an $F$-stable parabolic subgroup containing~$T_w$
  (as $P$ is $wF$-stable), and $L_w = g_w L g_w^{-1}$ is an $F$-stable Levi subgroup. From \cite[8.2]{bib:DL},
  \[ R_{T_w}^\theta \cong \Ind_{P_w^F}^{G^F} (R_{T_w,P_w}^\theta) \]
  where the Deligne--Lusztig representation $R_{T_w,P_w}^\theta$ is computed inside~$L_w$ and which becomes a representation of $P_w^F$ via
  $P_w^F \onto L_w^F$.

  But as $n_w \in L$, without loss of generality $g_w \in L$ (Lang's theorem) in which case $L = L_w$, $P = P_w$ and $P_w^F$ is
  $\glnfq$-conjugate to $P_{\vec n}(\fq)$ considered above. Finally $L$ decomposes as $\prod \GL_{n_i}$ (as $\fq$-group) compatibly with the decomposition
  of~$w$ and~$\theta$. An application of K\"unneth's theorem yields the result.
\end{proof}

\section{Decomposition of $\glnfq$-representations}\label{sec:decomp_gln}

Suppose that $V/\qpb$ is a finite-dimensional representation of a finite group~$\Gamma$. Then we can
define the (semisimplified) reduction of~$V$ ``modulo~$p$'' to be $\overline V := (M/\mm_{\zpb} M)\ss$ for
any $\Gamma$-stable $\zpb$-lattice $M \subseteq V$.  This is a semisimple representation over~$\fpb$ which, by the
Brauer--Nesbitt theorem, is independent of the choice of~$M$.

\subsection{Jantzen's formula}\label{sub:jantzens-theorem}

In order to state Jantzen's theorem on the decomposition of Deligne--Lusztig representations mod~$p$ in the special
case of $\GL_n$, we will need to introduce some notation.

As $G' = \SL_n$ is simply connected, for any simple root~$\alpha$ there is a $\omega'_\alpha \in X(T)$ such that $\langle
\omega'_\alpha,\beta\dual\rangle = \delta_{\alpha\beta}$ for all simple roots~$\beta$.  These are unique up to $X^0(T) = X(T)^W$; in fact,
$X(T) = X'(T) \oplus X^0(T)$ where $X'(T)$ is the sublattice spanned by the $\omega'_\alpha$. A possible choice is $\omega'_{\alpha_i} =
\epsilon_1+\dots+\epsilon_i$ ($1 \le i \le n-1$). Note that $A\subres$ \eqref{eq:A_res} is a fundamental domain for the translation action of
$pX'(T)$ on $X(T)\tens \R$\label{A_res_fund_dom}. Hence for any $\sigma \in W$ there is a unique $\rho'_\sigma \in X'(T)$ such that $\sigma
\cdot C_0 + p \rho'_\sigma$ is a restricted alcove. A simple argument shows that
\begin{equation}\label{eq:rho_sig}
  \rho'_\sigma = \sum_{\substack{\text{$\alpha$ simple}\\ \sigma^{-1}(\alpha) < 0}} \omega'_\alpha
\end{equation}
\cite[lemma 1]{bib:Jan-Dekompverhalten}. Denoting the longest Weyl group element by~$w_0$ we define, compatibly with \S\ref{sec:glncharp},
\begin{equation*}
  \rho' := \rho'_{w_0} = \sum_{\text{$\alpha$ simple}} \omega'_\alpha \in \frac 12\sum_{\alpha\in R^+} \alpha + (X(T)\tens \Q)^W.
\end{equation*}
Let $\varepsilon'_\sigma := \sigma^{-1}\rho'_\sigma$ and define
\begin{equation*}
  W_1 = \{ \sigma \in W: \sigma \cdot C_0 + p \rho'_\sigma = C_0 \}.
\end{equation*}
Via the dot action, $W$ acts on the set of alcoves modulo translations by elements of $pX'(T)$ (equivalently, on the set of restricted alcoves).
The stabiliser of~$C_0$ is $W_1$ by definition, and we see that the number of restricted alcoves is $(W:W_1)$.
It is not hard to see that $W_1$ is generated by $(1\; 2\;\dots\; n)$ (with the notation of~\S\ref{sec:glncharp}), so that
there are $(n-1)!$ restricted alcoves. \label{nr_of_alcoves}(For the root system of a simply connected group, $W_1$ is isomorphic to
the root lattice modulo the weight lattice; see \cite{bib:Jan-Dekompverhalten}, lemmas~3 and~2.)

It is known that the matrix
\begin{equation*}
  (\det(\tau)\ch W(-\varepsilon'_{w_0 \sigma} + \varepsilon'_\tau - \rho'))_{\sigma,\tau\in W}
\end{equation*}
with entries in $\Z[X(T)]^W$ is upper triangular with respect to some ordering of~$W$ (not unique). It is easy to see
that its diagonal entries are invertible, as the highest weight of the Weyl module is in $X^0(T)$ if $\sigma = \tau$. Denote by
$\gamma'_{\sigma,\tau}$ the entries of the inverse matrix. These depend on the choice of the $\omega'_\alpha$ (in a simple way).
For more details and references see the appendix, \S3.3.

Not very much seems to be known about the matrix $(\gamma'_{\sigma,\tau})$; it is known to be diagonal if and only if $n \le 3$ \cite{bib:Hulsurkar}.

\begin{thm}[Jantzen]\label{thm:jantzen_formula}
In the Grothendieck group of $\glnfq$-modules,
  \begin{equation*}
    \overline{R_w(\mu+\rho')} = \sum_{\sigma,\tau \in W} \gamma'_{\sigma,\tau} W(\sigma\cdot(\mu-w\varepsilon'_{w_0 \tau}) + q\rho'_\sigma).
  \end{equation*}
\end{thm}

\begin{rk}
  The formula is easily seen to be independent of the choice of the $\omega'_\alpha$. On the other hand, the left-hand side
  depends only on the $\gpxtw$-orbit of $(w,\mu+\rho')$ \eqref{lm:dl_reps} which is not obvious on the right-hand side.
\end{rk}

\begin{rk}\label{rk:jantzen_formula} For the proof see thm.~3.4 in the appendix.
  Originally Jantzen proved the analogue of this
  theorem for simply-connected, quasi-simple groups defined and not necessarily split over a finite field~\cite{bib:Jan-DL}. In fact, the
  above formula nearly follows from the one for $\SL_n$: each ingredient in the formula restricts to its counterpart for $\SL_n$. The only
  loss of generality is that for an irreducible representation~$F$ of $\glnfq$ appearing as Jordan--H\"older constituent of
  $\overline{R_w(\mu+\rho')}$, $F|_{\SL_n(\fq)}$ determines $F$ only up to determinant-power twist.  Taking into account the central character of
  $R_w(\mu+\rho')$, $F$ is determined up to a twist by $\det^r$ for integer multiples $r$ of $(q-1)/n$. Thus if $\gcd(q-1,n) = 1$, the
  formula follows from the one for $\SL_n$. 
\end{rk}

Let us analyse the statement of Jantzen's formula a little when $q = p$. Notice first that a typical highest weight appearing,
$\sigma(\mu-w\varepsilon'_{w_0 \tau}) + p\rho'_\sigma-\rho'$, is a small deformation of $\sigma\cdot \mu + p\rho'_\sigma$. 
If $\mu$ lies in alcove~$C$, the latter weight is contained in alcove $\sigma\cdot C + p\rho'_\sigma$. This alcove is automatically restricted if $C=C_0$,
which can always be achieved, up to a small error, by varying \wmu\ (see~\eqref{eq:jantzen-action}). We will continue to assume that $\mu$ lies in a small
neighbourhood of~$C_0$.

To use Jantzen's formula to find the complete decomposition of $\overline{\rwmu}$ into irreducible \gln-modules, we use Brauer's
formula~\eqref{prop:Brauer_formula} to express each $\gamma'_{\sigma,\tau} W(\lambda)$ as a linear combination of Weyl modules, thus
\begin{equation}\label{eq:2}\notag
  \overline{R_w(\mu)} = \sum_\nu a_\nu W(\nu),\ \text{some $a_\nu \in \Z$.}
\end{equation}
There is a small neighbourhood of the restricted region which contains all~$\nu$ occurring in this expression. Any non-dominant $W(\nu)$ can
be converted into a dominant one using~\eqref{eq:change_of_weyl_chamber}. Next, one has to decompose each $W(\nu)$ as $\GL_n$-module. This is a
difficult problem which has not been solved in general (\S\ref{sec:glncharp}), but in any case the possible highest weights of constituents
are controlled by the strong linkage principle. In particular, these are close to the boundary of their alcove if the same is true for~$\nu$.

Finally to decompose these as representations of \gln, one uses the Steinberg tensor product theorem~\eqref{thm:Steinberg} and Brauer's
formula~\eqref{prop:Brauer_formula}, noting that the Frobenius endomorphism is trivial on \gln.

\subsection{The generic case ($q = p$)}

In generic situations Jantzen found a way to describe the Jordan--H\"older constituents of $\overline {R_w(\mu+\rho')}$ (including multiplicities) in terms
of the constituents of certain induced modules of $G_r T \subseteq G$ ($G_r$ being the kernel of the Frobenius morphism~$F$). When $r = 1$, that is when $q = p$,
---and if we disregard multiplicities which will not concern us anyway---his result can be made completely explicit.

Note first that alcoves for varying~$p$ can naturally be identified with each other: using the isomorphism $X(T) \tens \R \to X(T) \tens \R$,
$\mu - \rho' \mapsto \mu/p - \rho'$, alcoves are described independently of~$p$. For example, we can identify the lowest alcove $C_0$ for each~$p$.

We will say that $\mu \in X(T)$ lies \emph{$\delta$-deep in an alcove~$C$} if
\begin{equation}\label{eq:suff_deep}
  n_\alpha p + \delta < \langle \mu+\rho',\alpha\dual \rangle <  (n_\alpha + 1)p - \delta  \quad\forall\alpha \in R^+
\end{equation}
where $C$ is the alcove determined by putting $\delta = 0$ in these inequalities ($n_\alpha \in \Z$). A statement in which $p$ is allowed to vary is
said to be true for $\mu$ \emph{\sd} in some alcove~$C$ if there is a $\delta > 0$, independent of~$p$, such that the statement is true
for all $\delta$-deep $\mu \in C$.

\begin{lm}\label{lm:extended-weyl}
  Suppose that $(p\nu,w) \in \widetilde W_p$ fixes an element of some alcove. Then $(p\nu,w) = (0,1)$.
\end{lm}

\begin{proof}
  If $w\cdot \mu + p\nu = \mu$ for $\mu$ in some alcove then $p\nu = (1-w)(\mu+\rho') \in (1-w)X(T) \subseteq \Z R$.
  Since $\Z R \subseteq X(T)$ is saturated for $\GL_n$, $\nu \in \Z R$ and $\mu$ is fixed by an element of~$W_p$. The lemma
  follows since the closure of any alcove is a fundamental domain for~$W_p$.
\end{proof}

\begin{prop}\label{prop:generic_decomp}
  Suppose that $C$ is an alcove and that $\mu \in X(T)$ lies \sd\ inside~$C$. Then the Jordan--H\"older constituents of
  $\overline{R_w(\mu+\rho')}$ are the $F(\lambda)$ with $\lambda$ restricted such that there exist 
  $\sigma \in W$, $\nu \in X(T)$ with $\sigma\cdot (\mu + (w-p)\nu)$ dominant and
  \begin{equation}\label{eq:generic_decomp}
    \sigma\cdot (\mu + (w-p)\nu) \uparrow w_0\cdot (\lambda-p\rho').
  \end{equation}
\end{prop}

\begin{rk}
  Note that $\lambda \mapsto w_0 \cdot (\lambda - p\rho')$ induces a bijection on $A\subres$ \eqref{eq:A_res}.
\end{rk}

\begin{proof}
  Note that possible values of the left-hand side of~\eqref{eq:generic_decomp} are precisely the weight coordinates in the \gpxtw-orbit of
  $(w,\mu+\rho')$ shifted by $-\rho'$. Thus, without loss of generality, $C = C_0$. Let
  \begin{equation*}
    D_1 = \{ u \in \widetilde W_p : u \cdot \mu \in X_1(T) \}.
  \end{equation*}
  The generalisation of Jantzen's result \cite[4.3]{bib:Jan-DL} to
  $\GL_n$ is the following identity in the Grothendieck group of $\gln$-representations, valid for $\mu$ \sd\ in~$C_0$:
  \begin{equation}
    \label{eq:jantzen_generic}
      \overline{R_w(\mu+\rho')} = \sum_{\substack{u \in D_1\\\nu \in X(T)}} [\widehat Z_1(\mu -p\nu + p\rho'):\widehat L_1(u\cdot \mu)]
      F(u\cdot (\mu + w\nu)),
  \end{equation}
  where $\mu + w\nu \in C_0$ (and so $u\cdot (\mu + w\nu) \in X_1(T)$) whenever $[\widehat Z_1(\mu -p\nu + p\rho'):\widehat L_1(u\cdot \mu)] \ne 0$.
  (The proof generalises without difficulty. Use also lemma~\ref{lm:extended-weyl}.)
  Here $\widehat Z_1(\lambda)$ and $\widehat L_1(\lambda)$ for $\lambda \in X(T)$ denote $G_1T$-modules as in \cite[{\S{}II.9}]{bib:Jan-reps},
  the latter being simple.
  
  Choose $\sigma \in W$ such that $\sigma(\mu-p\nu+\rho')$ is dominant.  Then by \cite[II.9.16(4)]{bib:Jan-reps},
  \begin{equation}\label{eq:multiplicity}
    [\widehat Z_1(\mu - p\nu + p\rho'):\widehat L_1(u\cdot \mu)]
    = [\widehat Z_1(\sigma\cdot(\mu -p\nu) + p\rho'):\widehat L_1(u\cdot \mu)].
  \end{equation}
  If this integer is non-zero then by~\cite[II.9.16(6)]{bib:Jan-reps},
  \begin{align}
    \notag\sigma\cdot (\mu - p\nu) &\uparrow w_0 u\cdot \mu + p ({\textstyle\sum_{R^+}} \alpha - \rho') \\
    & \qquad = w_0 \cdot (u\cdot \mu - p\rho').\label{eq:condition_for_nonzero_mult}
  \end{align}
  As mentioned just after def.~\ref{df:uparrow}, we may replace both sides by weights in the same alcoves as long as they
  remain in the same $W_p$-orbit. We may thus replace $\mu \in C_0$ by $\mu + w\nu \in C_0$ in~\eqref{eq:multiplicity}, provided that the resulting weights
  are still in the same $W_p$-orbit. To see this is the case, note that $\widetilde w := w_0\cdot(u\cdot (\sigma^{-1}\cdot (-) + p\nu)-p\rho')$ is the element
  in $\widetilde W_p$ that takes the left-hand side of~\eqref{eq:multiplicity} to the right-hand side. By lemma~\ref{lm:extended-weyl}, $\widetilde w$ has to
  lie in fact in~$W_p$. Therefore
  \begin{equation}\label{eq:condition_for_nonzero_mult2}
    \sigma\cdot ((\mu + w\nu) - p\nu) \uparrow w_0 \cdot (u\cdot (\mu + w\nu) - p\rho').
  \end{equation}
  This shows that all constituents of $\overline{R_w(\mu+\rho')}$ are of the right form.
  
  Conversely suppose that~\eqref{eq:generic_decomp} holds. Note that the set of all $\mu-(w-p)\nu$ allowed
  by~\eqref{eq:generic_decomp} for some $\lambda \in X_1(T)$ and $\sigma \in W$ has to be contained in a finite union of
  alcoves. A simple argument like the one after def.~\ref{df:generic_tau} shows that there are only finitely many possibilities for $\nu$ modulo $X^0(T)$.  Thus
  if $\mu$ is \sd\ in~$C_0$, $\mu + w\nu \in C_0$ for any such~$\nu$.  Since moreover $\lambda \in X_1(T)$ is in the
  $\widetilde W_p$-orbit of $\mu+w\nu$, $\lambda = u\cdot (\mu+w\nu)$ for some $u \in D_1$
  and~\eqref{eq:condition_for_nonzero_mult2} holds. As in the previous argument we may replace $\mu + w\nu$ by~$\mu$
  in~\eqref{eq:condition_for_nonzero_mult2} to obtain~\eqref{eq:condition_for_nonzero_mult}. It remains to show
  that~\eqref{eq:multiplicity} is non-zero. A result of Ye (see \cite[II.9.16]{bib:Jan-reps}) shows that $[\widehat
  Z_1(\lambda') : \widehat L_1(\mu')]_{\SL_n} \ne 0$, where $\lambda' = \sigma\cdot(\mu -p\nu) + p\rho'$ and $\mu' = u\cdot
  \mu$. Finally note that
  \begin{equation*}
    [\widehat Z_1(\lambda') : \widehat L_1(\mu')] = [\widehat Z_1(\lambda') : \widehat L_1(\mu')]_{\SL_n}.
  \end{equation*}
  One observes first that $\widehat Z_1(\lambda)$ and $\widehat L_1(\lambda)$ restrict to the corresponding objects for
  $\SL_n$: this uses that $G_1T \cong U_1^- \times T \times U_1$ as schemes \cite[II.9.7]{bib:Jan-reps} and that these modules
  have a central character. The equality of multiplicities then follows since by \cite[II.9.15]{bib:Jan-DL} any constituent
  of $\widehat Z_1(\lambda')$ is of the form $\widehat L_1(\nu')$ for some $\nu' \in W_p\cdot \lambda'$ (even $\nu' \uparrow
  \lambda'$), the stabiliser of~$\lambda'$ in the affine Weyl group is trivial, and since the natural projection $X(T)\otimes
  \R \to X(T\cap \SL_n)\otimes \R$ maps alcoves for $\GL_n$ bijectively to the ones for $\SL_n$ and compatibly with respect to
  the action of the affine Weyl group.
\end{proof}

\section{A Serre-type conjecture}\label{sec:conj}

\emph{From now on we will assume that $n > 1$.} This could be avoided by working adelically (as in section~\S\ref{sec:theo_evid}), but for~$n = 1$ the adelic
version of the conjecture just comes down to class field theory for~$\Q$.

\subsection{Serre weights}

The representation-theoretic analogue of the weight in Serre's Conjecture is the following \cite{bib:ASinn}, \cite{bib:BDJ}.

\begin{df}
  A \emph{Serre weight} is an isomorphism class of irreducible representations of $\gln$ over~$\fpb$.
  By cor.~\ref{cor:gln_modules}, a Serre weight is of the form $F(a_1,a_2,\dots,a_n)$ with $0 \le a_i-a_{i+1} \le p-1$ for all~$i$. It is
  called \emph{regular} if $0 \le a_i-a_{i+1} < p-1$ for all~$i$.
\end{df}

Note that the number of Serre weights is $p^{n-1}(p-1)$, which equals the number of semisimple conjugacy classes in $\GL_n(\fp)$.

\subsection{Hecke algebras}\label{sub:hecke_alg}

Fix a positive integer~$N$ with $(N,p) = 1$.  Let $\Gamma_1(N)$ be the group of matrices in $\SL_n(\Z)$ with last row
congruent to $(0,\dots,0,1)$ modulo~$N$. Also let $S_1(N)$ be the group of matrices in $\GL_n^+(\zn)$ with last row congruent
to $(0,\dots,0,1)$ modulo~$N$ and let $S_1'(N) = S_1(N) \cap \GL_n^+(\znp)$.  Here $\zn$ is the ring of rational numbers with
denominators prime to~$N$.

Then $(\Gamma_1(N),S_1'(N))$, $(\Gamma_1(N),S_1(N))$ are Hecke pairs (see~\S\ref{sub:hecke}).  The corresponding Hecke algebras over the
integers are denoted by $\hop$, $\ho$; clearly $\hop \subseteq \ho$ is a subalgebra.

For any prime number $l\nmid N$ choose $\omega_N(l) \in \SL_n(\Z)$ with last row congruent to $(0,\dots,0,l^{-1}) \pmod N$; then
$\omega_N(l)\go = \go\omega_N(l)$ does not depend on any choices.
For primes $l\nmid N$ and $0\le i \le n$ define the Hecke operator
\[ T_{l,i} := [\go \bigg(\begin{smallmatrix}l\! \\[-5pt]
    & \ddots \\ && 1 \end{smallmatrix}\bigg) \widehat{\omega_N(l)}\go], \]
in $\ho$ ($i$ diagonal entries being equal to~$l$, $n-i$ equal~1). Here $\widehat{\omega_N(l)}$ stands for $\omega_N(l)$ if the diagonal
matrix has an~$l$ as its $(n,n)$-entry and for~1 otherwise. $T_{l,i}$ does not depend on the order of the diagonal entries.
This follows from the proof of the following lemma:

\begin{lm}
  \begin{align*}
    &\ho = \Z[T_{l,1},T_{l,2},\dots, T_{l,n},T_{l,n}^{-1} : l\nmid N] \\
    &\ \ \ \cup \\
    &\hop = \Z[T_{l,1},T_{l,2},\dots, T_{l,n},T_{l,n}^{-1} : l\nmid Np]
  \end{align*}
\end{lm}

\emph{Proof sketch:} Let $\Sigma_1(N) = M_n(\Z) \cap S_1(N)$ and $S_N = M_n(\Z) \cap \GL_n^+(\Z_{(N)})$.
One checks that $(\go,\Sigma_1(N)) \subseteq (\SL_n(\Z),S_N)$ are strongly compatible Hecke pairs (\S\ref{sub:hecke}),
such that $T_{l,i}$ corresponds to $[\SL_n(\Z)\bigg(\begin{smallmatrix}l\! \\[-5pt]
& \ddots \\ && 1
\end{smallmatrix}\bigg)\SL_n(\Z)]$ ($i$ entries equal~$l$). Finally one uses that $\HH(\SL_n(\Z),S_1)$ is a polynomial
ring in the $T_{l,i}$ for all primes~$l$ and all $1 \le i \le n$ \cite[\S3.2]{bib:Shimura}, and one makes use of the
grading on the Hecke algebras considered here induced by the determinant. \qed

Whenever $M$ is an $\fpb[S'_1(N)]$-module and for any~$e$, $\hop$ acts on the group cohomology module $H^e(\Gamma_1(N),M)$.
We will mostly consider the situation when $M = F$, a Serre weight, with $S_1'(N)$ acting via the reduction mod~$p$ map
$S'_1(N) \onto \GL_n(\fp)$.

\begin{df}[\cite{bib:ASinn}]
  Suppose that $\alpha \in H^e(\Gamma_1(N),M)$ is an $\hop$-eigenvector, say $T_{l,i} \alpha = a(l,i)\alpha$ for all
  $l\nmid pN$, $1 \le i \le N$. We say that a Galois representation $\rho : G_\Q \to \GL_n(\fpb)$ is \emph{attached to}~$\alpha$
  if for all $l\nmid Np$, $\rho$ is unramified at~$l$ and
  \begin{equation}\label{eq:rho_attached}
    \sum_{i=0}^n (-1)^i l^{i(i-1)/2} a(l,i)X^i = \det(1-\rho(\Frob_l^{-1})X),
  \end{equation}
  \(Remember that $\Frob_l \in G_\Q$ is a \emph{geometric} Frobenius element at~$l$.\)
\end{df}

\begin{rk}
  \need{double-check compatibility of notation!}
  
  A conjecture of Ash \(\cite{bib:A}, conjecture~B\) implies that for any Serre weight~$F$ and any $\hop$-eigenvector in
  $H^e(\go,F)$ \(any $(N,p) = 1$, $e \ge 0$ and $n > 1$\) there is an attached \(semisimple\) Galois representation.
  To see this implication, we will use the notation of~\S\ref{sec:comp_adps} and let $\widetilde\Sigma'_1(N) := \soptilde \cap M_n(\Z)$. An
  $\hop$-eigenvector gives rise to an $\hoptilde$-eigenvector \(prop.~\ref{lm:compare_with_ash}\) and $(\gotilde,\widetilde\Sigma'_1(N))$ is a
  ``congruence Hecke pair of level $Np$'' \(\cite{bib:A}, def.~1.2\) as $n > 1$.
\end{rk}

Analogous to Serre's Conjecture, we would like to understand, conversely, when a given $n$-dimensional Galois representation occurs
in such a group cohomology module and, if so, for which prime-to-$p$ levels~$N$ and Serre weights~$F$.
Fix thus a Galois representation $\rho : G_\Q \to \GL_n(\fpb)$ which we assume to be
\emph{irreducible} and \emph{odd}, in the following sense.

\begin{df}[\cite{bib:ASinn}]\label{df:odd}
  We will say that $\rho$ is \emph{odd} if either $p = 2$ or $|n_+ - n_-| \le 1$ where $n_+$ \(resp. $n_-$\)
  is the number of eigenvalues of $\rho(c)$ equal to 1 \(resp. $-1$\) where $c \in G_\Q$ is a complex conjugation.
\end{df}

Associated to~$\rho$ there is a prime-to-$p$ integer $N^?(\rho)$, its Artin conductor (see, for example,~\cite{bib:ADP}).
In Serre's Conjecture this is the smallest prime-to-$p$ level in which $\rho$ appears.

\begin{df}
  Let $W(\rho)$ \(resp., $W_{\!\opt}(\rho)$\) be the set of \emph{regular} Serre weights~$F$ such that $\rho$ is attached
  to an $\hop$-eigenvector in $H^e(\go,F)$ for some $e \ge 0$ and some integer~$N$ prime to~$p$ \(resp.,  $N = N^?(\rho)$\).
\end{df}

\begin{rk}
  \need{edit!}
  As discussed in~\cite{bib:ADP}, rk.~3.2, when $n = 3$, $e$ can be taken to be~3, the virtual
  cohomological dimension of~$\go$, in the definition.
\end{rk}

Let us now state a Serre-type conjecture for $n$-dimensional Galois representations~$\rho$ that are tame at~$p$.
It depends on two ingredients to be defined in the next two subsections: a representation $V(\rho|_{I_p})$
of $\gln$ over~$\qpb$ and an operator~$\RR$ on the set of Serre weights.

\begin{conj}\label{conj:serre}
  Suppose that $\rho : G_\Q \to \GL_n(\fpb)$ is irreducible, odd, and \emph{tamely ramified at~$p$}. Then
  \[ W(\rho) = W_{\!\opt}(\rho) = W^?(\rho|_{I_p}), \]
  where $W^?(\rho|_{I_p}) := \RR(\JH(\overline{V(\rho|_{I_p})}))$.
\end{conj}

By $\overline{V(\rho|_{I_p})}$ we mean, as in~\S\ref{sec:decomp_gln}, the reduction ``modulo~$p$'' of a \gln-stable
$\zpb$-lattice in $V(\rho|_{I_p})$ and by $\JH(-)$ the set of Jordan--H\"older constituents (forgetting multiplicities).

\subsection{The operator $\RR$ on Serre weights}\label{sub:operator_R}

Consider the bijection
\begin{align*}
  \{ \text{regular Serre weights} \}  &\to (\Z/(p-1))^n \\
  F(a_1,\dots,a_n) &\mapsto (\overline{a_1},\dots,\overline{a_n}).
\end{align*}
For any $b_i \in \Z$ define then $F(b_1,\dots,b_n)\subreg$ to be the regular Serre weight corresponding
in this bijection to $(\overline{b_1},\dots,\overline{b_n})$.

We can then define the operator~$\RR$ by
\begin{align*}
  \{ \text{Serre weights} \} &\to \{ \text{regular Serre weights} \} \\
  F(a_1,\dots,a_n) &\mapsto F(a_n-(n-1),\dots,a_2-1,a_1)\subreg.
\end{align*}
Thus on regular Serre weights, $\RR$ is an involution up to twist: $\RR^2(F) = F \tens \det^{1-n}$. A more conceptual
description is the following.

\begin{df}\label{df:rho}
  We let
  \begin{equation*}
    \rho := (n-1,n-2,\dots,1,0) \in \frac 12\sum_{\alpha\in R^+} \alpha + (X(T)\tens \Q)^W.
  \end{equation*}
  It thus also satisfies the condition imposed on~$\rho'$ in \S\ref{sec:glncharp}.
\end{df}

\begin{rk}
  Note that $\RR(F(\mu)) \cong F(w_0 \cdot (\mu - p\rho))\subreg$ for any $\mu \in X_1(T)$.
\end{rk}

\subsection{The characteristic zero representation $V(\rho|_{I_p})$}\label{sub:V(rho|_I)}

To make this as conceptual as possible, we will define it in the more general context of connected reductive groups defined
and split over~$\fq$ (with connected centre) and then make it explicit for $\GL_n$. We will use the notion of dual groups over
a finite field, as formulated by Deligne--Lusztig \cite{bib:DL}. The notation will be as in \S\ref{sec:glnchar0}.

At first $G$ need not be split and there is no assumption on the centre.
Our conventions for the actions of~$F$ and $w \in W$ on $\mu \in X(T)$ and $\lambda \in Y(T)$ are as follows:
$F(\mu) = \mu \circ F$, $F(\lambda) = F\circ \lambda$, $w(\mu) = \mu \circ w^{-1}$, $w(\lambda) = w \circ \lambda$.

\begin{df}
  Suppose $G\d$ is a connected reductive group defined over~$\fq$ with relative Frobenius morphism~$F\d$ and $F\d$-stable maximal torus~$T\d$.
  A \emph{duality} between $(G,T)$ and $(G\d,T\d)$ is an isomorphism $\phi : X(T) \to Y(T\d)$ such that $F^* \phi = \phi F$ and such
  that both $\phi$ and $\phi\dual : X(T\d) \to Y(T)$ send roots bijectively to coroots.
\end{df}

$G\d$ is called the \emph{dual group} of~$G$; it always exists and is unique up to isomorphism.

We get natural identifications of the Weyl groups so that $w\phi = \phi w$, but the Frobenius actions on~$W$ are mutual inverses: $F\d(w) = F^{-1}(w)$.
There is a correspondence between rational conjugacy classes of Frobenius-stable maximal tori $\T \subseteq G$ and $\T\d \subseteq G\d$ so that a type~$w$ torus in~$G$
corresponds to a type $F\d(w^{-1})$ torus in~$G\d$ (note that corresponding tori are in duality).
It extends to a correspondence between rational conjugacy classes of pairs $(\T,\theta)$ and pairs $(\T\d,s)$
where $\theta : \T^F \to \qpb\s$ and $s \in {\T\d}^{F\d}$. This depends on the choice of a generator $(\zeta_{p^i-1})_{i=1}^\infty \in \plim \F_{p^i}\s$:
without loss of generality, $\T = T_w$. Then $\overline{\theta({}^{g_w}-)} : T^{wF} \to \fpb\s$; extend it arbitrarily to a character $\mu \in X(T)$.
Let $\bar\mu = \phi(\mu) \in Y(T\d)$ and choose a positive integer~$t$ such that $T\d_{F\d(w^{-1})}$ (equivalently, $T_w$) is split
over $\F_{q^t}$. Then the dual pair is
\begin{equation*}
  (T\d_{F\d(w^{-1})}, {}^{g\d_{F\d(w^{-1})}} N_{(F\d w^{-1})^t/F\d w^{-1}}(\bar\mu(\zeta_{q^t-1})))
\end{equation*}
\cite[13.13]{bib:Digne-Michel}. Here we use the notation $N_{A^t/A} = \prod_{i=0}^{t-1} A^i$  for any $A \in \End (Y(T\d))$
and $F\d w^{-1} = F\d \circ w^{-1}$.

An $F$-stable maximal torus $\T \subseteq G$ is said to be \emph{maximally split} if $\T \subseteq B$ for some $F$-stable Borel subgroup~$B$.
Equivalently, $\fqrank(\T) = \fqrank(G)$ \cite[6.5.7]{bib:Carter}. All maximally split tori in~$G$ are $G^F$-conjugate \cite[1.18]{bib:Carter}.

\begin{df}[{\cite[5.25]{bib:DL}}]
  A pair $(\T,\theta)$ and its dual pair $(\T\d,s)$ \(as above\) are called \emph{maximally split} if $\T\d \subseteq Z_{G\d}(s)^\circ$ is maximally split.
\end{df}

Note that if $s \in G\d$ is semisimple, then $Z_{G\d}(s)^\circ$ is connected reductive, and if $Z(G)$ is connected then $Z_{G\d}(s)^\circ = Z_{G\d}(s)$
(see (2.3) and (13.15) in \cite{bib:Digne-Michel}).

Recall that the tame inertia group $I_p^t$ is isomorphic to $\plim \F_{p^i}\s$. This isomorphism
is canonical with our conventions as we \emph{defined} $\fpb$ to be the residue field of~$\qpb$ and $\F_{p^i} \subseteq \fpb$
as the unique subfield of cardinality~$p^i$.
Recall also the fundamental characters $\omega_{p^i} : I_p \to \F_{p^i}\s$ for each~$i$ obtained by projection from
the above isomorphism (again canonical here). In particular, $\omega := \omega_1$ is the mod~$p$ cyclotomic
character.

\begin{prop}\label{prop:V(rho|_I)}
  Assume that $Z(G)$ is connected, and that $T$ \(hence also~$T\d$\) is split over~$\fq$. Then we have the following commutative diagram:
  \ifkuvio
  \begin{equation*}
    \xscale=3.5
    \Diagram
    {\bigg\{\begin{array}[h]{@{}c@{}}
        \text{maximally split} \\
        \text{$(\T\d,s)$}
      \end{array} \bigg\} / {G\d}^{F\d}}
    & \rBij^{\text{duality}} & 
    {\bigg\{\begin{array}[h]{@{}c@{}}
        \text{maximally split} \\
        \text{$(\T,\theta)$}
      \end{array} \bigg\} / {G}^{F}} \\
    \dBij && \dInto > {\text{\eqref{x}}} \\
    {\bigg\{\begin{array}[h]{@{}c@{}}
        \text{tame $\tau : I_p \to G\d(\fpb)$} \\
        \text{that extend to $G_q$}
      \end{array} \bigg\} / \cong} & \rDashto^{V_\phi} & 
    {\bigg\{\begin{array}[h]{@{}c@{}}
        \text{virtual representations} \\
        \text{of $G^F$ over $\qpb$}
      \end{array} \bigg\} / \cong}\\
      \endDiagram
  \end{equation*}
  \else

  \begin{minipage}{1.0\linewidth}
    \v{0.6in}
    \centerline{big diagram here - switch on ``kuvio'' to see it!!}
    \v{0.6in}
  \end{minipage}
  \fi
  Here $G_q := \Gal(\qpb/\Q_q)$ with inertia subgroup~$I_p$.

  If $(T_w,\theta_{w,\mu})$ is maximally split for some $\wmu\in \wxt$, then under the above
  bijections it corresponds to the inertial Galois representation
  \begin{equation}\label{eq:tau(w,mu)}
    \tau(w,\mu) := N_{(F\d w^{-1})^t/F\d w^{-1}} (\bar\mu(\omega_{rt})).
  \end{equation}
  Here $\bar\mu = \phi(\mu) \in Y(T\d)$ and $t$ is any positive integer such that $T\d_{F\d w^{-1}}$ is split over $\F_{q^t}$
  \(equivalently, $w^t = 1$ as $T\d$ is split\). 

  In particular, $V_\phi(\tau\wmu) \cong R_w(\mu)$ and $V_\phi$ is independent of the choice of
  $(\zeta_{p^i-1})_i$.
\end{prop}

\begin{rk}
  It is known that $V_\phi(\tau)$ is a genuine representation in every case \cite[10.10]{bib:DL}.
\end{rk}

\begin{proof}
  The bijection on the left is obtained as follows. The choice of $(\zeta_{p^i-1})_i$ induces a generator $g_{\mathrm{can}}$ of the maximal
  tame quotient $I_p^t \congto \plim \F_{p^i}\s$ of~$I_p$. The isomorphism class of~$\tau$ is determined by the conjugacy class of
  $\tau(g_{\mathrm{can}})$, i.e., a conjugacy class in~$G\d$ that is stable under $x \mapsto x^q$ (as $\tau$ extends to~$G_q$) and whose
  members have prime-to-$p$ order. An element $g \in G\d$ has order prime to~$p$ iff it is semisimple (embed $G\d$ in some $\GL_m$). By
  conjugating $g$ to~$T\d$ and using that $T\d$ is split over~$\fq$ we see that its conjugacy class contains $g^q$ iff it contains $F\d(g)$.
  A simple argument shows that $F\d$-stable semisimple conjugacy classes in~$G\d$ are in natural bijection with ${G\d}^{F\d}$-conjugacy
  classes of semisimple elements in ${G\d}^{F\d}$ (see the proof of \cite[3.7.3]{bib:Carter}; this uses that $Z(G)$ is connected). Finally
  one shows that $(\T\d,s) \mapsto s$ induces a bijection from maximally split pairs to semisimple elements in ${G\d}^{F\d}$ (both up to
  ${G\d}^{F\d}$-conjugacy).  This only uses existence and uniqueness up to rational conjugacy of maximally split tori in $Z_G(s)^\circ$.

  The explicit description of~$\tau$ associated to $(T_w,\theta_{w,\mu})$ follows immediately from the description of the dual pair above.
\end{proof}

\textbf{From now on suppose again that $G = \GL_n$ and $T$ is the torus of diagonal matrices.} Let $(G\d,T\d) = (G,T)$ and let
\begin{align}\label{eq:duality_gln}
  \phi : X(T) &\congto Y(T\d) \\
  \notag (a_1,\dots,a_n) &\mapsto (a_1,\dots,a_n)
\end{align}
(the notation should be self-evident). This is clearly a duality in the sense defined above.
Since a connected reductive group defined over~$\fq$ is determined
by its root datum together with the $F$-action on it \cite[3.17]{bib:Digne-Michel}, $(G\d,T\d)$ is well defined up to isomorphism and any
other duality between $(G,T)$ and $(G\d,T\d)$ differs by an automorphism
of $(X(T),R,X(T)\dual,R\dual)$ commuting with $F$ (the latter condition is automatic as $T$ is split here). It is known and easy
to verify that any such automorphism is, up to the Weyl group action, which leaves $V_\phi$ unchanged, either trivial or given by
$(a_1,\dots,a_n)\mapsto (-a_1,\dots,-a_n)$ on $X(T)$.

Thus there are two ways to define $V(\tau)$ which differ by $\tau \mapsto \tau\dual$. These two choices corresponds to the two choices of normalising
the Galois representation associated to a Hecke eigenvector (in~\eqref{eq:rho_attached}, geometric Frobenius elements could be replaced
by arithmetic ones). The above choice of~$\phi$ is the one that will work here.

\begin{df}\label{df:V(tau)}
  For a tame inertial Galois representation $\tau : I_p \to \GL_n(\fpb)$ that extends to~$G_q$, we set $V(\tau) := V_\phi(\tau)$
  with $\phi$ as in~\eqref{eq:duality_gln}.
\end{df}

We will finally describe explicitly the maximally split pairs $(\T,\theta)$, which enables us to characterise the image of~$V$.

\begin{df}
  Suppose that $\wmu\in \wxt$ with $\mu = (\mu_1,\dots,\mu_n)$. For each $1 \le i\le n$, let $n_i$ denote the smallest positive integer
  with $w^{n_i}(i) = i$. We say that $\wmu$ is \emph{good} if for all~$i$,
  \begin{equation*}
    \sum_{k\mmod n_i} \mu_{w^k(i)} q^{k} \not\equiv 0\pmod{\textstyle\frac{q^{n_i}-1}{q^d-1}}
  \end{equation*}
  for all $d | n_i$, $d \ne n_i$.
\end{df}

\begin{prop}
  Suppose that $\wmu\in \wxt$. The pair $(T_w,\theta_{w,\mu})$ is maximally split if and only if $\wmu$ is good.
\end{prop}

\begin{proof}
  As described above, the dual pair is $(T\d_{F\d(w^{-1})}, s_{w,\mu})$ where $s_{w,\mu} = {}^{g\d_{F\d(w^{-1})}} N_{(F\d w^{-1})^t/F\d
    w^{-1}}(\bar\mu(\zeta_{q^t-1}))$ ($t$ and $\bar\mu$ as before).  Note that if $\T\d \subseteq G\d$ is an $F\d$-stable maximal torus of
  type $\sigma \in W \cong S_n$, then the $\fqrank$ of~$\T\d$ is the number of orbits of~$\sigma$ on $\{1,\dots,n\}$. (Recall that
  $T$ and $T\d$ are split.) 

  \begin{sublm}
    Suppose $s \in {G\d}^{F\d}$ semisimple. Then $s$ lies in some $F\d$-stable  maximal torus of
  type~$\sigma$ iff $F\d(s') = \sigma^{-1}(s')$ for some $G\d$-conjugate $s' \in T\d$ of~$s$.
  \end{sublm}

  \begin{proof}
    If $s \in \T\d$ of type~$\sigma$ then there is a $g\in G\d$ such that $\T\d = {}^g T\d$ and $g^{-1}F\d(g)$ is a lift of~$\sigma$ in $N(T\d)$.
    Note that $s' := {}^{g^{-1}} s$ works.

    For the other direction we can reverse the argument just given to see that there is a $G\d$-conjugate $s_0 \in {\T\d}^{F\d}$ of~$s$ for some
    $F\d$-stable maximal torus~$\T\d$ of type~$\sigma$. Writing $s = {}^h s_0$ for some $h\in G\d$, it follows that $h^{-1}F\d(h) \in Z_{G\d}(s_0)$
    which is connected reductive. By Lang's theorem $h^{-1}F\d(h) = z^{-1}F\d(z)$ for some $z \in Z_{G\d}(s_0)$. Then $s \in {}^{hz^{-1}} \T\d$ which
    is of type~$\sigma$ as $hz^{-1} \in {G\d}^{F\d}$.
  \end{proof}

  It follows that $(T_w,\theta_{w,\mu})$ is maximally split iff whenever $F\d(s'_{w,\mu}) = \sigma^{-1}(s'_{w,\mu})$ for a $G\d$-conjugate
  $s'_{w,\mu} \in T\d$ of $s_{w,\mu}$ then $\sigma$ has at most as many orbits on $\{1,\dots,n\}$  as~$w$. As $G\d$-conjugate elements
  in~$T\d$ are $W$-conjugate \cite[0.12(iv)]{bib:Digne-Michel}, we need only consider $s'_{w,\mu} =  N_{(F\d w^{-1})^t/F\d
    w^{-1}}(\bar\mu(\zeta_{q^t-1}))$, which equals $\bigg(\begin{smallmatrix} x_1 \\[-5pt] & \ddots \! \\
    && x_n \end{smallmatrix}\bigg)$ for some $F$-stable sub-multiset $\{x_i\}_{i=1}^n$ of~$\fpb\s$.
  
  If $F\d(s'_{w,\mu}) = \sigma^{-1}(s'_{w,\mu})$ then for all~$i$ and all~$k$, $F^k(x_i) = x_i$ whenever $\sigma^k(i) = i$. It follows that
  such a~$\sigma$ has the maximal number of orbits iff
  \[ \forall i\ \forall k,\ F^k(x_i) = x_i \iff \sigma^k(i) = i. \]
  Thus $(T_w,\theta_{w,\mu})$ is maximally split iff for all~$i$, $\zeta_{q^{n_i}-1}^{\sum \bar\mu_{w^k(i)} q^k}$ is in no proper subfield
  of $\F_{q^{n_i}}$, i.e., iff \wmu\ is good.
\end{proof}

Using lemmas \ref{lm:cusp}, \ref{lm:para_ind} this implies:

\begin{coroll}
  The image of the map $V$ consists precisely of parabolic inductions of cuspidal representations.
\end{coroll}

For example, if $n = 3$, $w = (1\;2\;3)$ and $\mu = (i,j,k)$ then \wmu\ is good iff $m := i + qj + q^2 k \not\equiv 0 \pmod {q^2+q+1}$.
In this case 
\begin{equation*}
  \tau\wmu \sim \bigg(\begin{smallmatrix} \omega_3^m \\ & \omega_3^{qm}\! \\[-4pt]
    && \omega_3^{q^2m} \end{smallmatrix}\bigg)
\end{equation*}
and $R_w(\mu)$ is a cuspidal representation of $\GL_3(\fq)$ \eqref{lm:cusp}.

The following basic proposition will be used later. The corresponding result for $W(\rho)$ follows in the same way
as in~\cite{bib:ASinn}, lemma~2.5 and prop.~2.8.

\begin{prop}\label{prop:conj_basicprops}
  Suppose that the tame inertial Galois representation $\tau : I_p \to \GL_n(\fpb)$ extends to~$G_p$. Then
  \begin{enumerate}
  \item $W^?(\tau \tens \omega) = W^?(\tau) \tens \Det.$
  \item $W^?(\tau\dual) = \{F\dual \tens \Det^{1-n} : F \in W^?(\tau) \}.$
  \end{enumerate}
\end{prop}

\begin{proof}
  For (i) this follows from the facts that $R_T^{\theta\cdot \widetilde{\det}} \cong R_T^\theta \tens \widetilde\det$
  \cite[cor.\ 1.27]{bib:DL} and that $\RR(F\tens \det) \cong \RR(F) \tens \det$.

  For (ii) this follows from the facts that $R_T^{\theta^{-1}} \cong (R_T^\theta)\dual$
  \cite[p.\ 136]{bib:DL} and that $\RR(F\dual) \cong \RR(F) \tens \det^{1-n}$.
\end{proof}

\subsection{The generic case}

\begin{lm}\label{lm:generically_good}
  Suppose that $\mu \in X(T)$ lies \sd\ in~$C_0$. Then $(w,\mu)$ is good.
\end{lm}

\begin{proof}
  If $\mu = \sum a_i\epsilon_i$ is \sd\ in~$C_0$, note that $\sum_{j = 1}^r a_{i_j} \epsilon_j$ is as deep as we like in the lowest alcove for
  $\GL_r$ (whenever $1 \le i_1 < \dots < i_r \le n$). We are thus reduced
to the case when $w$ is an $n$-cycle. 
We need to show that if $\mu$ is \sd\ in~$C_0$,
\begin{equation}\label{eq:lm_suff_deep}
  \sum_{i \mmod n} a_{w^i(1)} p^i \not\equiv 0 \pmod {\textstyle\frac{p^n-1}{p^d-1}}
\end{equation}
for all $d | n$, $d \ne n$. Fix $n = de$ with $d < n$. Using
\begin{equation*}
  \frac{p^n-1}{p^d-1} = \sum_{j = 0}^{e-1} p^{dj},
\end{equation*}
equation~\eqref{eq:lm_suff_deep} becomes
\begin{equation}\label{eq:lm_suff_deep2}
  \sum_{i=0}^{d(e-1)-1} (c_{i}-c_{d(e-1)+d\{\frac {i}d\}}) p^{i} \not\equiv 0 \pmod {\textstyle\sum_{j = 0}^{e-1} p^{dj}},
\end{equation}
where $c_i = a_{w^i(1)}$ and $\{x\} \in [0,1)$ denotes the fractional part of a real number~$x$. As $\mu$ is in the lowest alcove,
$|c_i - c_j| \le p-1$ for all $i$, $j$. So if $\mu$ lies \sd\ in~$C_0$ then $c_i \ne c_j$ for all $i \ne j$ and~\eqref{eq:lm_suff_deep2} is
automatic as $(p-1)(1+p+\dots+p^i) < p^{i+1}$ for all~$i$.
\end{proof}

\begin{df}\label{df:generic_tau}
  Suppose that $\tau : I_p \to \GL_n(\fpb)$ is tame and that it can be extended to~$G_p$. Then $\tau$ is said to be \emph{$\delta$-generic}
  if $\tau \cong \tau(w,\mu)$ for some good $(w,\mu) \in \wxt$ such that $\mu$ is $\delta$-deep in~$C_0$.
\end{df}

A statement in which~$p$ is allowed to vary is said to be true for
\emph{sufficiently generic}~$\tau$ if there is a $\delta > 0$, independent of~$p$, such
that the statement holds for all $\delta$-generic $\tau$.

Recall that by lemma~\ref{lm:dl_reps} and prop.~\ref{prop:V(rho|_I)}, $(w,\mu)$ in the definition is well defined up to the
$\gpxtw$-action~\eqref{eq:jantzen-action} which can be expressed as
\begin{equation*}
   {}^{(\nu,\sigma)}\wmu = (\sigma w \sigma^{-1},(\sigma\cdot\mu + p\nu) + \nu_\epsilon ),
\end{equation*}
where $\nu_\epsilon = \rho' -\sigma \rho' - \sigma w \sigma^{-1} \nu$. Fix for now \wmu\ with $\mu \in C_0$. Consider the set
$\{\sigma\cdot\mu + p\nu : (\nu,\sigma) \in \gpxtw \} \subseteq X(T)$. Modulo $pX^0(T)$, it contains precisely $\# W_1 = n$ weights in each alcove.
To see this, note that $\sigma\cdot\mu + p\nu \in C_0$ iff $\sigma \in W_1$ and $\nu \in \rho'_\sigma + X^0(T)$ (\S\ref{sub:jantzens-theorem}), that 
$W_p$ acts transitively on the set of alcoves, and that no non-trivial element of $\widetilde W_p$ fixes any weight in~$C_0$ (lemma~\ref{lm:extended-weyl}).

Fix any alcove~$C$ and let us always assume for now that $\mu$ is \sd\ in~$C_0$ (the implied constant might depend on the statement).
Consider the set of weight coordinates of the $\gpxtw$-orbit of \wmu,
\begin{equation*}
  \{ \sigma\mu +(p-\sigma w\sigma^{-1})\nu : (\nu,\sigma) \in \gpxtw \}.
\end{equation*}
We claim that modulo $(p-1)X^0(T)$, this set contains precisely $\# W_1 = n$ weights in~$C$.

First of all let us show that $\mu' := \sigma\mu +(p-\sigma w\sigma^{-1})\nu \in C$ if and only if $\sigma\cdot\mu + p\nu \in
C$. Suppose first that $\mu' \in C$.  There exist $n_\alpha \in \Z$ for $\alpha \in R$ such that $\eta \in C$ implies that $|\langle
\eta+\rho',\alpha\dual\rangle| < n_\alpha p$ for all~$\alpha$. Thus we may even assume that $|\langle
\eta,\alpha\dual\rangle| \le (n_\alpha+1)(p-1)$ for all~$\alpha$ (by the assumption on~$\mu$, we can put a lower bound
on~$p$). Also $|\langle \mu,\alpha\dual\rangle| \le p-1$ for all~$\alpha$. Summing
\begin{equation*}
  |\langle \mu',\alpha\dual\rangle| \ge p |\langle \nu,\alpha\dual\rangle| -
    |\langle \nu,\sigma w^{-1}\sigma^{-1}\alpha\dual\rangle| - |\langle \mu,\sigma^{-1}\alpha\dual\rangle|
\end{equation*}
over all $\alpha \in R$, we find that $\sum |\langle \nu,\alpha\dual\rangle| \le \sum (n_\alpha + 2)$, so that $\nu$ can only take a finite number
of values modulo $(p-1)X^0(T)$. If $\mu$ is \sd\ in~$C_0$, for all those values of~$\nu$, $\sigma\mu +(p-\sigma w\sigma^{-1})\nu$
lies in the same alcove as $\sigma\cdot\mu + p\nu$. This shows the ``only if'' implication; the converse is similar but much easier.

It remains to show that
$\sigma_1\mu +(p-\sigma_1 w\sigma_1^{-1})\nu_1 = \sigma_2\mu +(p-\sigma_2 w\sigma_2^{-1})\nu_2 \in C$ if and only if
$\sigma_1 \cdot \mu + p\nu_1 = \sigma_2 \cdot \mu + p\nu_2 \in C$. For the ``only if'' implication, note
that the first statement implies $\sigma_1^{-1}\nu_1 - \sigma_2^{-1}\nu_2 \in {\Z R}$ by an argument as in the proof of lemma~\ref{lm:extended-weyl}.
Thus $\sigma_i\cdot \mu + p\nu_i \in C$ are in the same $W_p$-orbit, so they are equal.
For the ``if'' implication, note that the second statement implies $\sigma_1 = \sigma_2$, $\nu_1 = \nu_2$.

\begin{prop}\label{prop:generic_pred}
  Suppose that $\tau : I_p \to \GL_n(\fpb)$ is tame, can be extended to~$G_p$ and that $\lambda \in X_1(T)$. Suppose
  either that \(a\) $\tau$ is sufficiently generic or \(b\) $\lambda$ is \sdr. Then
  \begin{equation*}
    F(\lambda) \in W^?(\tau)
  \end{equation*}
  if and only if
  \begin{gather*}
    \tau \cong \tau(w',\lambda'+\rho)\text{\ for some $\lambda' \in X(T)_+$ such that $\lambda' \uparrow \lambda$} \\
    \text{and some $w' \in W$.}
  \end{gather*}
\end{prop}

\begin{proof}
  We will first show the result under assumption (a), then we will show how (b) reduces to (a).

  Write $\tau \cong \tau(w,\mu+\rho)$. If $\mu$ lies \sd\ in~$C_0$, we may assume by lemma~\ref{lm:generically_good} that $(w,\mu+\rho)$ is good and
  by prop.~\ref{prop:generic_decomp} that
  $W^?(\tau)$ consists of the $F(\lambda)$ with $\lambda \in X_1(T)$ such that there exists a dominant $\lambda' \uparrow \lambda$
  satisfying
  \begin{equation}\label{eq:1}
    \exists (\sigma,\nu) \in \wxt,\ \lambda' = \sigma\cdot(\mu + (w-p)\nu).
  \end{equation}
  From~\eqref{eq:jantzen-action} it follows that~\eqref{eq:1}  is equivalent to
  \begin{equation*}
    \exists w' \in W,\ (w,\mu + \rho) \sim (w',\lambda'+\rho).
  \end{equation*}
  Finally, as this $\gpxtw$-orbit is good by the choice of~$\mu$, this is equivalent to
  \begin{equation*}
    \exists w' \in W,\ \tau(w,\mu+\rho) \cong \tau(w',\lambda'+\rho).
  \end{equation*}
  
  To reduce (b) to (a), suppose the proposition holds if $\tau$ is $\epsilon$-generic.  Suppose that $\tau$ is \emph{not}
  $\epsilon$-generic. Write $\tau \cong \tau(w,\mu+\rho)$ for some good $(w,\mu+\rho)$. Using~\eqref{eq:jantzen-action} we first claim that $\mu$
  can be chosen to be $\delta$-close to, i.e., $(-\delta)$-deep in, $C_0$ (for some $\delta > 0$ independent of $p$). By
  using a fundamental parallelepiped for the lattice $(p-w)X(T) \subseteq X(T)$, we see from~\eqref{eq:jantzen-action} with $\sigma = 1$ that $\mu$
  may be chosen to lie in a finite union of alcoves, say $\bigcup_{i=1}^N (\sigma_i \cdot C_0 + p\nu_i)$. The claim now
  follows from~\eqref{eq:jantzen-action} using $\sigma = \sigma_i^{-1}$. As $\tau$ is not $\epsilon$-generic, we may increase~$\delta$ if necessary
  to assure that $\mu$ is $\delta$-close to the boundary of~$C_0$ (that is, ($-\delta$)-deep but not $\delta$-deep). The
  analysis of Jantzen's decomposition formula~\eqref{thm:jantzen_formula} in~\S\ref{sub:jantzens-theorem} shows then that the
  highest weights of each constituent of $\overline{R_w(\mu+\rho)}$---and thus the highest weights of each Serre weight in
  $W^?(\tau)$---is $\delta'$-close to the boundary of some restricted alcove for some $\delta' > 0$ depending on~$\delta$.
  Therefore if $\lambda$ is $\delta'$-deep in a restricted alcove, $F(\lambda) \in W^?(\tau)$ implies that $\tau$ is
  $\epsilon$-generic.  By restricting~$\lambda$ yet further in its alcove, we can moreover achieve that for any $\lambda'
  \uparrow \lambda$ with $\lambda'$ dominant and any $w' \in W$, $\tau(w',\lambda'+\rho)$ is $\epsilon$-generic.
\end{proof}

\begin{coroll}\label{cor:no_wts_predicted}
  Suppose that $\tau$ is sufficiently generic. Then
  \begin{equation*}
    \# W^?(\tau) = \# W_1 \cdot \# \{ (C',C) : \text{$C'$ dominant, $C$ restricted, and $C'\uparrow C$} \},
  \end{equation*}
  where $C'$ and~$C$ denote alcoves. Note that $\# W_1 = n$.
\end{coroll}

In particular, the number of weights predicted in the generic case is 2, 9, 88, 1640, \dots\ if $n = 2$, 3, 4, 5, \dots

\begin{proof}
  Write $\tau \cong \tau(w,\mu+\rho)$ with $\mu$ sufficiently deep in~$C_0$. Note first that if $\varepsilon \in (p-1)X^0(T)$
  then $(w',\lambda'+\rho) \sim (w',\lambda' + \varepsilon + \rho)$, $F(\lambda) \cong F(\lambda+\varepsilon)$ and
  $\lambda' \uparrow \lambda$ implies $\lambda' + \varepsilon \uparrow \lambda + \varepsilon$. Thus we only need to consider the $\lambda'$ in
  prop.~\ref{prop:generic_pred} up to $(p-1)X^0(T)$.
  
  It suffices by the argument after def.~\ref{df:generic_tau} to prove the following statement.
  \begin{equation}\label{eq:5}
    \begin{minipage}[h]{0.85\linewidth}
      \emph{Suppose that $\tau \cong \tau(w_i,\lambda_i'+\rho)$ with $\lambda_i'$ dominant and $\lambda_i' \uparrow \lambda$ for some
        restricted weight~$\lambda$ \($i = 1$, $2$\).  Then $w_1 = w_2$ and $\lambda_1' = \lambda_2'$.}
    \end{minipage}
  \end{equation}
  
  We can write $w_i = \sigma_i w\sigma_i^{-1}$, $\lambda_i' = \sigma_i\cdot \mu + (p-\sigma_i w\sigma_i^{-1}) \nu_i$ for some
  $(\nu_i,\sigma_i) \in \gpxtw$.  As in the proof of lemma~\ref{lm:extended-weyl}, it follows that $\lambda \equiv \lambda_i'
  \equiv \mu+(p-1)\sigma_i^{-1}\nu_i \pmod {\Z R}$ so that $\sigma_1^{-1}\nu_1-\sigma_2^{-1}\nu_2 \in \Z R$.  As the
  $\lambda_i'$ are in the same $W_p$-orbit by assumption, so are the $\mu+(p-w)\sigma_i^{-1} \nu_i$. As $\sigma_1^{-1}\nu_1 -
  \sigma_2^{-1}\nu_2 \in \Z R$, the weights $\mu + p\sigma_2^{-1}\nu_2 - w \sigma_i^{-1} \nu_i$ are in
  the same $W_p$-orbit. But as they lie in the same alcove $C_0 + p\sigma_2^{-1}\nu_2$, they have to be equal and we obtain
  that $\sigma_1^{-1}\nu_1 - \sigma_2^{-1}\nu_2 \in X^0(T) \cap \Z R = \{0\}$. So the  dominant alcoves $\sigma_i \cdot C_0 + p\nu_i$  are
  related by $\sigma_1\sigma_2^{-1} \in W$. Thus $\sigma_1 = \sigma_2$, which implies the claim.
\end{proof}

\section{Comparison with the ADPS conjecture ($n=3$)}\label{sec:comp_adps}

The framework of the conjecture used here differs slightly from that of~\cite{bib:ADP}---we prefer to use left cosets, left actions and to
ignore the nebentype character. First we explain how to relate them. When comparing the weight predictions, note that the conjecture in~\cite{bib:ADP}
is stated for general~$n$ and for odd Galois representations~$\rho$ that are neither necessarily tame at~$p$ (at least in niveau~1)
nor irreducible.  For irreducible~$\rho$, their predictions only depend on $\rho|_{I_p}$. We will restrict to the irreducible, tame-at-$p$ case
to compare with our conjecture, and we will assume that $n=3$, the case they studied in detail.  (For larger~$n$ in generic cases their weights will
all be predicted here, but the discrepancy grows with~$n$, even if Doud's extension in the niveau~$n$ case~\cite{bib:Doud_ss} is taken into
account.) We will moreover interpret their recipe in the most favourable way, that is, include the ``extra weights'' described
in~\cite{bib:ADP}, def.~3.5. We should point out though that~\cite{bib:ADP} never claims to predict all possible weights.

Let $\gotilde$ be the group of matrices in $\SL_n(\Z)$ with first row congruent modulo~$N$ to $(1,0,\dots,0)$, and
let $\soptilde \subseteq \GL_n^+(\znp)$ be defined by the same congruence condition. Then $\hoptilde$ is
the Hecke algebra defined by the Hecke pair $(\gotilde,\soptilde)$, but instead of left cosets (as in
\S\ref{sub:hecke}) using \emph{right cosets}. If the congruence condition is weakened to the first row being $(*,0,\dots,0)$ modulo~$N$, the
corresponding objects are denoted by $\gztilde$, $\szptilde$, $\hzptilde$. Note that $(\gotilde,\soptilde)$
and $(\gztilde,\szptilde)$ are strongly compatible (\S\ref{sub:hecke}).

Letting
\begin{equation*}
  \eta =  \left(\begin{smallmatrix}    &&&N \\ &&1 \\[-5pt] &\reflectbox{$\ddots$} \\ 1  \end{smallmatrix}\right),
\end{equation*}
observe that $g \mapsto \eta \cdot {}^t g \cdot \eta^{-1}$ induces \textit{anti-isomorphisms} of groups $\gi \to \gitilde$, $\sip \to \siptilde$,
and of (commutative) algebras $\hip \to \hiptilde$ ($i = 0,1$).

A Serre weight~$F$ (with usual left $\sip$-action) becomes a right $\siptilde$-module, denoted~$\widetilde F$, as follows: $m\widetilde s :=
{}^t(\eta^{-1} \widetilde s \eta)m$ ($m \in F$, $\widetilde s \in \siptilde$, $i = 0,1$). It is easy to see that with this action, $\widetilde
F$ is a ``right Serre weight'' with the same highest weight. The following lemma is immediate.

\begin{lm}
  \label{lm:compare_with_ash}
  The above anti-isomorphisms induce an isomorphism
  \[ H^e(\go,F) \cong H^e(\gotilde,\widetilde F), \]
  as modules for $\hop \cong \hoptilde$.
\end{lm}

Any character $\epsilon : (\Z/N)\s \to \fpb\s$ can be considered as character of $\szptilde$ via its natural projection to
$\szptilde/\soptilde \cong (\Z/N)\s$. Let $\widetilde F(\epsilon) = \widetilde F \tens \fpb(\epsilon)$.

\begin{lm}
  Fix a ring homomorphism
  \[  \sigma : \hoptilde \cong \hzptilde \to \fpb.  \]
  The following are equivalent:

  \begin{enumerate}
  \item There is an $\hoptilde$-eigenvector for~$\sigma$ in $H^e(\gotilde,\widetilde F)$ for some~$e$.
  \item There is an $\hzptilde$-eigenvector for~$\sigma$ in $H^e(\gztilde,\widetilde F(\epsilon))$ for some~$e$ and for some
    $\epsilon: (\Z/N)\s \to \fpb\s$.
  \end{enumerate}
\end{lm}

\begin{proof}
  Note that the proof is complicated by the fact that $p$ could divide $\phi(N)$.
  If $M$ is any $\szptilde$-module then
  $(\Z/N)\s$ acts naturally (and $\delta$-functorially) on $H^e(\gotilde, M)$, commuting with the action of $\hoptilde$
  (as observed in~\cite{bib:AStev}, p.~196). The Hochschild--Serre spectral sequence
  \[ E_2^{p,q} : H^p((\Z/N)\s ,H^q(\gotilde,\widetilde F(\epsilon))) \Rightarrow H^{p+q}(\gztilde,\widetilde F(\epsilon)) \]
  is compatible with the action of $\hoptilde \cong \hzptilde$. The reason is that the Grothendieck spectral sequence for a composition $F_1
  \circ F_2$ is compatible with natural transformations $F_2 \to F_2$ since the spectral sequences for the hyperderived functors
  $(\mathbb{R}^i F_1)(C)$ are functorial in the cochain complex~$C$ \cite[\S2.4]{bib:Tohoku}.

  If $M$ is any $\hoptilde$-module, denote by $M_\sigma$ the generalised $\sigma$-eigenspace.
  
  Supposing (ii), considering the generalised $\sigma$-eigenspace of the above spectral sequence we find that $(E_2^{p,q})_\sigma \ne 0$ for
  some $p$, $q$, whence (i). (All terms of the spectral sequence are finite-dimensional, as explained just before~\eqref{eq:6}.)

  Conversely, assuming (i), pick~$q$ smallest such that $H^q(\gotilde,\widetilde F)_\sigma \ne 0$. Observing that
  \[ H^q(\gotilde,\widetilde F(\epsilon)) \cong H^q(\gotilde,\widetilde F)(\epsilon) \]
  as $(\Z/N)\s$-module, we can choose $\epsilon$ so that $H^q(\gotilde,\widetilde F(\epsilon))_\sigma$ has a $(\Z/N)\s$-fixed vector. By
  the minimality of~$q$, $(E_\infty^{0,q})_\sigma \ne 0$, whence (ii).
\end{proof}

\need{be more precise? state [ADP] conj.?}

For the remainder of this section, suppose that $n = 3$.
For simplicity we will say that a Serre weight $F(\lambda)$ ($\lambda \in X_1(T)$) is in the lower alcove $C_0$ if $\lambda \in C_0$.  If
$F(\lambda)$ is a regular Serre weight, we will use the notation $\r F(\lambda) := F(\r \lambda)$ with $(\lambda \mapsto \r\lambda) \in W_p$
as in prop.~\ref{prop:weyl_simple}. (Note that both definitions do not actually depend on any choices.)

\begin{df}
  For $\lambda \in X_1(T)$ let $\AA(\lambda)$ be the set of regular Serre weights consisting of $F := F(\lambda-\rho)\subreg$
and, in case $F \in C_0$, also $\r F \in C_1$.
\end{df}

The next result should be compared with prop.~\ref{prop:generic_pred}.

\begin{prop}\label{prop:combi_conj_gl3}
  Suppose that the tame inertial Galois representation  $\tau : I_p \to \GL_n(\fpb)$  can be extended to~$G_p$.
  Let
  \begin{equation*}
    \CC(\tau) = \{ \lambda \in X_1(T) : \exists w \in W,\ \text{$(w,\lambda)$ good and $\tau \cong \tau(w,\lambda)$}  \}.
  \end{equation*}
  Then 
  \begin{equation}\label{eq:wts_combinatorial}
    W^?(\tau) = \bigcup_{\lambda \in \CC(\tau)} \AA(\lambda).
  \end{equation}
\end{prop}

It will become clear from the proof that for sufficiently generic~$\tau$, $\CC(\tau)$ consists of three weights each in the upper and the
lower alcove. We will use the following lemma.

\begin{lm}\label{lm:combi_conj_gl3}
  \textup{(i)} If $\tau \sim \nivonematrix$ with $i \ge j \ge k$, $i-k \le p-1$,
  \begin{gather*}
    \CC(\tau) = \big\{ (i,j,k), (j,k,i-p+1), (k+p-1,i,j), \\
    (k+p-1,j,i-p+1), (i,k,j-p+1), (j+p-1,i,k)\big\} 
    + (p-1)X^0(T).
  \end{gather*}

  \textup{(ii)} If $\tau \sim \nivtwomatrix$ with $m = j+pk$ and $i \ge j > k$, $i-k \le p-1$,
  \begin{gather*}
    \CC(\tau) = X_1(T) \cap \big\{(i,j,k), (j,k,i-p+1), (k+p,i,j-1), \\
    (k+p,j-1,i-p+1), (i,k+1,j-p), (j+p,i,k-1),  \\
    (i+p-1,j,k), (j,k,i-2p+2)\big\} 
    + (p-1)X^0(T).
  \end{gather*}

  \textup{(iii)}  If $\tau \sim \nivthreematrix$ with $m = i+pj+p^2 k$ and $i > j \ge k$, $i-k \le p$,
  \begin{gather*}
    \CC(\tau) =  X_1(T) \cap \big\{(i,j,k), (j+1,k,i-p), (k+p,i-1,j), \\
     (k+p,j+1,i-p-1), (i,k+1,j-p),  (j+p,i,k-1)\big\} 
    + (p-1)X^0(T).
  \end{gather*}
\end{lm}

\begin{proof} Suppose that $\lambda = (x',y',z') \in \CC(\tau)$.
  
(i) By prop.~\ref{prop:V(rho|_I)} and lemma~\ref{lm:dl_reps}, $\tau \cong \tau(w,\lambda)$ with $(w,\lambda)$ good implies that $(w,\lambda) \sim
(1,(i,j,k))$. Thus there is a permutation $(x,y,z)$ of $(x',y',z')$ such that $x \equiv i,\ y \equiv j,\ z\equiv k \pmod {p-1}$.
This is invariant under the change of coordinates
\begin{equation*}
  \theta : (x,y,z;i,j,k) \mapsto (z,x,y;k+p-1,i,j).
\end{equation*}

We may assume without loss of generality that $y = j$ and (using~$\theta$) that either $x \ge y \ge z$ or $x < y < z$.
In the first case, $(x',y',z') = (x,y,z)$. It is then evident that precisely the following weights are obtained:
$(i,j,k)$, $(i+p-1,j,k) = (j+p-1,i,k)$ (if $i=j$), $(i,j,k-p+1) = (i,k,j-p+1)$ (if $j = k$),
$(i+p-1,j,k-p+1) = (k+p-1,j,i-p+1)$ (if $i = j = k$). The second case is analogous, yielding
precisely $(k+p-1,j,i-p+1)$ (due to the inequalities being strict).

(ii) Here there is a permutation $(x,y,z)$ of $(x',y',z')$ such that $x \equiv i \pmod {p-1},\ y+pz \equiv m \pmod {p^2-1}$.
Without loss of generality, $y+pz = j+pk$. Note that $|y-z| \le 2p-2$. Thus $(y,z) = (j,k) +
n(p,-1)$ with $-2 \le n \le 1$.

If $n = -2$: since $j-2p < i-2p+2 < i-p+1 < k+2$, this can't happen.

If $n = -1$: use that $y = k+1 > i-p+1 > j-p = z$ to get one of $(i,k+1,j-p)$, $(k+1,i-p+1,j-p)$ and $(k+1,j-p,i-2p+2)$.

If $n = 0$, at most $(i+p-1,j,k)$, $(i,j,k)$, $(j,k,i-p+1)$, $(j,k,i-2p+2)$ arise.

If $n = 1$, the only possibility is $(j+p,i,k-1)$, since $(j+p) - (k-1) > p-1$ and $j+p > i > k-1$.

(iii) Here there is a permutation $(x,y,z)$ of $(x',y',z')$ such that $x+py+p^2z \equiv m \pmod {p^3-1}$.
This is invariant under the change of coordinates
\begin{equation*}
  \theta' : (x,y,z;i,j,k) \mapsto (z,x,y;k+p,i-1,j).
\end{equation*}
So, without loss of generality, either $x \ge y \ge z$ or $x < y < z$.

In the first case, $(x',y',z') = (x,y,z)$. Without loss of generality, $A+pB + p^2 C = 0$, with $A =
x-i$, $B = y-j$, $C = z-k$.  Noting that
\begin{align*}
  |A-C| &= |(x-z)-(i-k)| \\
  &\le \max(p,2p-3) \le 2p-2, \\
  |B-C| &\le p-1,
\end{align*}
it follows that
\begin{equation*}
  |(1+p+p^2)C| = |(A-C) + p(B-C)| \le p^2+p-2.
\end{equation*}
Thus $C = 0$, and $A+pB = 0$ implies
\begin{equation*}
  |(1+p)B| = |A-B| \le p
\end{equation*}
and hence $B = A = 0$. So, $(x',y',z') = (i,j,k)$.

In the second case, a completely analogous argument shows that $(x',y',z') = (k+p,j+1,i-p-1)$.
\end{proof}

\begin{proof}[Proof of prop.~\ref{prop:combi_conj_gl3}] 
First note that, for $\lambda \in X_1(T)$,
\[ \RR(\JH(W(\lambda))) \]
consists of $F := \RR(F(\lambda))$ and, if $F \in C_0$, also~$\r F$. Also note that for $(x',y',z') \in X_1(T)$,
$F(x'-2,y'-1,z')\subreg = \RR(F(z'+p-1,y',x'-p+1))$ (note that the latter weight is also restricted). Thus
\begin{equation}\label{eq:AARR}
  \AA(x',y',z') = \RR(\JH(W(z'+p-1,y',x'-p+1))).
\end{equation}
With the convention that $\AA(\lambda) := \varnothing$ ($\lambda \not\in X_1(T)$), $\RR(0) := \varnothing$, \eqref{eq:AARR} is even true 
for any $(x',y',z') \in X(T)$ satisfying $x'-y' = p$ or $y'-z' = p$ or $x'-z' = 2p$ (by~\eqref{eq:change_of_weyl_chamber}).

If $\tau \cong \tau(1,(i,j,k)) \sim \nivonematrix$, without loss of generality, $i \ge j \ge k$, $i-k \le p-1$.
By thm.~\ref{thm:jantzen_formula}, $\overline{R_1(i,j,k)}$ equals
\begin{gather*}
  W(k+p-1,j,i-p+1) + W(i,k,j-p+1) + W(j+p-1,i,k) \\
  {}+ W(i,j,k) + W(j,k,i-p+1) + W(k+p-1,i,j).
\end{gather*}
The lemma follows from~\eqref{eq:AARR}, term by term.

If $\tau \sim \nivtwomatrix$, we can write $m = j+pk$ with
(unique) $i \ge j > k$, $i-k \le p-1$ (replacing $m$ with~$pm$ if necessary). Then $\tau \cong \tau((2\; 3), (i,j,k))$ and
$\overline{R_{(2\; 3)}(i,j,k)}$ equals
\begin{gather*}
  W(k+p-1,j,i-p+1) + W(i,k,j-p+1) + W(j+p-2,i,k+1) \\
  {}+ W(i,j-1,k+1) + W(j-1,k+1,i-p+1) + W(k+p-2,i,j+1).
\end{gather*}
Note that the last two weights in the lemma do not contribute (\eg, for $(i+p-1,j,k)$ to occur we need $i=j$
in which case $F(i+p-3,j-1,k)\subreg = F(i-2,j-1,k)\subreg$). The remaining six weights $(x',y',z')$ all
verify $x'-y'$, $y'-z' \in [0,p]$.
The lemma follows from~\eqref{eq:AARR}, term by term.

If $\tau \sim \nivthreematrix$, a simple exercise shows that either $m$ or $-m$ equals
$i+pj+p^2k$ for some (unique) $i > j \ge k$, $i-k\le p$. In the first case,
$\tau \cong \tau((1\; 2\; 3), (i,j,k))$ and $\overline{R_{(1\; 2\; 3)} (i,j,k)}$ equals
\begin{gather}
  \notag W(k+p-1,j,i-p+1) + W(i-1,k,j-p+2) \\
  \label{eq:niv3_reduction}{} + W(j+p-1,i-1,k+1) + W(i-2,j+1,k+1) \\
  \notag {} + W(j-1,k+1,i-p+1) + W(k+p-2,i,j+1).
\end{gather}
Everything works as in the previous situation, except that the fourth through the sixth weight in the lemma can fail to be restricted by
having their second and third coordinate differ by $p+1$. Using the cyclic symmetry~$\theta'$ exploited in the lemma, we may assume without
loss of generality that $i = k+p$ and $i \ne j+1$ (because not all three equalities can hold simultaneously). Then we can already match the
first four terms of~\eqref{eq:niv3_reduction} with the first four weights in the lemma using~\eqref{eq:AARR}. This is even true for the
fifth: that weight in the lemma fails to be restricted iff $j - k \le 1$ and then either $y'-z' = p$ or $x'-z' = 2p$. If $j - k = p - 1$ the
same argument works for the sixth also, so let us assume that $j - k < p-1$.

Note that term~6 in~\eqref{eq:niv3_reduction} equals $-F(i-1,k+p-1,j+1)$ (by~\eqref{eq:change_of_weyl_chamber}) which cancels the
irreducible constituent in~$C_0$ of the reducible $W(j+p-1,i-1,k+1)$ (term~3).  We will be done if we show that $\RR(F(i-1,k+p-1,j+1))$ is
contained in the union of $\RR(\JH(W))$ where $W$ runs over terms~1, 2, 4, 5 in~\eqref{eq:niv3_reduction}. Term~2 suffices:
\begin{align*}
  \RR(F(i-1,k,j-p+2)) = \RR(F(i-1,k+p-1,j+1)).
\end{align*}

In the second case, we dualise: in light of prop.~\ref{prop:conj_basicprops} we only have to show that $\CC(\tau\dual) = \{ -w_0 \lambda :
\lambda \in \CC(\tau)\}$ and that $\r$ and~$\dual$ commute on regular Serre weights, but this is obvious.
\end{proof}

\begin{thm}\label{thm:comparison-with-adps}\

  \textup{(i)} If $\tau$ is of niveau~1, the regular Serre weights predicted in~\cite{bib:ADP} agree exactly with the ones here.
  
  \textup{(ii)} If $\tau$ is of niveau~2, we can write $\tau \sim \nivtwomatrix$, with $m = j + pk$ and $i \ge j > k$, $i - k \le p-1$ \(up
  to swapping $m$ and $pm$\). Then the regular Serre weights predicted in~\cite{bib:ADP} are precisely the ones given by formula~\eqref{eq:wts_combinatorial}
  when the sixth weight on the list in lemma~\ref{lm:combi_conj_gl3}\(ii\) is removed.

  \textup{(iii)} If $\tau$ is of niveau~3, we can write $\tau \sim \nivthreematrix$ with
  $m = i +pj+p^2 k$ and $i > j \ge k$, $i-k \le p$ \(up to dualising~$\tau$\). Then the regular Serre weights
  predicted in~\cite{bib:ADP} are precisely the ones given by formula~\eqref{eq:wts_combinatorial}
  when the following weights are removed from the list in lemma~\ref{lm:combi_conj_gl3}\(iii\): the last three and those among the
  first three of the form $(x',y',z')$ with $x'-z' = p$ and $x'-1 > y' > z'$.
\end{thm}

\begin{proof} We use the explicit description of $\CC(\tau)$ in terms of congruences as in the proof of~\ref{lm:combi_conj_gl3}.

(i) This is obvious.

(ii) Note that according to~\cite{bib:ADP} we write $m = j+pk$ (note that $0 \le j-k \le p-1$) and
$pm \equiv (k+p) + p(j-1) \pmod{p^2-1}$ (note that $0 \le (k+p) - (j-1) \le p-1$ unless $j = k+1$, in which case $pm$ cannot
be expressed in this way). So the regular weights predicted there are $F(i-2,j-1,k)\subreg$, $F(j-2,i-1,k)\subreg$,
$F(j-2,k-1,i)\subreg$ and, if $j \ne k+1$, $F(i-2,k+p-1,j-1)\subreg$, $F(k+p-2,i-1,j-1)\subreg$, $F(k+p-2,j-2,i)\subreg$
together with the reflections~$\r F$ for any~$F$ in this list that is in the lower alcove. Suppose first that $j \ne k+1$.
As $F := F(k+p-2,i-1,j-1) \in C_0$ and $\r F = F(j-2,i-1,k)\subreg$, the latter weight is redundant in the list just given and we obtain
the union of $\AA(i,j,k)$, $\AA(j,k,i-p+1)$, $\AA(i,k+1,j-p)$, $\AA(k+p,i,j-1)$, $\AA(k+p,j-1,i-p+1)$ as required. If $j = k+1$,
the fourth and fifth weight in lemma~\ref{lm:combi_conj_gl3}\(ii\) fail to be restricted and we can match up the three terms on the list just given
with the first three weights in the lemma by noting that $F(j-2,i-1,k)\subreg = F(k+p-2,i-1,j-1)\subreg$.

(iii) Let $\alpha := p-(i-k)$, $\beta := j-k$, $\gamma := i-j-1$. These are permuted by~$\theta'$ from the proof of
lemma~\ref{lm:combi_conj_gl3}(iii) and we can assume without loss of generality that either (a) $\alpha$, $\beta$, $\gamma$ are all non-zero,
(b) $\alpha = 0$ and the other two non-zero, or (c) $\alpha = \beta = 0$, $\gamma\ne 0$. Note that one of the first three weights
in the lemma will be excluded by the condition in the theorem iff we are in case (b) in which case precisely $(i,j,k)$ is affected.

If (a) holds, we write $m = i+pj+p^2 k$, $pm \equiv (k+p)+p(i-1) +p^2 j \pmod {p^3-1}$, $p^2 m \equiv (j+p) + p(k+p-1) + p^2(i-1) \pmod
{p^3-1}$. So the regular weights predicted by~\cite{bib:ADP} are $F(i-2,j-1,k)\subreg$, $F(k+p-2,i-2,j)\subreg$, $F(j-1,k-1,i-p)\subreg$
together with the reflections~$\r F$ for any~$F$ in this list that is in the lower alcove. Now note that the first three weights in the lemma
are all restricted.

If (b) holds, the expression for~$m$ we have to use is $m = k+p(j+1)+p^2 k$ and the weights predicted by~\cite{bib:ADP} are as in (a)
except that the first becomes $F(j-1,k-1,k)\subreg$ which equals the third. On the other hand, we should only use the second
and the third weights of the lemma, and we are fine as both are restricted.

If (c) holds, the expressions for~$m$ and $p^2 m$ are as in (b) whereas $pm$ does not have an expression of the required form. We are fine again
as precisely the first weight among the first three in the lemma fails to be restricted.
\end{proof}

\begin{rk}\label{rk:doud_extension}
  Doud independently extended the conjecture of~\cite{bib:ADP} to include the remaining weights in niveau~3 predicted here \cite{bib:Doud_ss}.
\end{rk}

\section{Computational evidence for the conjecture}\label{sec:comp_evid}

\subsection{Verification of ``extra weights''}

In~\cite{bib:ADP}, Ash, Doud and Pollack consider various explicit irreducible, odd~$\rho$ that are tame at~$p$ and test computationally
whether eigenclasses to which $\rho$ is attached occur in the weights predicted by them (in level $N^?(\rho)$ and nebentype determined by $\det(\rho)$;
see~\cite[p.\,524]{bib:ADP}).  Among them are seven examples of such~$\rho$ of niveau~2, for which conjecture~\ref{conj:serre} predicts one
further weight than the ADPS conjecture. There is another such example in~\cite[\S3]{bib:Doud_3d}. Darrin Doud and David Pollack agreed
to test with their respective computer programs the existence of an eigenclass with the correct eigenvalues in this ``extra weight.''
They indeed verified its existence (in the sense that~\eqref{eq:rho_attached} is satisfied for all $l \le 47$) except in the one case of
level $N=144$, which could not be handled by their programs.

To summarise, here is a table of the extra weight confirmed in each case:
\begin{equation*}
  \begin{array}[h]{c|c|c|c}
      p & \text{level(s) $N$} & \rho|_{I_p} & \text{weight} \\[2pt]
      \hline
      \topspace{20pt}
      5 & \text{73, 83, 89, 
        151, 157} & \bigg(\begin{smallmatrix} \omega_2^8 \\ & \omega_2^{16} \\ && 1 \end{smallmatrix}\bigg) & F(6,3,0) \\[10pt]
      7 & 67 & \bigg(\begin{smallmatrix} \omega_2^{12} \\ & \omega_2^{36} \\ && \omega^3 \end{smallmatrix}\bigg) & F(13,8,3) \\[10pt]
      11 & 17 & \bigg(\begin{smallmatrix} \omega_2^{40} \\ & \omega_2^{80} \\ && 1 \end{smallmatrix}\bigg) & F(16,9,2) \\
  \end{array}
\end{equation*}

The image of~$\rho$ in these cases is either $S_4$ ($N = 17,$ 67, 73), $A_5$ ($N = 89$, 151, 157) or a suitable semi-direct product $(\Z/3
\times \Z/3) \rtimes S_3$ when $N = 83$ \cite{bib:ADP}, \cite{bib:Doud_3d}.

\subsection{Exhaustive calculations}

In the example of level~73 listed above, Doud verified upon request that no eigenclasses to which $\rho$ is
attached occur in regular weights outside $W^?(\rho|_{I_p})$ (as before, in level $N^?(\rho)$ and nebentype determined by $\det(\rho)$).

In \cite[\S 4, \S 5.2, \S 5.3]{bib:Doud_ss}, Doud documents similar exhaustive calculations for several (tame)~$\rho$ of niveau~3, and the results
are again consistent with conj.~\ref{conj:serre}. (As noted in rk.~\ref{rk:doud_extension}, the
extension of the ADPS conjecture in~\cite{bib:Doud_ss} for~$\rho$ of niveau~3, which Doud found independently, agrees on the subset
of regular weights with $W^?(\rho|_{I_p})$.)
In one example only roughly half the non-predicted weights are ruled out due to computational limitations. 

\section{Evidence for a conjecture of Gee}

After an earlier version of this work~\cite{bib:thesis}, Toby Gee made another conjecture for the weights in this context in
terms of the existence of local crystalline lifts with prescribed Hodge--Tate numbers (in the spirit of the
Buzzard--Diamond--Jarvis conjecture)~\cite[\S 4.3]{bib:Gee-lifts}. This conjecture is motivated by the hope of being able to
globalise local lifts and the Fontaine--Mazur--Langlands conjecture. It naturally led him to make a second conjecture to the
effect that $F \in W(\rho)$ implies $F(\lambda) \in W(\rho)$ whenever the Serre weight~$F$ is a constituent of $W(\lambda)$
and $\lambda$ is restricted. In fact for groups that are compact at infinity and $\GL_n$ at~$p$, this second conjecture is
implied by the first.

We verify that Toby Gee's second conjecture holds for the conjectural weight set $W^?(\rho|_{I_p})$ in generic situations.

\begin{prop}\label{prop:gee_evidence}
  Suppose that $\lambda$ is \sd\ in a restricted alcove, $\lambda' \in X_1(T)$, and that $F(\lambda')$ is a Jordan--H\"older constituent
of $W(\lambda)$ as representation of $\gln$. Then for any tame $\tau : I_p \to \GL_n(\fpb)$ that can be extended to~$G_p$,
\begin{equation*}
  F(\lambda') \in W^?(\tau) \Rightarrow F(\lambda) \in W^?(\tau).
\end{equation*}
\end{prop}

\begin{proof}
By prop.~\ref{prop:stronglinking}, the constituents of $W(\lambda)$ as $\GL_n$-module are of the form $F(\mu)$ for dominant $\mu \uparrow \lambda$. We can choose
such a~$\mu$ such that $F(\lambda')$ is a constituent of $F(\mu)$ considered as representation of $\gln$ and we write $\mu = \mu_0 + p \mu_1$ with
$\mu_0 \in X_1(T)$, $\mu_1 \in X(T)_+$. Note that for $n$ fixed, $\mu_1$ can only take finitely many values modulo~$pX^0(T)$. Let us write
\begin{equation*}
  \ch F(\mu_1) = \sum_{\varepsilon \in X(T)} a_\varepsilon e(\varepsilon)\quad \text{with $a_\varepsilon \in \Z$}.
\end{equation*}

\medskip

\emph{Claim:} If $\lambda$ lies \sd\ in its alcove, then
\begin{equation*}
  F(\mu) = \sum_{\varepsilon \in X(T)} a_\varepsilon F(\mu_0 + \varepsilon)
\end{equation*}
in the Grothendieck group of $\gln$-representations.

\medskip

Restricting $\lambda$ in its alcove if necessary, we may assume that $\mu_0+\varepsilon$ lies in the same alcove
as~$\mu_0$ whenever $a_\varepsilon \ne 0$. In the Grothendieck group of $\GL_n$-modules we can write (using prop.~\ref{prop:stronglinking})
\begin{equation*}
  F(\mu_0) = \sum_{\mu_0' \uparrow \mu_0} b_{\mu'_0,\mu_0} W(\mu'_0),
\end{equation*}
where $b_{\mu'_0,\mu_0} = 0$ if $\mu'_0$ is not dominant. Using thm.~\ref{thm:Steinberg} and prop.~\ref{prop:Brauer_formula},
in the Grothendieck group of $\gln$-modules,
\begin{align*}
  F(\mu) &= F(\mu_0) \tens F(\mu_1) \\
  &= \sum_{\mu_0' \uparrow \mu_0} b_{\mu'_0,\mu_0} W(\mu'_0) \tens F(\mu_1) \\
  &= \sum_{\mu_0' \uparrow \mu_0}\sum_{\varepsilon \in X(T)} a_\varepsilon b_{\mu'_0,\mu_0} W(\mu'_0+\varepsilon) \\
  &= \sum_{\varepsilon \in X(T)} a_\varepsilon F(\mu_0 + \varepsilon).
\end{align*}
The last step made use of the translation principle \cite[II.7.17(b)]{bib:Jan-reps}, which implies that the $b_{\mu'_0,\mu_0}$ only depend on the alcoves
$\mu'_0$ and~$\mu_0$ lie in, and the fact that the $a_\varepsilon$ depend only on the $W$-orbit of~$\varepsilon$.

Using the claim, $F(\lambda') \cong F(\mu_0 +\varepsilon)$ for some weight~$\varepsilon$ of $F(\mu_1)$ and some dominant
$\mu \uparrow \lambda$. If $F(\lambda') \in W^?(\tau)$, $\tau \cong \tau(w,\lambda''+\rho)$ for some dominant $\lambda'' \uparrow \mu_0+\varepsilon$
by prop.~\ref{prop:generic_pred}.
But by the remark after def.~\ref{df:uparrow} such a~$\lambda''$ is of the form $\mu_0' + w'\varepsilon$ for some dominant $\mu'_0 \uparrow \mu_0$ and some $w' \in W$
(in fact, $w'$ underlies the affine Weyl group element taking the alcove of $\mu'_0$ to the alcove of~$\mu_0$). The following simple
manipulation---using~\eqref{eq:jantzen-action} and valid for all $\sigma \in W$---is the key point of the proof:
\begin{multline}\label{eq:3}
  \tau \cong \tau(w,\mu'_0 + w'\varepsilon + \rho) \cong \tau(w,\mu'_0 + pw^{-1}w'\varepsilon + \rho) \\ \cong \tau(\sigma w\sigma^{-1},\sigma\cdot (\mu'_0 +
  pw^{-1}w'\varepsilon) + \rho).
\end{multline}
We choose $\sigma \in W$ so that $\sigma\cdot (\mu'_0 + pw^{-1}w'\varepsilon)$ is dominant. Note that $a_\varepsilon \ne 0$ implies that
$\pi\varepsilon \le \mu_1$ for all $\pi \in W$ \cite[II.2.4]{bib:Jan-reps}. Then the following lemma applies and shows that
\begin{equation}\label{eq:4}
  \sigma\cdot (\mu'_0 + pw^{-1}w'\varepsilon) \uparrow \mu'_0 + p\mu_1 \uparrow \mu_0 + p \mu_1\uparrow \lambda
\end{equation}
(using \cite[II.6.4(4)]{bib:Jan-reps}). Finally apply prop.~\ref{prop:generic_pred} to~\eqref{eq:3}.
\end{proof}

\begin{lm}
  Suppose that $\mu$, $\nu \in X(T)_+$. If $\varepsilon \in X(T)$ such that $w\varepsilon \le \nu$ for all $w \in W$ then
  \[ \sigma \cdot (\mu + p\varepsilon) \uparrow \mu + p\nu \quad \forall \sigma \in W. \]
\end{lm}

\begin{rk}
  In fact the converse is true if $\mu \in C_0$ \(but not in general\).
\end{rk}

\begin{proof}
  We will use two reduction steps:

  \begin{enumerate}
  \item[(R1)] Suppose the lemma is true for~$\varepsilon$ and that $\alpha \in R^+$ such that $\langle \varepsilon,\alpha\dual\rangle \ge 0$.
    Then the lemma is true for $\varepsilon - i\alpha$ for all $0 \le i \le \langle \varepsilon,\alpha\dual\rangle$.
  \item[(R2)] Suppose that $\varepsilon \lneqq \nu$ are both dominant. Then there exists $\alpha \in R^+$ such that $\varepsilon \le \nu -\alpha$ and
    $\nu - \alpha$ is dominant.
  \end{enumerate}
  
  Assume first the validity of these two claims. Note that the lemma is true for $\varepsilon = \nu$ \cite[II.6.4(5)]{bib:Jan-reps}.
  Suppose next that $\varepsilon$ is dominant.  By (R2) there is a sequence $\varepsilon = \varepsilon_0 \le \varepsilon_1 \le \dots \le
  \varepsilon_r = \nu$ with $\varepsilon_j$ dominant and $\beta_j := \varepsilon_j-\varepsilon_{j-1} \in R^+$ for all $j > 0$. Note that
  $\langle \varepsilon_j,\beta_j\dual\rangle = \langle \varepsilon_{j-1},\beta_j\dual\rangle + \langle \beta_j,\beta_j\dual\rangle \ge 2$.
  Then (R1) with $i = 1$ implies inductively that the lemma is true for~$\varepsilon$. Finally for a general~$\varepsilon$ choose $w \in W$
  such that $w\varepsilon$ is dominant. Write $w = s_1\cdots s_r$, a reduced expression in terms of simple reflections~$s_j$. A standard
  argument shows that $\varepsilon = \varepsilon_r \le \varepsilon_{r-1} \le \dots \le \varepsilon_0 = w\varepsilon$ with
  $\varepsilon_j = s_{j+1}\cdots s_{r-1}s_r \varepsilon$.  Since the lemma is true for $w\varepsilon$, (R1) with $i$ maximal shows
  inductively that the lemma is true for~$\varepsilon$.

  To prove (R1), choose $w \in W$ such that $\lambda := w\cdot (\mu+p\varepsilon) \in X(T)_+-\rho'$. Then
  \begin{equation*}
    0 \le pi < \langle \mu+\rho',\alpha\dual\rangle + p\langle \varepsilon,\alpha\dual\rangle = \langle \lambda+\rho',w \alpha\dual \rangle.
  \end{equation*}
  In particular, $w \alpha \in R^+$. Then \cite[II.6.9]{bib:Jan-reps} applies (note that the case $i = 0$ is vacuous and use \cite[II.6.4(5)]{bib:Jan-reps}):
  \begin{equation*}
    \sigma\cdot (s_{w\alpha}w \cdot (\mu+p\varepsilon) + p i w\alpha) \uparrow \lambda\quad \forall \sigma \in W.
  \end{equation*}
  Replacing $\sigma$ by $\sigma s_{w\alpha}w$ and using that the lemma holds for~$\varepsilon$ proves (R1):
  \begin{equation*}
    \sigma \cdot (\mu + p(\varepsilon-i\alpha)) \uparrow \mu + p\nu \quad \forall \sigma \in W.
  \end{equation*}

  (R2) is also known as Stembridge's lemma and is true for arbitrary root systems; see~\cite[2.3]{bib:Rapoport_Satake} for a short proof due to Waldspurger.
\end{proof}

\section{Theoretical evidence for the conjecture}\label{sec:theo_evid}

Recall that we assume that $n > 1$. Let $\A := \A_\Q$ and define
\begin{align*}
  U_1(N) &:= \{ g \in \GL_n(\ZZ) : \text{last row $\equiv (0,\dots,0,1) \pmod N$} \}, \\
  \Sigma_1(N) &:= \{ g \in \GL_n(\A^\infty) : g_N \in U_1(N) \},
\end{align*}
where $g_N = \prod_{l | N} g_l$. Then $(\uo,\sigo)$ is a Hecke pair, and we denote by $\ha$ the associated Hecke algebra.
Recall that the Hecke pair $(\go,\so)$ and the associated Hecke algebra $\ho$ were defined in~\S\ref{sub:hecke_alg}.

\begin{lm}\label{lm:loc_glob_hecke}
  There is an isomorphism of Hecke algebras
  \begin{align*}
    \ha &\congto  \ho \\
  \intertext{determined by requiring that}
    [\uo s \uo] &\mapsto [\go s \go]
  \end{align*}
  for all $s \in \so$.
\end{lm}

\begin{proof}
It suffices to show that $(\go,\so) \subseteq (\uo,\sigo)$ are strongly compatible Hecke pairs (\S\ref{sub:hecke}).
To see that $\so\uo = \sigo$, note that by strong approximation and as $n > 1$,
\[ \GL_n(\Q) \uo = G(\A^\infty) \supseteq\sigo, \]
so for $\sigma \in \sigo$ write
$\sigma = \gamma u$ ($\gamma \in \GL_n(\Q)$, $u \in \uo$). Without loss of generality, $\det \gamma > 0$. Then it follows immediately that
$\gamma \in \so$. Also, $\uo \cap \so^{-1}\so = \go$ is obvious. 

Finally we need to show that $\uo s\uo = \go s \uo$ for all $s\in\so$, or equivalently that $\uo = \go (\uo\cap {}^s \uo)$.
As $s_N \in \uo$ and $\uo$ is compact open, $\uo \cap {}^s \uo \supseteq \uo \cap U(M)$
for some $(M,N) = 1$, where $U(M) = \{g \in \GL_n(\ZZ) : g \equiv 1 \pmod M\}$. Since $\go \onto \SL_n(\Z/M)$, it follows
that
\[ \{u \in \uo : \det u \equiv 1 \pmod M\} \subseteq \go(\uo\cap {}^s \uo). \]
The desired equality follows by noting that the determinant of the right-hand side is~$\ZZ\s$, which can be seen by using the theorem on elementary
divisors for all $l | M$.
\end{proof}

The following proposition will be used to obtain cohomology classes from algebraic automorphic representations. It is similar in spirit to
\cite[\S3]{bib:AStev} for $n = 3$. For $x \in \R$, $\lfloor x\rfloor$ denotes the largest integer less than or equal to~$x$.

\begin{prop}\label{prop:coho_autrep}
  Suppose that $\pi$ is a cuspidal automorphic representation of $\GL_n(\A_\Q)$ of conductor~$N$. Suppose moreover that
  for some integers
  \[ c_1 > c_2 > \dots > c_n, \]
  $\pi_\infty$ corresponds, under the Local Langlands Correspondence, to a representation of $W_\R$ sending $z\in\C\s$ to
  \[ \dia(z^{-c_1}\bar z^{-c_n},z^{-c_2}\bar z^{-c_{n-1}},\dots,z^{-c_n}\bar z^{-c_1})\tens (z\bar z)^{(n-1)/2} \in \GL_n(\C) \]
  and~$j$ to an element of determinant $(-1)^{\sum c_i + \lfloor n/2\rfloor}$ \(in particular, $\pi$ is regular algebraic; c.f.\
  \cite{bib:Clozel}, def.~1.8 and def.~3.12\). Let~$r$ be the irreducible representation of $\GL_{n/\C}$ with highest weight
  $(c_1-(n-1),c_2-(n-2),\dots, c_n)$. Then there is an $\ho$-equivariant injection
    \[ (\pi^\infty)^\uo \INTO H^e(\go,r) \]
  for any~$e$ in the range
  \begin{equation}\label{eq:range_of_e}
    \Big\lfloor\Big(\frac n2\Big)^2\Big\rfloor \le e < \Big\lfloor\Big(\frac {n+1}2\Big)^2\Big\rfloor.
  \end{equation}
\end{prop}

\begin{rk}\ 

  \begin{enumerate}
  \item   As $N$ is the conductor of~$\pi$, $(\pi^\infty)^\uo$ is one-dimensional. Thus we get a Hecke eigenclass in group cohomology.
  \item It is known that $\go$ has virtual cohomological dimension $n(n-1)/2$. In particular, $H^e(\go,r) = 0$ for $e > n(n-1)/2$
    \(see \cite{bib:Serre_coh_disc}, p.~132 and the remark on p.~101\).
  \end{enumerate}
\end{rk}

\begin{proof} Let $G := \GL_{n}$. For any open compact subgroup $U \subseteq G(\A^\infty)$, let
\begin{gather*}
  \widetilde X_U := G(\R)/ \rO(n) \times G(\A^\infty)/U, \\
  X_U = G(\Q) \backslash \big( G(\R)/ \rO(n) \times G(\A^\infty)/U \big),
  \end{gather*}
and denote by $\pi_U : \widetilde X_U \to X_U$ the natural projection. Then $\widetilde X_U$ and~$X_U$ are
real manifolds of dimension $\binom{n+1}2$ ($X_U$ is not necessarily connected). If $U$ is sufficiently small, $G(\Q)$
acts properly discontinuously on $\widetilde X_U$ and the constant sheaf on $\widetilde X_U$ with fibre~$r$
gives rise to a local system on the quotient $X_U$, which will be denoted by~$\L_r$: for any open subset $Z \subseteq X_U$,
$\L_r(Z)$ is the set of locally constant functions
\begin{equation}
  \label{eq:loc_sys}
  \{f: \pi_U^{-1}(Z) \to r : f(\gamma x) = \gamma f(x) \ \forall \gamma \in G(\Q),\ x \in \pi_U^{-1}(Z)\}.  
\end{equation}

Notice that $r\dual$ is the representation of~$G$ associated to $\pi_\infty$ defined in~\cite{bib:Clozel}, pp.~112--113
(where it is denoted by~$\tau$). By~\cite[3.15]{bib:Clozel} there is a $G(\A^\infty)$-equivariant injection
\[ \bigoplus_\Pi H^e (\ssl_n,\rO(n); \Pi_\infty \tens r) \tens \Pi^\infty \INTO \ilim_V H^e(X_V,\L_r) \]
where $\Pi$ runs through all cuspidal automorphic representations of $G(\A_\Q)$ whose central character agrees with that of $r\dual$ on
$\R\s_+$, and where the limit is over all (sufficiently small) compact open subgroups $V \subseteq G(\A^\infty)$. The cohomology groups
on the left-hand side are $(\mathfrak{g},K)$-cohomology.
 The $G(\A^\infty)$-action on the
right-hand side is as in sublemma~\ref{sublm:coho_locsym}(ii) below. Here $\ssl_n$ denotes the complexified Lie algebra of $\SL_n(\R)$.

When $n$ is even, lemma~3.14 in~\cite{bib:Clozel} shows that
\[ H^e (\ssl_n,\rO(n); \pi_\infty \tens r) \cong \wedge^{e-n^2/4}\, \C^{n/2-1}, \]
(by the remark on p.~120 in the same reference, there is no quadratic character appearing on the left-hand side).

When $n$ is odd, the condition on the determinant of~$j$ made above implies that $\pi_\infty$ is the induction, using
a parabolic subgroup of type $(2,2,\dots, 2,1)$, of $\sigma_1\tens \sigma_2\tens \cdots \tens \sigma_{(n-1)/2} \tens \chi$ (keeping Clozel's notation), where
$\sigma_i$ is the discrete series representation of central character $|.|^{-c_i-c_{n+1-i}+n-1}\sgn^{c_i+c_{n+1-i}+1}$ and lowest weight
$c_i-c_{n+1-i}+1$, and $\chi = |.|^{-c_{(n+1)/2}+(n-1)/2}\sgn^{c_{(n+1)/2}}$. This has the consequence that the character considered in~\cite{bib:Clozel},
p.~120 is \emph{even} and again we get (without quadratic character on the left-hand side):
\[ H^e (\ssl_n,\rO(n); \pi_\infty \tens r) \cong \wedge^{e-(n^2-1)/4}\, \C^{(n-1)/2}. \]

Thus we get an $\ho$-equivariant homomorphism
\[ (\pi^\infty)^\uo \INTO \Big(\ilim_U H^e(X_U,\L_r)\Big)^\uo \]
for any~$e$ in the range claimed above. It remains to identify the right-hand side as a group cohomology module.

Let $H^e(X,\L_r) = \ilim_V H^e(X_V,\L_r)$ to simplify notation ($X$ itself will not have any meaning). The following elementary sublemma
will be useful.

\begin{sublm}\label{sublm:coho_locsym}
  Suppose that $U$, $V$ are sufficiently small compact open subgroups of $G(\A^\infty)$ and $e\ge 0$ arbitrary.
  \begin{enumerate}
  \item If $U \subseteq V$ consider the natural projection map $f: X_U \to X_V$. Then $f^* \L_r \cong \L_r$ \(canonically\) and
    the induced map $f^* : H^e(X_V,\L_r) \to H^e(X_U,\L_r)$ is an injection.
  \item If $g\in G(\A^\infty)$ and $U \subseteq gVg^{-1}$, denote by $[g]$ the natural map $X_U \to X_V$ given by right multiplication by~$g$.
    Again there is a canonical isomorphism $[g]^* \L_r \cong \L_r$ and an induced map $[g]^* : H^e(X_V,\L_r) \to H^e(X_U,\L_r)$.
    It is compatible with the maps defined in \(i\) and yields a smooth left action of $G(\A^\infty)$ on the direct limit $H^e(X,\L_r)$.
  \item The image of the natural map $H^e(X_U,\L_r) \to H^e(X,\L_r)$, which is an injection by \(i\), is precisely the subspace of $U$-invariants.
  \end{enumerate}
\end{sublm}

Choose an auxiliary prime $q \nmid 2N$, and let
\[ U = \{ g \in U_1(N) : g \equiv 1 \pmod q \} \unlhd\, \uo.\]
The projection of~$U$ to $G(\Q_q)$ contains no elements of finite order, which
implies that $U$ is sufficiently small in the above sense, so that $\L_r$ is defined on~$X_U$. (In fact, any other
sufficiently small open normal subgroup~$U$ of~$\uo$ would do.) By the sublemma,
$H^e(X,\L_r)^\uo = H^e(X_U,\L_r)^{\uo/U}$.

For now, we allow $r$ to be any $\C[G(\Q)]$-module. Let $\Gamma := G(\Q) \cap \uo$, an arithmetic subgroup of~$G$.

\medskip

\emph{Claim:} $H^\expbull(X_U,\L_r)^{\uo/U}$ and $H^\expbull(\Gamma,r)$ are universal $\delta$-functors, and they are canonically isomorphic.

\medskip

First note that if $H \le K$ are two groups and $V$ is an injective $K$-module (over~$\C$, say), then $V|_H$ is an injective $H$-module.
The reason is that the left adjoint of the forgetful functor $\textbf{$K$-mod} \to \textbf{$H$-mod}$ is
$\C K \tens_{\C H}\mathrm{-}$, which is exact. By putting $H = \Gamma$, $K = G(\Q)$, we see that $H^\expbull(\Gamma,r)$ is a universal
$\delta$-functor.

As for $H^\expbull(X_U,\L_r)^{\uo/U}$, note that it is at least a $\delta$-functor: $\uo/U$ is a
finite group so that taking $\uo/U$-invariants is an exact functor (we are in characteristic zero!). To demonstrate universality, it suffices
to show that $H^e(X_U,\L_r) = 0$ if $e > 0$ and $r$ is an injective $\C[G(\Q)]$-module. By the strong approximation theorem,
\[ G(\A) = \coprod_{i=1}^t G(\Q)g_i UG(\R)  \]
for some $g_i \in G(\A^\infty)$, which implies that
\[ X_U \cong \coprod_{i=1}^t (G(\Q) \cap {}^{g_i} U)\backslash G(\R)/\rO(n). \]
Under this isomorphism, $\L_r$ gives rise to a local system on each space in the disjoint union.
It is easy to see that on the $i$-th piece it is the
one induced by the constant sheaf on $G(\R)/\rO(n)$ with fibre~$r$ under the $(G(\Q) \cap {}^{g_i} U)$-action (as
in~\eqref{eq:loc_sys}). It will be denoted by~$\L_r$ as well. 
By~\cite{bib:Tohoku}, corollaire~3 to th\'eor\`eme~5.3.1, $H^e((G(\Q) \cap
{}^{g_i} U)\backslash G(\R)/\rO(n),\L_r) = 0$ if $e > 0$ and $r$ injective as $(G(\Q) \cap {}^{g_i} U)$-module; in particular if $r$ is
injective as $G(\Q)$-module. (Note that for the constant sheaf~$\underline r$, $H^i(G(\R)/ \rO(n),\underline r) = 0$ for $i > 0$ since
$G(\R)/ \rO(n)$ is contractible; see~\cite{bib:Bredon}, thm.~III.1.1 for the comparison of sheaf cohomology with singular cohomology.)

To check that the two universal $\delta$-functors above are canonically isomorphic, it is enough to identify them in degree~0. By~\eqref{eq:loc_sys},
$H^0(X_U,\L_r)^{\uo/U}$ is the set of locally constant, $G(\Q)$-invariant functions $f:G(\A)/\uo \rO(n) \to r$. By the strong approximation
theorem, using that $\det \uo = \ZZ\s$, such a function is determined by its values on $G(\R)$; by local constancy it is even determined
by $f(1)\in r$. It follows easily that the set of possible values of $f(1)$ is precisely $r^\Gamma = H^0(\Gamma,r)$. This establishes the claim.

\medskip

\emph{Claim:} The map of $\delta$-functors $H^\expbull(\Gamma,r) \xrightarrow{\res} H^\expbull(\go,r)$ is a (canonically split) injection.

\medskip

As $(\Gamma:\go) = 2$ (sign of the determinant), this is clear: $\frac 12 \cores$ provides the splitting, where
$\cores$ is the corestriction map.

\medskip

\emph{Claim:} The above canonical injection
\[ H^e(X,\L_r)^\uo \INTO H^e(\go,r) \]
of $\delta$-functors is $\ha \cong \ho$-equivariant.

\medskip

Note that the Hecke action on the left is defined in terms of the $G(\A^\infty)$-action of sublemma~\ref{sublm:coho_locsym}, whereas the one on the
right is the usual one on group cohomology (see~\S\ref{sub:hecke}). Both Hecke actions are $\delta$-functorial, so again it suffices to check
the claim in degree~0. Given $s \in S_1(N)$, we know by lemma~\ref{lm:loc_glob_hecke} that the Hecke operator $T_s = [\go s\go]\in \ho$ corresponds to
$T_s = [\uo s\uo] \in \ha$. Moreover, the strong compatibility~\eqref{lm:loc_glob_hecke} implies that if $s_i \in \so$ ($1 \le i\le n$) are chosen such that
\[ \go s\go = \coprod s_i\go, \]
then also
\[ \uo s\uo = \coprod s_i\uo. \]
An element of $H^0(X,\L_r)^\uo$ is a locally constant, $G(\Q)$-invariant function
\begin{equation*}
  f:G(\A)/\uo \rO(n) \to r
\end{equation*}
which is determined by $f(1) \in
r^\Gamma \subseteq r^\go$. By the sublemma, $T_s(f)$ is the function sending $g\in G(\A)$ to $\sum f(gs_i)$; in particular, the image of~1 is $\sum f(s_i) =
\sum s_i f(1) = T_s(f(1))$ (we used that $f$ is locally constant).  This verifies the Hecke equivariance.\end{proof}

The following lemma will be needed below. If $K$ is a CM field, we denote by~$K^+$ its totally real subfield, so that
$[K:K^+] \le 2$. By the Galois group of a number field~$K$ we mean the Galois group of the normal closure of $K/\Q$.

\begin{lm}\label{lm:CM_existence} Suppose that $p > 2$.

  \begin{enumerate}
  \item The Galois group of a quartic \(i.e., degree~4 over~$\Q$\), totally complex CM field can be either of $\Z/2 \times \Z/2$, $\Z/4$ or~$D_8$.
  \item There is a quartic, totally complex CM field~$K$ with Galois group $\Delta \cong D_8$, unramified at~$p$ such that $\Frob_{p}\in \Delta$ is
    \(a\) trivial, \(b\) the complex conjugation, \(c\) a \(non-central\) element of order~2 not fixing $K^+$, \(d\) a non-central element of order~2 fixing $K^+$, 
    or \(e\) an element of order~4.
  \end{enumerate}
\end{lm}

Note that (ii)(a)--(e) exhaust the conjugacy classes of~$\Delta$. The analogous result is true for the other two kinds of quartic, totally complex CM fields
and also if $p = 2$ \cite[\S13]{bib:thesis}.

For both the proof of the lemma and prop.~\ref{prop:GL4_theo_evid} below it will be useful to keep at hand a diagram of the subgroup lattice of~$D_8$,
together with explicit generators of each subgroup.

\begin{proof}
  (i) The Galois group is a transitive permutation group on four letters which has a central element of order 2 (as $L$ is CM). The result
  follows by considering the centralisers of a 2-cycle (it is the Klein 4-group) and of a permutation of cycle type $(2,2)$ (it is dihedral
  of order 8).

(ii) It would be possible to give a proof which works more generally, as alluded to in rk.~\ref{rk:theo_evid_n=2m}. We give a more direct argument instead.

Consider $K = \Q(\sqrt{a+b\sqrt d})$ with integers $a$, $b$, $d$, with normal closure (over~$\Q$) denoted by~$L$. If $d > 0$ and $a^2-b^2d >
0$ lie in different, non-trivial square classes of~$\Q\s$ and $a < 0$ then $K$ is a quartic CM field with dihedral Galois group of order~8.  For, $K$ is a
totally complex quadratic extension of $\Q(\sqrt d)$, a totally real quadratic field. Moreover, $K/\Q$ is not Galois, as it would otherwise
contain a square root of $(a+b\sqrt d)(a -b\sqrt d) = a^2-b^2d > 0$, which is ruled out by the assumptions.

Note that cases (c) and (d) are equivalent upon replacing $K$ by one of the two quartic, totally complex subfields $K' \subseteq L$ that are not conjugate to~$K$.
Let us henceforth assume that we are not in case (c).

In addition to requiring $a < 0$, $a^2-b^2d > 0$ and $d > 1$ with $d$ square-free, we also demand that $b > 0$, $a < -(b^2d+1)/2$ and that:
\begin{itemize}
  \setlength{\itemsep}{5pt}
\item $a \equiv d \equiv 1$, $b \equiv 0 \pmod p$ and $d \nmid a$ in case (a),
\item $\leg ap = -1$, $d \equiv 1$, $b \equiv 0 \pmod p$ and $d \nmid a$ in case (b),
\item $\leg{2a-1}p = -1$, $d \equiv 1$, $b \equiv a-1 \pmod p$ in case (d),
\item $\leg dp = \leg{a^2-b^2d}p = -1$ and $d\nmid a$ in case (e).
\end{itemize}
(Choose $d$ first and $a$ last.)
In the fourth case, choose~$d$ with $\leg dp = -1$, $\leg{d-1}p = 1$.  Then $a \equiv d$, $b \equiv 1 \pmod p$ will work.

Clearly the conditions ensure that $a^2-b^2d$ and $d$ lie in different, non-trivial square classes.  The corresponding CM
field~$K$ is unramified at~$p$, as $d$ and $a^2-b^2d$ are prime to~$p$. In the first two cases, $L^+$ is split at~$p$, as
$\leg dp = \leg{a^2-b^2 d} p = 1$. Moreover $\qp(\sqrt{a+b\sqrt d}) = \qp$ in the first, but not the second, case as the
reduction mod~$p$ of $a+b\sqrt d$ is a square, resp.\ a non-square, in~$\fp\s$. Thus $K$ is as required in the first two
cases. In the third case, $K^+$ is split at~$p$ whereas the other two quadratic subfields of~$L$ are inert at~$p$,
establishing that $K$ is as in (d). The fourth case is similar with $F := \Q(\sqrt{d(a^2-b^2d)})$ split at~$p$ and the other
two quadratic subfields of~$L$ inert at~$p$, once we see that $F$ is indeed the subfield of~$L$ fixed by the elements of
order~4 in~$\Delta$. As $L = \Q(\alpha, \alpha')$ with $\alpha = \sqrt{a+b\sqrt d}$, $\alpha' = \sqrt{a-b\sqrt d}$, any
element of~$\Delta$ is determined by its action on~$\alpha$ and~$\alpha'$.  The conjugates of~$\alpha$ are $S =
\{\pm\alpha,\pm \alpha'\}$. Given $s_1$, $s_2 \in S$, $s_1 \ne \pm s_2$ there is a $\tau \in \Delta$ such that $\tau(\alpha)
= s_1$ and $\tau(\alpha') = s_2$ (as $\# \Delta=8$). Thus an element of order~4 in~$\Delta$ is given by~$\tau$ with
$\tau(\alpha) = \alpha'$, $\tau(\alpha') = -\alpha$. In particular, $\tau(\sqrt d) = -\sqrt d$ and hence $\tau$ fixes
$\alpha\alpha'\sqrt d$, as required.
\end{proof}

Fix an isomorphism $\iota: \C \to \qpb$.

\begin{prop}\label{prop:GL4_theo_evid} Suppose that $n = 4$ and that $p > 2$.
  Given $\mu \in X(T)_+$ with $\mu_1+\mu_4 = \mu_2+\mu_3$ and suppose that $w$ is in the dihedral subgroup $\langle (1\; 2\; 4\; 3),
  (1\;2)(3\;4) \rangle \subseteq S_4 \cong W$ of order~8.

  Then there is an irreducible, odd Galois representation $\rho : G_\Q \to \GL_4(\fpb)$ with $\rho|_{I_p} \cong \tau(w,\mu+\rho)$,
  integers~$N$ prime to~$p$ and $e \ge 0$, a Serre weight~$F$ occurring as Jordan--H\"older constituent of $W(\mu)$, and a Hecke eigenclass in
  \[ H^e(\go,F) \]
  with attached Galois representation~$\rho$.
\end{prop}

Note that the definition $\tau(w,\mu)$ in~\eqref{eq:tau(w,mu)} makes sense even if $\wmu$ is not good.

\need{What can one get if $p=2$? I think still works --- idea: take a prime $p' \equiv 1\pmod 4$ (exclude finitely many, in terms of~$K$).
Choose all $q_i$ as below such that moreover $p' | q_i^2 + 1$. Then look at $p'$-torsion instead of 2-torsion.}

\begin{rk}\label{rk:theo_evid_n=2m}
  This all generalises to $\GL_{2m}$, $m > 2$, \emph{assuming that the automorphic induction needed exists and satisfies the required local
  compatibility properties}. Let us just state the general result and say
  a few words about the changes in the proof.  Here one starts with $\mu\in X(T)_+$ with $\mu_i + \mu_{2m+1-i}$ being independent of~$i$.
  The tame inertial Galois representations obtained are all $\tau(w,\mu+\rho)$ where $w \in S_{2m}$ such that $w$
  respects the equivalence relation induced by $i \sim 2m+1-i$. For generic such~$\mu$ in the lowest alcove one thus obtains $2^m m!/(2m)!$
  of all predicted tame inertial Galois representations in weight $F(\mu)$~\eqref{prop:generic_pred}.
  
  The only part of the proof that does not immediately generalise is the construction of appropriate CM fields. The
  largest possible Galois group for a totally complex CM field~$K$ of degree $2m$ over~$\Q$ is the ``hyperoctahedral'' group $\Delta :=
  (\Z/2)^m \rtimes S_m$ with $S_m$ acting in the natural way. \(It is the largest since it is isomorphic to the centraliser of an element of
  cycle type $2^m$ in $S_{2m}$.  The subgroup of $w \in S_{2m}$ defined in the previous paragraph is the centraliser of $(1\; 2m)(2\;
  2m-1)\dots (m\; m+1)$.\) For each conjugacy class~$C$ of~$\Delta$ we need to be able to choose such a $K = K(C)$ which is unramified at~$p$
  and with $\Frob_{p} \in C$.  First one finds a totally real number field $K^+$ of degree~$m$ over~$\Q$, unramified at~$p$, whose
  Galois group is~$S_m$ and with $\Frob_{p} \in \overline C \subseteq S_m$. \(Use weak approximation on degree~$m$ polynomials over~$\Q$. In particular
  one may force that the Frobenius elements at auxiliary unramified primes are of all cycle types in their action on the roots. Finally
  an elementary lemma of Jordan says that no proper subgroup of a finite group contains an element of each conjugacy class.\) One chooses
  an auxiliary prime~$q$ split in~$K^+$ and uses weak approximation to find $\alpha \in (K^+)\s$ such that \(i\)~$\alpha$ is totally
  negative, \(ii\) $\ord_\qq(\alpha)$ is~0 for all but one prime $\qq | q$ for which it is~1, \(iii\) $p$ is unramified in $K =
  K^+(\sqrt\alpha)$, and \(iv\) the set of $\pp | p$ in~$K^+$ that split in~$K$ correspond to the conjugacy class~$C$. \(By analysing the
  conjugacy classes of~$\Delta$ one sees that the class of the Frobenius element in~$\Delta$ is determined precisely by its image~$\overline C$
  in the Galois group of~$K^+$---\ie\ the information of how many primes $\pp | p$ there are in~$K^+$ of each residue degree~$d$---plus, for
  each $d \ge 1$, the number of~$\pp$ of degree~$d$ that split in~$K$.\)
\end{rk}

\begin{proof}[Proof of prop.~\ref{prop:GL4_theo_evid}] By lemma~\ref{lm:CM_existence}, choose a quartic totally complex CM field $K/\Q$, unramified at~$p$,
with normal closure~$L$ and Galois group $\Delta := \Gal(L/\Q)$ dihedral of order~8. The conjugacy class of $\Frob_p$ will be irrelevant until the end of the
proof. Let $\mu(K)$ be the torsion subgroup of~$\oks$ and let $w(K)$ be its order; finally let $c \in \Delta$ denote
the complex conjugation (the unique central element of order~2).

We now want to make a careful choice of a Hecke character~$\chi$ over~$K$. For this recall (or notice):

\begin{sublm} Fix an ideal~$\f$ in~$\ok$.
  There is a bijection between Hecke characters~$\chi$ over~$K$ of conductor dividing~$\f$ and 3-tuples $(\epsilon,\epsilon_\f,\epsilon_\infty)$,
  where $\epsilon : I_K^\f \to \C\s$ \($I_K^\f$ being the ideals prime to~$\f$\), $\epsilon_\f : (\ok/\f)\s \to \C\s$, and
  $\epsilon_\infty : K_\infty\s \to \C\s$ continuous such that for all $x \in K\s$, $x$ prime to~$\f$,
  \begin{equation}
    \label{eq:eps_heckechar}
    \epsilon((x)) = \epsilon_\f(x)\epsilon_\infty(x).
  \end{equation}
  \(By weak approximation, $\epsilon_\f$ and $\epsilon_\infty$ are in fact determined by~$\epsilon$.\)
  The bijection is determined by demanding that
  \[ \chi(x) = \epsilon_\f(x_\f)^{-1}\epsilon_\infty(x_\infty)^{-1}\epsilon((x)) \]
  for all $x \in \A_K\s$ that are prime to~$\f$.
\end{sublm}

\need{ref?}

Fix for each $\sigma : K \to \C$ an integer $n_\sigma$ with the property that $n_\sigma + n_{\sigma c} = w$ for all~$\sigma$
(some~$w$).  These will be pinned down later.  Let $\epsilon_\infty : K_\infty\s \to \C\s$ be given by $\epsilon_\infty(x) =
|x|^{-3/2} \prod_{\sigma} \sigma(x)^{n_\sigma}$.  (Here, $|.|$ denotes the usual adelic norm on $K\s \backslash \A_K\s$
and on its subgroup $K_\infty\s$, and $\sigma(x)$ means $\sigma(x_v)$ for the unique place $v|\infty$ which is induced by~$\sigma$ on~$K$.)

\medskip

\emph{Claim:} $\epsilon_\infty(\oks)$ is finite, and hence contained in $\mu_{w(L)}(\C)$.

\medskip

Fix an embedding $j:L \to \C$ and for $\tau \in \Delta$ let $m_\tau = n_{j\tau|K}$. In particular, $m_{\tau} + m_{\tau c} = w$ for all~$\tau$.
It will suffice to show that $\prod_\tau \tau(-)^{m_\tau}$ kills $(\O_L\s)\subtf$. For,
$j\prod_\tau \tau(-)^{m_\tau} = \prod_{\sigma} \sigma(-)^{n_\sigma[L:K]}$ on $\O_K\s$.

By the unit theorem, $(\O_L\s)\subtf \INTO \Map(S_\infty,\R)_0$ as $\Delta$-module, where $S_\infty$ is the set of archimedean places of~$L$
and the subscript ``0'' denotes the subspace of $f : S_\infty \to \R$ with $\sum_v f(v) = 0$. As $\Delta$ acts transitively on $S_\infty$ with
stabiliser $\langle c\rangle$, $\Map(S_\infty,\R)_0 \cong \R[\Delta/\langle c\rangle]_0$ as $\R \Delta$-module, where the subscript ``0'' now refers
to $\sum_{\Delta/\langle c\rangle} \lambda_g g$ with $\sum_{\Delta/\langle c\rangle} \lambda_g = 0$ (\ie, the augmentation ideal). It will suffice to show that
for $\bar\nu \in \Delta/\langle c\rangle$, the action of $\sum_\Delta m_\tau \tau(-)$ on $\bar\nu \in \R[\Delta/\langle c\rangle]$ is independent of $\bar\nu$.
Indeed,
\begin{align*}
  \sum_{\tau \in \Delta} m_\tau \tau\bar\nu = \sum_{\tau \in \Delta/\langle c\rangle} (m_\tau \tau + (w-m_\tau)\tau c)\bar\nu
  = w \sum_{\tau\in \Delta/\langle c\rangle} \overline{\tau\nu} = w \sum_{\tau\in \Delta/\langle c\rangle} \overline{\tau}
\end{align*}
is independent of $\bar\nu$. This proves the claim.

Note that $L$ does not have any \emph{abelian} totally complex CM subfields, so the claim implies that $\epsilon_\infty(\oks) \subseteq \{\pm 1\}$.

Using the Cebotarev density theorem, choose distinct rational primes $q_i \nmid 2p$ ($1 \le i \le t$, any $t\ge 3$) that stay inert in~$K$
(equivalently, $\Frob_{q_i}\in \Delta$ has order 4). Denote by~$\qq_i$ the prime of~$K$ lying above~$q_i$.

If $\alpha \in \oks$ and $\alpha \equiv 1\pmod {\prod \qq_i}$ then in particular $\epsilon_\infty(\alpha) \equiv 1 \pmod {q_1}$ (in the
subring $\overline\Z \subseteq \C$).  But $\epsilon_\infty(\alpha) \in \{\pm 1\}$ by above and hence it is 1 (as $q_1$ odd). Therefore
$\epsilon_\infty|_{\oks}$ can be written as
\[ \epsilon_\infty|_{\oks} : \oks \to (\ok/{\textstyle\prod} \qq_i)\s \xrightarrow{\;\theta\,} \C\s, \]
where $\theta$ is not uniquely determined! 
Letting $A$ be the image of~$\oks$ in $(\ok/\/\prod\qq_i)\s$, we see that $\theta$ is determined by $\epsilon_\infty$ on~$A$
but nowhere else (the characters of $(\ok/\prod\qq_i)\s/A$ separate points).

Let $B_p$ be the $p$-Sylow subgroup of $(\ok/\prod \qq_i)\s$.
Observe that
\begin{equation*}
  \prod_{i=1}^t ((\ok/\qq_i)\s)^{q_i^2-1} \not\subseteq A\cdot B_p.
\end{equation*}
This is because the size of the 2-torsion on the left-hand side is exactly $2^t \ge 8$, whereas on the right it is bounded above
by~4 due to the unit theorem. Therefore we can assume, without loss of generality, that $\theta$ is non-trivial on
$\prod_{i=1}^t ((\ok/\qq_i)\s)^{q_i^2-1}$ while being of order prime to~$p$ (simply first extend the given map on~$A$
to $A\cdot B_p$ by making it trivial on~$B_p$).

Let $\f = \prod \qq_i$ and $\epsilon_\f = \theta^{-1}$. Writing $\epsilon_\f = \prod \epsilon_{\qq_i}$ (with the obvious meaning), we see
that $\epsilon_{\qq_i}$ has order not dividing $q_i^2 - 1$ for some~$i$. By permuting the~$\qq_i$, let us assume that this happens when $i =
1$ and set $\qq = \qq_1$, $q = q_1$.

By construction, $\epsilon_\f\epsilon_\infty$ is trivial on~$\oks$. Now $\epsilon$ can be defined by~\eqref{eq:eps_heckechar} on the
finite index subgroup of $I_K^\f$ generated by $(x)$ with $x \in K\s$ prime to~$\f$ and extended arbitrarily to $I_K^\f$. The above
sublemma yields a Hecke character~$\chi$ over~$K$; we record here some of its properties:
\begin{equation}\label{eq:chi_props}
  \begin{split}
    & \circ\ {\textstyle \chi_\infty(x) =  |x|^{3/2} \prod_{\sigma} \sigma(x)^{-n_\sigma}}, \\
    & \circ\ \text{$\chi$ has conductor dividing~$\textstyle\prod\qq_i$ (prime to $p$),} \\
    & \circ\ \text{$\chi_\qq|_{\O_{K_\qq}\s}$ has order dividing $q^4-1$ but not dividing $q^2-1$}, \\
    & \circ\ \text{$\chi(\textstyle\prod_{v\nmid \infty} \O_{K_v}\s)$ has order prime to $p$.}
  \end{split}
\end{equation}

By~\cite[\S{}III.6]{bib:AC} we can consider the automorphic induction $\AI_{K/\Q}(\chi)$, which is obtained in two stages: first inducing along the cyclic
extension $K/K^+$: $\Pi := \AI_{K/K^+} (\chi)$; then inducing along the cyclic
extension $K^+/\Q$: $\pi := \AI_{K^+/\Q}(\Pi)$.

Let us write $\mu+\rho = (a,b,c,d)$, so that $a > b > c > d$ and $a+d = b+c$.  Suppose that the $n_\sigma$ above chosen so that $\{n_\sigma
 \}_\sigma = \{a, b, c, d\}$ (note that there are only~8 possible choices as we demanded above that $n_\sigma + n_{\sigma c}$ is
independent of~$\sigma$).

\medskip

\emph{Claim:} $\pi$ is a cuspidal automorphic representation of $\GL_4(\A_\Q)$ of conductor prime to~$p$
to which prop.~\ref{prop:coho_autrep} applies with $(c_1,c_2,c_3,c_4) = (a,b,c,d)$. 

\medskip

Note the following facts about Arthur--Clozel's cyclic automorphic inductions: (i) they construct them using cyclic base change
(\cite{bib:AC}, thm.~III.6.2), (ii) global cyclic base change is compatible with local base change at all (finite or infinite) places (see
\cite{bib:AC}, thm.~III.5.1), (iii) local cyclic base change is compatible with restriction under the Local Langlands Correspondence (see
\cite{bib:AC}, p.~71 in the archimedean case and \cite{bib:HT}, thm.~VII.2.6 in the non-archimedean case).

As $\chi \ne \chi^c$ (look at either of the infinite components), $\Pi$ is cuspidal and is determined by
\begin{equation*}
  \BC_{K/K^+}(\Pi) \cong \chi \times \chi^c,
\end{equation*}
where $\BC_{K/K^+}$ denotes base change from $K^+$ to~$K$ (\cite{bib:AC}, bottom of p.~216). In particular, under the
Local Langlands Correspondence the infinite components of~$\Pi$ correspond to the representations sending
\begin{gather*}
  z \mapsto |z|^3 \dia ( z^{-a}\bar z^{-d}, z^{-d}\bar z^{-a}),\ \text{resp.} \\
  z \mapsto |z|^3 \dia ( z^{-b}\bar z^{-c}, z^{-c}\bar z^{-b}),
\end{gather*}
for $z \in W_\C = \C\s$. Repeating the argument shows that $\pi$ is cuspidal and
that under the Local Langlands Correspondence $\pi_\infty$ corresponds to a representation sending
\[ z \mapsto |z|^3 \dia ( z^{-a}\bar z^{-d}, z^{-d}\bar z^{-a}, z^{-b}\bar z^{-c}, z^{-c}\bar z^{-b}) \]
for $z \in W_\C$. As $a \ne d$ and $b \ne c$, by the classification of representations of $W_\R$ (see \eg~\cite{bib:Tate_backgr}, (2.2.2)),
this representation is the direct sum of
\begin{align*}
  z &\mapsto |z|^3
  \begin{pmatrix}
    z^{-a}\bar z^{-d} \\ & z^{-d}\bar z^{-a}
  \end{pmatrix} \\
  j & \mapsto \begin{pmatrix}
    & 1 \\ (-1)^{a+d}
  \end{pmatrix}
\end{align*}
and the same with $(a,d)$ replaced by $(b,c)$. This shows that $(c_1,c_2,c_3,c_4) = (a,b,c,d)$ in the notation of prop.~\ref{prop:coho_autrep}.

Let $S$ be the set of primes~$l$ that either ramify in~$K$ or divide a prime where $\chi$ is ramified.
For $l\not\in S$, $\pi_l$ is an unramified principal series which corresponds to
\begin{equation}\label{eq:sigma_l}
  \sigma_l := \bigoplus_{\lambda|l} \Ind^{W_l}_{W_\lambda} \chi_\lambda
\end{equation}
under the Local Langlands Correspondence (see \cite{bib:AC}, pp.~214f). In particular, the conductor~$N$ of~$\pi$ is prime to~$p$.
This establishes the claim. We get, for any~$e$ as in~\eqref{eq:range_of_e}, an $\ho$-equivariant injection
\begin{equation}
  \label{eq:clozel}
    (\pi^\infty)^\uo \INTO H^e(\go,r)
\end{equation}
with $r$ of highest weight $\mu = (a-3,b-2,c-1,d)$.

Let $\Sigma := \Ind_{W_K}^{W_\Q} \chi$, where $W_K$, $W_\Q$ denote the global Weil groups of~$K$ and~$\Q$. Since
\begin{equation*}
  \Sigma|_{I_q} \cong \bigoplus_{i \mmod 4}(\chi_\qq|_{I_\qq})^{q^i},
\end{equation*}
$\Sigma$ is irreducible (this uses~\eqref{eq:chi_props}). The previous paragraph shows that $\Sigma_v$ and~$\pi_v$ correspond to each other
under the unramified Langlands Correspondence for almost all places~$v$. Therefore we can use corollary~4.5 of~\cite{bib:Henniart} to see
that at \emph{all} finite places~$v$, the $L$-factors (and even the $\epsilon$-factors) of $\Sigma_v$ and~$\pi_v$ agree. In particular,
$\Sigma$ and $\pi$ are ramified at the same set of finite places (namely those finite primes at which the $L$-factor has degree less than~4;
for~$\pi$ this characterisation follows from~\cite[\S3]{bib:jacquet}). It follows that \emph{$S$ is precisely the set of prime divisors of~$N$}.

\def\mathbi#1{\textbf{\em #1}}

For $l\nmid N$, let $\mathbi{t}_l = \{t_{l,1},\dots,t_{l,4}\}$ denote the eigenvalues of $\sigma_l(\Frob_l)$.
It is known and easy to see that $[G(\zl)\bigg(\begin{smallmatrix}l\! \\[-5pt]
  & \ddots \\ && 1 \end{smallmatrix}\bigg)G(\zl)]$ with $i$ diagonal entries being equal to~$l$ has eigenvalue $s_i(\mathbi{t}_l) l^{i(4-i)/2}$ on
$\pi_l^{G(\zl)}$, where
$s_i$ denotes the $i$-th elementary symmetric function. Therefore, with the notation of~\S\ref{sub:hecke_alg},
\begin{equation*}
[\uo \bigg(\begin{smallmatrix}l\! \\[-5pt] & \ddots \\ && 1 \end{smallmatrix}\bigg)_{\!l} \uo] =
[\uo \bigg(\begin{smallmatrix}l\! \\[-5pt] & \ddots \\ && 1 \end{smallmatrix}\bigg) \widehat{\omega_N(l)} \uo]
\end{equation*}
has the same eigenvalue on $(\pi^\infty)^\uo$. Since this Hecke operator corresponds to $T_{l,i} \in \ho$ by lemma~\ref{lm:loc_glob_hecke},
\eqref{eq:clozel} yields a Hecke eigenclass in $H^e (\go,r)$ whose $T_{l,i}$-eigenvalue is $s_i(\mathbi{t}_l) \cdot l^{i(4-i)/2}$
($\forall l\nmid N$, $\forall i$). Equivalently, there is an eigenclass in $H^e (\go,r\tens_{\C,\iota} \qpb)$ with $T_{l,i}$-eigenvalue
$\iota (s_i(\mathbi{t}_l) \cdot l^{i(4-i)/2})$ ($\forall l\nmid N$, $\forall i$).

\medskip

\emph{Claim:} There is a Hecke eigenclass in $H^e(\go,F)$ with $T_{l,i}$-eigenvalue
\begin{equation*}
\overline{\iota (s_i(\mathbi{t}_l) \cdot
  l^{i(4-i)/2})}
\end{equation*}
($\forall l\nmid Np$, $\forall i$) for some Jordan--H\"older constituent~$F$ of $W(\mu)$ (as representation of $G(\fp)$).

\medskip

By~\cite{bib:Jan-reps}, II.2.9 and I.10.4, $r$ has a model~$M$ over $\Z_{(p)}$ (a representation of the reductive group scheme
$\GL_{n/\Z_{(p)}}$). Let $\overline M$ denote its reduction mod~$p$, a representation of $\GL_{n/\fp}$.

By~\cite{bib:Serre_coh_disc}, \S2.4, thm.~4, $\go$ is of type (WFL). In particular, the $\go$-module~$\Z$ has a resolution with finite free
$\go$-modules and, a fortiori, for any noetherian ring~$A$, $H^e(\go,P)$ is a finite $A$-module
whenever $P$ is a finite $A$-module with commuting $\go$-action, and $H^e(\go,-)$ commutes
with flat base extension (see~\cite{bib:Serre_coh_disc}, remark on p.~101).

Consider now only the Hecke operators $T_{l,i}$ with $l\nmid Np$. For any $\Z_{(p)}$-algebra~$R$, let $r_R := M \tens_{\Z_{(p)}} R$.
Note that $r_{\zpb}$ is a $\GL_n(\Z_{(p)})$-invariant $\zpb$-lattice in $r_{\qpb} \cong r \tens_{\C,\iota} \qpb$.
Since
\begin{equation}\label{eq:6}
  H^e(\go,r_{\qpb}) \cong H^e(\go,r_{\qp}) \tens_{\qp} \qpb
\end{equation}
(Hecke equivariantly) and this space is finite-dimensional over~$\qpb$, the simultaneous generalised eigenspaces for $T_{l,i}$ with $l\nmid Np$
can  be defined over some finite extension $E/\qp$. Thus
the above system of Hecke eigenvalues also occurs in $H^e(\go,r_E)$. Consider the following Hecke-equivariant map:
\begin{equation*}
  H^e(\go,r_{\oe})\subtf \INTO H^e(\go,r_E).
\end{equation*}
The image of the map is a lattice in $H^e(\go,r_E)$; this follows by looking at the long exact sequence
associated to $0 \to r_{\oe} \to r_E \to r_E/r_{\oe} \to 0$. By scaling the Hecke eigenclass in $H^e(\go,r_E)$, we may assume it lies in this
sublattice and has non-zero reduction in $H^e(\go,r_{\oe})\subtf \tens_{\oe} k_E$.

Consider the Hecke-equivariant map
\begin{equation*}
  H^e(\go,r_{\oe}) \tens_{\oe} k_E \onto H^e(\go,r_{\oe})\subtf \tens_{\oe} k_E.
\end{equation*}
Let $\HH$ denote the $k_E$-linear span of~$\hop$ in the endomorphism ring of the left-hand side. This is a finite-dimensional commutative
$k_E$-algebra and the above system of Hecke eigenvalues determines a maximal ideal $\mm$ in~$\HH$. Since this system of Hecke eigenvalues
occurs in an $\HH$-module~$V$ iff $V_{\mm} \ne 0$, it follows that it occurs also in 
$H^e(\go,r_{\oe})\tens_{\oe} k_E$. 

Finally, the long exact sequence associated to $0 \to r_{\oe} \to r_{\oe} \to r_{k_E} \to 0$ yields a Hecke-equivariant injection
\[ H^e(\go,r_{\oe}) \tens_{\oe} k_E\INTO H^e(\go,r_{k_E}) \INTO H^e(\go,r_{\fpb}). \]

Thus there is a Hecke eigenclass in $H^e(\go,r_{\fpb})$ with $T_{l,i}$-eigenvalue $\overline{\iota (s_i(\mathbi{t}_l) \cdot l^{i(4-i)/2})}$
($\forall l\nmid N$, $\forall i$).

By~\cite[I.2.11(10)]{bib:Jan-reps}, the formal characters of $M$, $r$ and~$\overline M$ are equal (under the natural
identifications).  Since the formal characters of both $r$ and $W(\mu)$ are given by the Weyl character formula for the
highest weight~$\mu$ \cite[II.5.10]{bib:Jan-reps}, the $G$-modules $r_{\fpb} = \overline M \tens_{\fp} \fpb$ and $W(\mu)$
have the same formal character so that they are isomorphic up to semisimplification (as $G$-modules, and hence as
$G(\fp)$-modules). By devissage the same system of Hecke eigenvalues obtained in $H^e(\go,r_{\fpb})$ also occurs in
$H^e(\go,F)$ for some Jordan--H\"older constituent~$F$ of $W(\mu)$. This establishes the claim.

\need{I suppose the actual reduction can be chosen to be W (or its dual) though; somewhere in Jantzen's papers?}

The Hecke character $\eta := \chi^{-1}|.|^{3/2}$ is algebraic with algebraic infinity type $\eta_\infty(x) = \prod \sigma(x)^{n_\sigma}$.
Recall the definition of the associated $p$-adic Galois character $\eta^{(p)}$ 
(using the global Artin map; see \eg~\cite{bib:HT}, pp.~20f):
\begin{align}\notag
  \eta^{(p)} : G_K^{ab} \cong \overline{K\s K_{\infty,+}\s}\backslash \A_K\s &\to \qpb\s \\ \label{eq:etap}
  x &\mapsto \iota\eta(x^\infty)\prod_{\tau:K\to\qpb} \tau(x_p)^{n_{\iota^{-1}(\tau)}}.
\end{align}
Here, the convention is that $\tau(x_p)$ means $\tau(x_v)$ for the unique $v|p$ induced by~$\tau$ on~$K$. In particular,
$\eta^{(p)}|_{G_{K_\lambda}} = \iota(\chi^{-1}_\lambda |.|_\lambda^{3/2})$ under the local Artin map for all $\lambda\nmid p$.

\medskip

\emph{Claim:} The Galois representation
\begin{equation*}
  \rho := \Ind_{G_K}^{G_\Q} \bigl(\overline {\eta^{(p)}}\bigr)
\end{equation*}
is attached to the eigenclass in $H^e(\go,F)$ constructed above.
It it continuous, irreducible, odd and its ramification outside~$p$ occurs precisely at all $l|N$.

\medskip

Clearly, $\rho$ is continuous. By Mackey's formula, using the local Artin map, for any prime $l \ne p$,
\begin{equation*}
  \rho|_{I_l} \cong \bigoplus_{\lambda | l} \bigoplus_{g \in G_\lambda I_l\backslash G_l} \Ind_{I_\lambda^g}^{I_l}
     (\overline{\iota\chi}^{-1}_\lambda|_{I_\lambda}^g).
\end{equation*}
By Frobenius reciprocity, $I_l$ acts trivially on the direct
summand corresponding to the index $(\lambda,g)$ if and only if $I_\lambda = I_l$ and $\overline{\iota\chi}|_{I_\lambda} = 1$. Thus the claim about
ramification outside~$p$ follows from~\eqref{eq:chi_props} and the fact that $S$ is the set of primes dividing~$N$. Specialising now to $l = q$ we even get:
\[ \rho|_{I_q} \cong \bigoplus_{i \mmod 4}(\overline {\iota\chi}_\qq|_{I_\qq})^{-q^i}. \]
Note that even the order of $\overline{\iota\chi}_\qq|_{I_\qq}$ does not divide $q^2-1$, by~\eqref{eq:chi_props}. Hence $\rho|_{G_q}$ is
irreducible; a fortiori, so is~$\rho$.

For $l\nmid Np$, we know that
\[ \rho|_{G_l} \cong \bigoplus_{\lambda|l} 
  \Ind_{G_\lambda}^{G_l} (\overline{\eta^{(p)}}|_{G_\lambda}) \]
is unramified. Using an explicit basis, we see that $\rho(\Frob_l)$ has characteristic polynomial
\begin{equation*}
  X^{[K_\lambda:\ql]} - \overline{\eta^{(p)}(\Frob_\lambda)}
\end{equation*}
on the $\lambda$-direct summand. A similar consideration applied to~$\sigma_l$ in~\eqref{eq:sigma_l} shows that the eigenvalues of $\rho(\Frob_l)$ are
$\overline{\iota(t_{l,j}^{-1} l^{-3/2})}$ (recall that the $t_{l,j}$ are the eigenvalues of $\sigma_l(\Frob_l)$ and that
$\overline{\eta^{(p)}}|_{G_\lambda} = \overline{\iota(\chi^{-1}_\lambda|.|_\lambda^{3/2})}$).

By the following simple computation, and the fact that $S$ is the set of prime divisors of~$N$,
 we see that $\rho$~is attached to the eigenclass constructed above: for all $l\nmid Np$,
\begin{align*}
  \sum_{i=0}^4 (-1)^i l^{i(i-1)/2} \overline{s_i(\iota \mathbi{t}_l) \cdot \iota l^{i(4-i)/2}} X^i = \prod_{j=1}^4 (1- \overline{\iota
    (t_{l,j} l^{3/2}) } \cdot X).
\end{align*}

Finally, note that
\[ \rho|_{G_\R} \cong \Bigl(\Ind_{G_\C}^{G_\R} (1)\Bigr)^{\oplus 2}, \]
which has eigenvalues $1$ and $-1$ twice each on complex conjugation. Thus $\rho$ is odd and the claim is established.

To determine $\rho|_{I_p}$, note that
\begin{equation*}
  \rho|_{I_p} \cong \bigoplus_{\pp|p} \bigoplus_{i \mmod f_\pp} \overline {\eta^{(p)}}|_{I_\pp}^{p^i}
\end{equation*}
where $f_\pp$ is the inertial degree. Also, as $\chi$ is unramified at all $\pp|p$ we get from~\eqref{eq:etap},
\begin{equation*}
  \overline{\eta^{(p)}} : x_p \mapsto \prod_{\tau: K \to \qpb} \overline{\tau(x_p)}^{\, n_{\iota^{-1}(\tau)}}
\end{equation*}
for $x_p \in \prod_{\pp|p} \O_{K_\pp}\s$. Fix for each $\pp | p$ an embedding $\tau_\pp : K\to \qpb$ which induces the place~$\pp$ on~$K$ and
denote by $\phi : \qpnr \to \qpnr$ the arithmetic Frobenius. Recall that the composite $I_{K_\pp} \onto \O_{K_\pp}\s \to k_\pp\s$, where the
first map is induced by local class field theory and the second is $x_\pp \mapsto \bar x_\pp$, is the fundamental tame character
$\theta_\pp$ of level $f_\pp$ (see~\cite{bib:Serre_proprietes}, prop.~3 with $L = K_d$ in Serre's notation; notice the different sign
convention for the local Artin map). We get
\begin{equation*}
  \overline{\eta^{(p)}} : x_\pp \mapsto \overline{\tau_\pp}\theta_\pp^{\sum_{i \mmod f_\pp} p^i n_{\iota^{-1}(\phi^i \tau_\pp)}}
\end{equation*}
for $x_\pp \in \O_{K_\pp}\s$.

Now we let the $n_\sigma$ vary through the 8 allowed permutations of $\{a, b, c, d\}$ (recall that $n_\sigma + n_{\sigma c}$
has to be independent of~$\sigma$). To see which $\rho|_{I_p}$ are obtained for a fixed conjugacy class of $\Frob_p\in
\Delta$, it thus only matters how complex conjugation acts on the set of~$\pp|p$, and what $f_\pp$ is in each case.  With the
notation of lemma~\ref{lm:CM_existence}(iv) we obtain $\tau(w,\mu+\rho)$ where $w$ can equal 1 in case (a), $(1\;4)(2\;3)$ in
case (b), either of $(1\;2)(3\;4)$, $(1\;3)(2\;4)$ in case (c), either of $(1\;4)$, $(2\;3)$ in case (d), and either of
$(1\;2\;4\;3)$, $(1\;3\;4\;2)$ in case (e). This completes the proof of prop.~\ref{prop:GL4_theo_evid}.
\end{proof}

Suppose that $F \cong F(\mu)$ is a regular Serre weight and that $\tau : I_p \to \GL_n(\fpb)$ is tame and can be extended to~$G_p$.
Suppose that $F \in W^?(\tau)$.

\begin{df}\
  \begin{enumerate}
  \item We say that an irreducible, odd Galois representation $\rho : G_\Q \to \GL_n(\fpb)$ \emph{provides evidence} for $(F,\tau)$ if
    $\rho|_{I_p} \cong \tau$ and $F \in W(\rho)$.
  \item 
    Suppose that none of the Serre weights in $\JH(W(\mu))$ lie on an alcove boundary~\eqref{eq:hyperplanes}; in particular they are all regular.
    We say that an irreducible, odd Galois representation $\rho : G_\Q \to \GL_n(\fpb)$ \emph{provides weak evidence} for $(F,\tau)$ if
    $\rho|_{I_p} \cong \tau$, $W(\rho) \cap \JH(W(\mu)) \ne \varnothing$, and $W^?(\tau) \cap \JH(W(\mu)) = \{ F \}$.
  \end{enumerate}
\end{df}

By $\JH(W(\mu))$ in (ii) we mean the Jordan--H\"older constituents of $W(\mu)$ as $\gln$-representation. Let us denote by~$C$
the alcove containing~$\mu$.  At least for $\mu$ \sd\ in~$C$ it is clear that all constituents of $W(\mu)$ besides~$F$ lie in
alcoves strictly below~$C$.  (This is because of prop.~\ref{prop:stronglinking} and since by the claim in the proof of
prop.~\ref{prop:gee_evidence}, all $\GL_4(\fp)$-constituents of $F(\lambda)$ for $\lambda$ \sd\ in alcove~$C_{0'}$
or~$C_{0''}$ lie in alcove~$C_0$. For general~$n$, the statement is easily seen to be true at least for the finer partial
order on alcoves induced by a certain function $d$ sending alcoves to the integers \cite[II.6.6]{bib:Jan-reps}.) Thus if the
conjecture correctly predicts the weights of~$\tau$ in all alcoves strictly below~$C$, $\rho$ provides actual evidence for
$(F,\tau)$.

\begin{thm}\label{thm:gl4_th_evid} 
  Suppose that $F \cong F(\mu)$ with $\mu_1+\mu_4 = \mu_2+\mu_3$ lies sufficiently deep in one of the four possible restricted alcoves.

  If $F \in C_0$ then for 8~of the 24~tame inertial representations~$\tau$ with $F \in W^?(\tau)$,
  prop.~\ref{prop:GL4_theo_evid} provides evidence for $(F,\tau)$.
 
  If $F \in C_1$ \(resp., $C_4$, $C_5$\) then for 8~of the 48 \(resp., 120, 192\) tame inertial representations~$\tau$ with
  $F \in W^?(\tau)$, prop.~\ref{prop:GL4_theo_evid} provides weak evidence for $(F,\tau)$.
\end{thm}

\begin{proof}
  Note that the Galois representations~$\rho$ obtained from prop.~\ref{prop:GL4_theo_evid} for the given $\mu \in X(T)_+$
  satisfy $F \in W^?(\rho|_{I_p})$: as $\rho|_{I_p} \cong \tau(w,\mu+\rho)$ for some $w \in W$ we may apply
  prop.~\ref{prop:generic_pred} with $\lambda = \lambda' = \mu$.  Also, by prop.~\ref{prop:generic_pred} and~\eqref{eq:5} the
  set $\{ \tau : F \in W^?(\tau) \}$ has cardinality $\# W \cdot \{ C' : \text{$C'$ dominant, $C' \uparrow C$} \}$, where $C$
  is the alcove containing~$\mu$.
  
  It remains to verify that $W^?(\tau) \cap \JH(W(\mu)) = \{ F \}$ where $\tau \cong \tau(w,\mu+\rho)$ (for one of the 8~values of $w
  \in W$ as in prop.~\ref{prop:GL4_theo_evid}). Suppose thus that $F' \in W^?(\tau) \cap \JH(W(\mu))$.  Then there exists a constituent
  $F(\lambda)$ of $W(\mu)$ as $G$-module ($\lambda \in X(T)_+$) such that $F' \in \JH(F(\lambda))$.  From the proof
  of prop.~\ref{prop:gee_evidence}, in particular from \eqref{eq:3}, \eqref{eq:4}, it follows that there exist $\mu' \in X(T)_+$ and $w' \in W$
  such that $\mu' \uparrow \lambda \uparrow \mu$ and $\tau \cong \tau(w',\mu'+\rho)$. But~\eqref{eq:5}
  implies that $\mu' = \lambda = \mu$, so that $F' \cong F(\lambda) \cong F(\mu) \cong F$, as required.
\end{proof}

\section{Weights in Serre's Conjecture for Hilbert modular forms}\label{sec:hilb_mod}

In~\cite{bib:BDJ}, Buzzard, Diamond and Jarvis formulate a Serre-type conjecture for Hilbert modular
forms. Theorem~\ref{thm:relation_to_bdj} below will show that their weight conjecture in the tame case is related, via an operation on the Serre weights
analogous to~$\RR$ in \S\ref{sub:operator_R}, to the decompositions of irreducible representations of $\GL_2(\F)$ over~$\qpb$ when reduced mod~$p$
(where $\F$ is a finite field of characteristic~$p$). They work with a totally real number field~$K$ that is \emph{unramified at~$p$}.

Suppose that $\rho : G_K \to \GL_2(\fpb)$ is an irreducible, totally odd representation. A \emph{Serre weight} in this context is an
isomorphism class of irreducible representations of $\GL_2(\ok/p) \cong \prod_{\pp|p} \GL_2(k_\pp)$ over~$\fpb$ where $k_\pp$ is the residue
field of~$K$ at~$\pp$. Any such representation is isomorphic to $\bigotimes_{\pp|p} W_\pp$ with $W_\pp$ an irreducible representation of
$\GL_2(k_\pp)$. The weight conjecture in~\cite{bib:BDJ} defines the $W_\pp$ independently of one another in terms of $\rho|_{I_\pp}$.
Let us therefore restrict our attention to a single prime~$\pp | p$.

Fix an embedding $\overline K\to \qpb$ inducing the place~$\pp$ on~$K$. Let $I_\pp := \Gal(\qpb/K_\pp^{nr})$ denote the corresponding
inertia subgroup.  Let $k_\pp' \subseteq \fpb$ be the quadratic extension of $k_\pp$. Let $f := [k_\pp:\fp]$.
There are canonical fundamental tame
characters $\psi : I_\pp \onto k_\pp\s$ of level~$f$ and $\psi' : I_\pp \onto (k_\pp')\s$ of level~$2f$.

For $i \in \Z/f$, let $\lambda_i$ be the $p^i$-th power of $k_\pp\s \xrightarrow{\subseteq} \fpb\s$
and for $i\in \Z/2f$ let $\lambda_{i'}$ be the $p^i$-th power of $(k_\pp')\s \xrightarrow{\subseteq} \fpb\s$.
Also let $\psi_i := \lambda_i \circ \psi$ for $i \in \Z/f$ and $\psi_{i'} := \lambda_{i'} \circ \psi'$ for $i \in \Z/2f$.

To describe the set $W_{\Ser,\pp}$ of isomorphism classes of irreducible representations of $\GL_2(k_\pp)$ over~$\fpb$ (\emph{Serre weights at~$\pp$}),
note first that theorem~\ref{thm:mod_reps} shows that
\[ W_{\Ser,\pp} = \{ F(a,b): 0 \le a-b \le p^f-1, 0 \le b < p^f-1 \}. \]
If we write $a-b = \sum_{i=0}^{f-1} m_i p^i$, $b = \sum_{i=0}^{f-1} b_i p^i$ with $0 \le m_i,\ b_i \le p-1$ then
by the Steinberg tensor product theorem~\eqref{thm:Steinberg}, 
\[ F(a,b) \cong \bigotimes_{i=0}^{f-1} F(b_i+m_i,b_i)^{(p^i)}. \]
Since $F(b_i+m_i,b_i) \cong \Sym^{m_i} \fpb^2 \tens \Det^{b_i}$ (see \S\ref{sub:case-gl_n}),
\[ F(a,b) \cong \bigotimes_{i=0}^{f-1} (\Sym^{m_i} k_\pp^2 \tens \Det^{b_i}) \tens_{k_\pp,\phi^i} \fpb \]
where $\phi:k_\pp \to k_\pp$ is the $p$-power Frobenius element. This representation will also be denoted by $F_{\vec m,\vec b}$.

Suppose that $\rho$ is \emph{tame at~$\pp$}. Then we can write $\rho|_{I_\pp} \cong \chi_1 \oplus \chi_2$.  We say that $\rho|_{I_\pp}$ is
of \emph{niveau~1} if $\chi_i^{p^f-1} = 1$ ($i = 1$, 2) and of \emph{niveau~2} otherwise.  Let us recall the definition of the conjectured
set of weights $W_\pp^?(\rho)$ from~\cite{bib:BDJ} in the tame case. If $\rho|_{I_\pp}$ is of niveau~1, $W_\pp^?(\rho)$ consists of all
$F_{\vec m,\vec b}$ such that
\begin{equation}\label{eq:BDJ_niv1}
  \rho|_{I_\pp} \sim
  \begin{pmatrix}
    \prod_J \psi_i^{m_i+1} \\ & \prod_{J^c} \psi_i^{m_i+1}
  \end{pmatrix} \prod \psi_i^{b_i}  
\end{equation}
for some $J \subseteq \Z/f$. If $\rho|_{I_\pp}$ is of niveau~2, $W_\pp^?(\rho)$ consists of all $F_{\vec m,\vec b}$ such that
\[ \rho|_{I_\pp} \sim
\begin{pmatrix}
  \prod_J \psi_{i'}^{m_i+1} \\ & \prod_{J^c} \psi_{i'}^{m_i+1}
\end{pmatrix}\prod \psi_i^{b_i} \]
for some $J \subseteq \Z/2f$ projecting \emph{bijectively} onto~$\Z/f$ (under the natural map). Here we are abusing notation in that
the indices of~$m$ and~$b$ should be taken ``mod~$f$''.

Associated to each $\rho|_{I_\pp}$ define a representation $V_\pp(\rho|_{I_\pp})$ of $\GL_2(k_\pp)$ over~$\qpb$.  The Teichm\"uller lift will
again be denoted by~$\widetilde\ $. For characters $\chi_i : k_\pp\s \to \qpb\s$, $I(\chi_1,\chi_2)$ will denote the induction from the
Borel subgroup of upper-triangular matrices to $\GL_2(k_\pp)$ of $\chi_1 \tens \chi_2$, whereas for a character $\chi : (k_\pp')\s \to
\qpb\s$ which does not factor through the norm $(k_\pp')\s \to k_\pp\s$, the cuspidal representation $\kappa(\chi)$ of $\GL_2(k_\pp)$ was
defined in def.~\ref{df:cusp}.

\begin{df}\label{df:v_p(rho)} \
  \begin{enumerate}
  \item If $\rho|_{I_\pp} \sim
    \begin{pmatrix}
      \prod \psi_i^{c_i} \\ & \prod \psi_i^{c'_i}
    \end{pmatrix}$ is of niveau~1,
    \[ V_\pp(\rho|_{I_\pp}) := I(\pprod \widetilde\lambda_i^{c_i}, \pprod \widetilde\lambda_i^{c_i'}).\]

  \item If $\rho|_{I_\pp} \sim
    \begin{pmatrix}
      \prod \psi_{i'}^{\gamma_i} \\ & \prod \psi_{i'}^{p^f \gamma_i}
    \end{pmatrix}$ is of niveau~2,
    \[ V_\pp(\rho|_{I_\pp}) := \kappa(\pprod \widetilde\lambda_{i'}^{\gamma_i}). \]
  \end{enumerate}
\end{df}

Note that in (ii), $i$ runs through $\Z/2f$.  In particular, $V_\pp(\rho|_{I_\pp} \tens (\chi \circ \psi)) \cong V_\pp(\rho|_{I_\pp}) \tens
\widetilde\chi$ for any character $\chi: k_\pp\s \to \fpb\s$. Also note that this is the same as $V(\rho|_{I_\pp})$ in def.~\ref{df:V(tau)} (in
light of \S\ref{sub:case-gl_n-char0}). We prefer to use the above description here as we can then use the decomposition formulae derived
in~\cite{bib:Di}. (To identify his $\Theta(\chi)$ with $\kappa(\chi)$, compare their characters at elements whose characteristic polynomial
is irreducible using \cite[7.3]{bib:DL}.)

A \emph{regular Serre weight at~$\pp$} is any Serre weight $F_{\vec m,\vec b}$ with $0 \le m_i < p-1$ for all~$i$. The set
of regular Serre weights at~$\pp$ is denoted by $W_{\reg,\pp}$.
Define $\RR_\pp : W_{\reg,\pp} \to W_{\reg,\pp}$ by
\[ \RR_\pp(F(a,b)) = F(b+(p-2){\textstyle\sum}_{i=0}^{f-1} p^i,a), \]
(compare this with $\RR$ in \S\ref{sub:operator_R}).

\begin{thm}\label{thm:relation_to_bdj}
  Suppose that $\rho : G_K \to \GL_2(\fpb)$ is irreducible, totally odd, and tame at~$\pp$.
  \begin{enumerate}
  \item $W_\pp^?(\rho) \cap W_{\reg,\pp} = \RR_\pp(\JH(\overline{V_\pp(\rho|_{I_\pp})})\cap W_{\reg,\pp})$.
  \item There is a \emph{multi-valued} function $\RR_{\ext,\pp} : W_{\Ser,\pp} \to W_{\Ser,\pp}$ that extends $\RR_\pp$
    such that
    \[ W_\pp^?(\rho) = \RR_{\ext,\pp}(\JH(\overline{V_\pp(\rho|_{I_\pp})})). \]
  \end{enumerate}
\end{thm}

\label{RR_ext} The following definition of $\RR_{\ext,\pp}$ will be shown to satisfy part (ii) of the theorem. Suppose that $F \cong F(a,b)$ with $0 \le
a-b \le p^f-1$. We can write $a-b = \sum_{i=0}^{f-1} m_i p^i$ for some $0 \le m_i \le p-1$. Define a collection $\SS(F)$ of subsets of
$\Z/f$ by: $S \in \SS(F)$ if and only if for all $s \in S$, $m_s \ne 0$ and there is an~$i$ such that $m_i = p-1$, $m_{i+1} = \dots =
m_{s-1} = p-2$ and $\{i,i+1,\dots,s-1\} \cap S = \varnothing$. Then $\RR_{\ext,\pp}(F)$ is defined to be
\[ \Bigl\{F(a',b') : \exists S \in \SS(F),\ a' \equiv b-\textstyle\sum\limits_{i\not\in S} p^i,\ b' \equiv a-\sum\limits_{i\in S} p^i \!\!\pmod{p^f-1}\Bigr\}. \]

In particular, for this choice of $\RR_{\ext,\pp}$, if $F$ is a regular Serre weight then $\SS(F) = \{\varnothing\}$, so 
$\RR_\pp(F) = \RR_{\ext,\pp}(F)$ unless $F$ is a twist of $F((p-2)\sum p^i,0)$ in which case $\RR_{\ext,\pp}(F)$ contains one more weight.

The proof will require several lemmas, proved below.

\begin{lm}\label{lm:relate_decomp_bdj}
  Suppose that $0 \le m_i \le p-1$ \($i \in \Z/f$\).

  \textup{(i)} Suppose that $\rho|_{I_\pp}$ is of niveau~1. Then $F_{\vec m,\vec b}$ is a constituent of
  $\overline{V_\pp(\rho|_{I_\pp})}$ if and only if
  \[ \rho|_{I_\pp} \sim
  \begin{pmatrix}
    \prod_{J^c} \psi_i^{p-1-m_i} \\ & \prod_{J} \psi_i^{p-1-m_i}
  \end{pmatrix} \prod \psi_i^{m_i+b_i} \]
  for some $J \subseteq \Z/f$.
 
  \textup{(ii)} Suppose that $\rho|_{I_\pp}$ is of niveau~2. Then  $F_{\vec m,\vec b}$ is a constituent of
  $\overline{V_\pp(\rho|_{I_\pp})}$ if and only if
  \[ \rho|_{I_\pp} \sim
  \begin{pmatrix}
    \prod_{J^c} \psi_{i'}^{p-1-m_i} \\ & \prod_{J} \psi_{i'}^{p-1-m_i}
  \end{pmatrix} \prod \psi_i^{m_i+b_i} \]
  for some $J \subseteq \Z/2f$  projecting bijectively onto~$\Z/f$.
\end{lm}

Let us explain the idea of the proof of the theorem. The above lemma is the key tool that lets us relate the conjectured weight
set $W_\pp^?(\rho)$ with the decomposition of $\overline{V_\pp(\rho|_{I_\pp})}$. This works perfectly for regular Serre weights.
In general the problem is that the number of  constituents of $\overline{V_\pp(\rho|_{I_\pp})}$ might be a lot smaller than $\# W_\pp^?(\rho)$.
This suggests looking for a multi-valued function extending~$\RR$. In view of lemma~\ref{lm:relate_decomp_bdj}, we have to find rules to convert
an expression of the form
\begin{equation*}
  \rho|_{I_\pp} \sim
  \begin{pmatrix}
    \prod_J \psi_i^{\alpha(i)} \\ & \prod_{J^c} \psi_i^{\alpha(i)}
  \end{pmatrix}\chi
\end{equation*}
for some $J \subseteq \Z/f$, $0\le \alpha(i) \le p-1$ and some character~$\chi$ into an expression of the form
\begin{equation*}
  \rho|_{I_\pp} \sim
  \begin{pmatrix}
    \prod_L \psi_i^{\beta(i)} \\ & \prod_{L^c} \psi_i^{\beta(i)}
  \end{pmatrix}\chi'
\end{equation*}
for some $L \subseteq \Z/f$, $1\le \beta(i) \le p$ and some character~$\chi'$ in such a way that the map
\begin{equation*}
  (\alpha,\chi) \mapsto (\beta,\chi')
\end{equation*}
\emph{does not depend on~$J$ and works equally well for the analogous expressions of niveau~2}. The theorem shows, roughly speaking,
that there are enough such rules to explain all of $W_\pp^?(\rho)$.

To make this principle concrete, consider $f = 3$ and $\vec\alpha = (0,1,p-1)$ and $\chi = 1$. It is very instructive to check that
there are such rules giving rise to the following pairs $(\beta,\chi')$:
\begin{equation*}
  ((p,p,p-2),1),\ ((p,2,p-1),\psi_1^{-1}),\ ((p,p,p),\psi_2^{-1}).
\end{equation*}
For example, here are two instances of the second rule:
\begin{align*}
  \begin{pmatrix}
    \psi_1 \\ & \psi_2^{p-1}
  \end{pmatrix} &\sim
  \begin{pmatrix}
    \psi_1^2 \\ & \psi_0^p\psi_2^{p-1}
  \end{pmatrix}\psi_1^{-1} \\
\noalign{\noindent and}
  \begin{pmatrix}
    \psi_{1'}\psi_{2'}^{p-1} \\ & \psi_{4'}\psi_{5'}^{p-1}
  \end{pmatrix} &\sim
  \begin{pmatrix}
    \psi_{3'}^p\psi_{1'}^2\psi_{2'}^{p-1} \\ & \psi_{0'}^p\psi_{4'}^2\psi_{5'}^{p-1}
  \end{pmatrix}\psi_1^{-1}.
\end{align*}
In the end, these rules consist of multiple uses of the identity
\begin{equation*}
  \psi_{j+1} = \psi_{i}^p\psi_{i+1}^{p-1}\cdots \psi_{j}^{p-1}
\end{equation*}
when $\alpha(i) = \cdots = \alpha(j) = 0$ ($\alpha(i) = 1$ is allowed if $\psi_i$ is itself to be expanded in this manner!).
Of course this works equally well for $\psi_{(j+1)'}$. To compare with the formalism below, let us indicate in each case
the corresponding choice of \I:
\begin{equation*}
  \underline{0,}_- \underline{1,}_- p-1, \quad \underline{0,}_+ 1,\ p-1, \quad \underline{0,}_- \underline{1,}_+ p-1.
\end{equation*}
Note that the last of these is not covered by the $\RR_{\ext,\pp}$ we defined above. In fact, it is not hard to see that
axiom~A4 below could be weakened to:
\begin{enumerate}
\item[A4$'$] If an \I-interval is positive, its successor does not lie in any \I-interval.
\end{enumerate}
This corresponds to removing the condition $m_s \ne 0$ in the definition of $\RR_{\ext,\pp}$ above. If we denote this modified
version of $\RR_{\ext,\pp}$ by $\RR_{\ext,\pp}'$ then it is clear that \emph{any} multi-valued function between
$\RR_{\ext,\pp}$ and $\RR_{\ext,\pp}'$ (\ie, such that there is a containment pointwise) satisfies thm.~\ref{thm:relation_to_bdj}(ii).

\separator

For our purposes, an \emph{interval} in~$\Z/f$ is any ``stretch'' of numbers
$\ll i,j\rr = \{i,i+1,\dots,j\}$ in~$\Z/f$. The start and end points are remembered so that, for example, $\ll 0,p-1\rr \ne \ll 1,0\rr$ even though
the underlying sets are the same. The \emph{successor} of an interval $\ll i, j\rr$ is $j+1$.

Suppose that $\alpha$ is a function $\Z/f \to \{0,1,\dots,p-1\}$, and suppose that $\I$ a collection of disjoint
intervals~$I$ in~$\Z/f$, each labelled with a sign (thought of as pertaining to the entry following that interval). Define the set $\L_{[0,p-1]}$ 
to consist of all $(\alpha,\I)$ which satisfy the following rules:

\begin{enumerate}
\item[A1] For each interval $I\in\I$, $\alpha(I) \subseteq \{0,1\}$.
\item[A2] If $i \in \bigcup\I$ then $\alpha(i) = 1$ if and only if $i$ is start point of an \I-interval and $i-1 \in \bigcup \I$.
\item[A3] If $i \not\in \bigcup \I$ and $\alpha(i) = 0$, then $i-1 \in \bigcup \I$.
\item[A4] If an \I-interval is positive, its successor does not lie in any \I-interval and has $\alpha$-value in $[0,p-2]$.
\item[A5] If an \I-interval is negative, its successor lies in another \I-interval or has $\alpha$-value in $[2,p-1]$.
\end{enumerate}

Note that every function $\alpha : \Z/f \to \{0,1,\dots,p-1\}$ can be equipped with intervals and signs satisfying these rules.

Similarly, suppose that $\beta$ is a function $\Z/f \to \{1,2,\dots,p\}$, and suppose that $\I$ a collection of disjoint
intervals in~$\Z/f$, each labelled with a sign (thought of as pertaining to the entry following that interval). Define the set $\L_{[1,p]}$ 
to consist of all $(\beta,\I)$ which satisfy the following rules:

\begin{enumerate}
\item[B1] For each interval $I\in\I$, $\beta(I) \subseteq \{p-1,p\}$.
\item[B2] The set of start points of \I-intervals is $\beta^{-1}(p)$.
\item[B3] If an \I-interval is positive, its successor does not lie in any \I-interval and has $\beta$-value in $[1,p-1]$.
\item[B4] If an \I-interval is negative, its successor lies in another \I-interval or has $\beta$-value in $[1,p-2]$.
\end{enumerate}

Note that every function $\beta : \Z/f \to \{1,2,\dots,p\}$ can be equipped with intervals and signs satisfying these rules.

To define a map $\phi : \L_{[0,p-1]} \to \L_{[1,p]}$, represent~$\alpha$ as the string of numbers $\alpha(0)$,
$\alpha(1)$, \dots, $\alpha(f-1)$; underline each \I-interval and put the corresponding sign just after the last entry of the interval. In this way
the function~$\phi$ has the following effect on each interval and its successor (it leaves all other entries unchanged):
\begin{gather*}
  \underline{(1), 0, \dots, 0,}_{\pm} a,\dots \mapsto \underline{p,p-1,\dots, p-1,}_\pm a\pm 1,\dots \\
  \underline{\dots,0,0,}_- \underline{1,0,\dots} \mapsto \underline{\dots,p-1,p-1,}_- \underline{p,p-1,\dots}
\end{gather*}

\begin{lm}\label{lm:combi_bij}
  The map~$\phi$ is well defined and in fact a bijection.
\end{lm}

\begin{lm}\label{lm:succ_of_ints}
  Suppose that $\alpha:\Z/f \to \{0,1,\dots,p-1\}$. Then the following are equivalent for a subset $S\subseteq \Z/f$:
  \begin{enumerate}
  \item $S \in \SS(F_{\vec p-\vec 1-\vec \alpha,\vec x})$ for some~$\vec x$.
  \item $S \in \SS(F_{\vec p-\vec 1-\vec \alpha,\vec x})$ for all~$\vec x$.
  \item $S$ is the set of successors of positive intervals in~$\I$ for some~$\I$ with $(\alpha,\I) \in \L_{[0,p-1]}$.
  \end{enumerate}
\end{lm}

\begin{proof}[Proof of the theorem] 

(i) This is a straightforward application of lemma~\ref{lm:relate_decomp_bdj}. First consider the niveau~1 case.
Suppose $F \in W_\pp^?(\rho)$ and $F$ regular. By twisting,
we can assume without loss of generality that $F = F_{\vec b-\vec 1,\vec 0}$ ($1 \le b_i \le p-1$) and 
\begin{equation*}
  \rho|_{I_\pp} \sim
  \begin{pmatrix}
    \prod_J \psi_i^{b_i} \\ & \prod_{J^c} \psi_i^{b_i}
  \end{pmatrix}
\end{equation*}
for some $J \subseteq \Z/f$. By lemma~\ref{lm:relate_decomp_bdj}, the regular Serre weight $F_{\vec p-\vec 1 -\vec b,\vec b}$ is a constituent
of $\overline{V_\pp(\rho|_{I_\pp})}$. Applying $\RR_\pp$ produces $F_{\vec b-\vec 1,\vec 0}$. Reversing the argument yields the other
inclusion.

The niveau~2 case works exactly the same way.

(ii) \emph{Step 1: Show that $\RR_{\ext,\pp}(F) \subseteq W_\pp^?(\rho)$ if $F$ is a constituent of $\overline{V_\pp(\rho|_{I_\pp})}$.}

Without loss of generality (twisting $\rho$ and~$F$) we may assume that $F = F_{\vec m,\vec 0}$ ($0 \le m_i \le p-1$).
If \emph{$\rho|_{I_\pp}$ has niveau~1}, then by lemma~\ref{lm:relate_decomp_bdj} we can write
\begin{equation}
  \label{eq:rho_gl2_niv1}
  \rho|_{I_\pp} \sim
  \begin{pmatrix}
    \prod_J \psi_i^{p-1-m_i} \\ & \prod_{J^c} \psi_i^{p-1-m_i}
  \end{pmatrix} \prod \psi_i^{m_i}
\end{equation}
for some subset $J \subseteq \Z/f$. Define $\alpha:\Z/f \to \{0,1,\dots,p-1\}$, $i \mapsto p-1-m_i$.  Given $S \in \SS(F)$, we can by
lemma~\ref{lm:succ_of_ints} choose a collection \I\ of signed intervals such that $(\alpha,\I) \in \L_{[0,p-1]}$ and $S$ is the set of
successors of positive \I-intervals. Let $J_+$ (resp.\ $J_-$) denote those elements of~$J$ that succeed positive (resp.\ negative) intervals
of~$\I$. Similarly define $J^c_+$ and $J^c_-$. Let $J_0$ (resp.\ $J^c_0$) denote those elements of~$J$ (resp.\ $J^c$) that do not lie in any
interval of~$\I$. Note that $S = J_+ \cup J^c_+$. Then
\begin{equation}
  \label{eq:rewrite_rho_niv1}
  \rho|_{I_\pp} \sim \begin{pmatrix} \chi_1 \\ &\chi_2 \end{pmatrix} \prod_{S} \psi_i^{-1}\prod_i \psi_i^{m_i}
\end{equation}
where
\begin{align*}
  \chi_1 &= \prod_{J_+} \psi_i^{\alpha(i)+1}\prod_{J_0\setminus (J_+\cup J_-)} \psi_i^{\alpha(i)} \prod_{J_0\cap J_-} \psi_i^{\alpha(i)-1}
  \prod_{\substack{j+1\in J_- \cup J^c_+\\ \ll i,j\rr\in \I}}
  (\psi_{i}^p\psi_{i+1}^{p-1}\cdots \psi_{j}^{p-1})
\end{align*}
and $\chi_2$ is obtained by interchanging the roles of~$J$ and~$J^c$.
Note that each~$\psi_i$ appears with non-zero exponent in precisely one of $\chi_1$, $\chi_2$ (the way they are expressed here); call this
non-zero exponent $\beta(i)$. It is not hard to see that $\phi(\alpha,\I) = (\beta,\I)$.
Thus
\begin{equation*}
  \chi_1 = \prod_L \psi_i^{\beta(i)},\ \chi_2 = \prod_{L^c} \psi_i^{\beta(i)}
\end{equation*}
for some $L \subseteq \Z/f$ and all exponents $\beta(i)$ are in $[1,p]$, so
\eqref{eq:rewrite_rho_niv1} gives rise to a Serre weight $F(A,B) \in W_\pp^?(\rho)$ (by~\eqref{eq:BDJ_niv1}).
Combining equations~\eqref{eq:rho_gl2_niv1} and \eqref{eq:rewrite_rho_niv1} we find that
\begin{equation*}
  \det (\rho|_{I_\pp}\cdot \pprod \psi_i^{-m_i}) = \psi_0^{-\sum m_i p^i} = \psi_0^{\sum (\beta(i) - 2\cdot \ones(i)) p^i}.
\end{equation*}
Using this, we easily see that $F(A,B)$ satisfies
\begin{equation*}
  A \equiv -\sum_{S^c} p^i,\ B \equiv \sum m_i p^i - \sum_{S} p^i \pmod{p^f-1}.
\end{equation*}
We are done except for showing that any other weight $F(A',B')$ satisfying these congruences is in the conjectured weight set.
But these congruences determine $F(A,B)$ except for the pairs
$\{F(x,x), F(x+p^f-1,x)\}$ and for all~$x$, $F(x,x) \in W_\pp^?(\rho)$ if and only if $F(x+p^f-1,x) \in W_\pp^?(\rho)$
(this follows directly from the definition). Therefore $\RR_{\ext,\pp}(F) \subseteq W_\pp^?(\rho)$.

If \emph{$\rho|_{I_\pp}$ has niveau~2}, then
\begin{equation*}
  \rho|_{I_\pp} \sim
  \begin{pmatrix}
    \prod_J \psi_{i'}^{p-1-m_i} \\ & \prod_{J^c} \psi_{i'}^{p-1-m_i} 
  \end{pmatrix} \prod \psi_i^{m_i}
\end{equation*}
for some $J \subseteq \Z/2f$ projecting bijectively onto~$\Z/f$. The argument is now formally identical to the niveau~1 case provided we replace each~$\psi_i$
by $\psi_{i'}$ and ``$\ll i,j\rr \in \I$'' in the subscript of the right-most product in the expression for~$\chi_1$ by ``$\ll i,j\rr \in
\widetilde\I$'', where $\widetilde\I$ is the set of intervals in $\Z/2f$ which project bijectively onto the \I-intervals in~$\Z/f$.

\emph{Step 2: Show that all weights~$F$ in $W_\pp^?(\rho)$ are obtained in this way.}

If \emph{$\rho|_{I_\pp}$ has niveau~1}, then
we can twist by characters and assume without loss of generality that $F = F_{\vec \beta-\vec 1,\vec 0}$ ($1 \le \beta(i) \le p$) and
\begin{equation}\label{eq:rho_bdj_niv1}
  \rho|_{I_\pp} \sim
  \begin{pmatrix}
    \prod_L \psi_i^{\beta(i)} \\ & \prod_{L^c} \psi_i^{\beta(i)}
  \end{pmatrix}
\end{equation}
for some $L \subseteq \Z/f$. Define a collection~$\I$ of disjoint signed intervals in~$\Z/f$ which is in bijection with $\beta^{-1}(p)$, as
follows. Whenever $\beta(i) = p$ and $i \in L$ (resp.\ $L^c$) choose~$j$ such that all numbers in $\beta(\ll i,j\rr-\{i\}) \subseteq \{p-1\}$,
$\ll i,j\rr \subseteq L$ (resp.\ $L^c$) and $j$ is maximal with respect to these properties (\ie, $j$ cannot be replaced by $j+1$).  In that
case $\ll i,j\rr$ is the \I-interval corresponding to $i \in \beta^{-1}(p)$. We let it be negative if and only if $\beta(j+1) = p$ or $j+1
\in L$ (resp.\ $L^c$). Observe that $(\beta,\I) \in \L_{[1,p]}$.

Let~$\Sigma_L$ (resp.\ $\Sigma_{L^c}$) be the set of successors of \I-intervals contained in~$L$ (resp.\ $L^c$).
The notation $L_0$, $L^c_0$ has the same meaning as in the previous part. Note that $S = \Sigma_L \cap L_0^c \cup \Sigma_{L^c} \cap L_0$
is the set of successors of positive \I-intervals. We see that
\begin{equation}\label{eq:rewrite_rho_bdj_niv1}
  \rho|_{I_\pp} \sim
  \begin{pmatrix} \chi_1 \\ & \chi_2
  \end{pmatrix} \prod_{S} \psi_i
\end{equation}
where
\begin{equation*}
  \chi_1 = \prod_{L_0 \cap \Sigma_{L^c}} \psi_i^{\beta(i)-1} \prod_{L_0 \setminus (\Sigma_{L}\cup \Sigma_{L^c})} \psi_i^{\beta(i)} 
  \prod_{L_0 \cap \Sigma_{L}} \psi_i^{\beta(i)+1} \prod_{\Sigma_L\setminus (L_0 \cup L_0^c)} \psi_i
\end{equation*}
and $\chi_2$ is obtained by interchanging the roles of~$L$ and~$L^c$.  Every~$\psi_i$ occurs with a non-zero exponent in at most one of
$\chi_1$, $\chi_2$ (the way they are expressed here); call this exponent $\alpha(i) \in \{0,1,\dots,p-1\}$. By lemma~\ref{lm:relate_decomp_bdj},
taking into account the twist, this decomposition shows that $F' = F_{\vec p-\vec 1-\vec\alpha,\vec\alpha+\vec\ones}$ is a constituent of
$\overline{V_\pp(\rho|_{I_\pp})}$ (here $\ones$ is the characteristic function of~$S$).

It is not hard to see that $\phi^{-1}(\beta,\I) = (\alpha,\I)$. In particular, by lemma~\ref{lm:succ_of_ints} $S \in \SS(F')$.
Equations~\eqref{eq:rho_bdj_niv1}, \eqref{eq:rewrite_rho_bdj_niv1} yield
\begin{equation*}
  \det (\rho|_{I_\pp}) = \psi_0^{\sum (\alpha(i) + 2\cdot \ones(i))p^i} = \psi_0^{\sum \beta(i)p^i}.
\end{equation*}
We see that the weight in 
$\RR_{\ext,\pp}(F')$ corresponding to $S \in \SS(F')$ is $F_{\vec \beta-\vec 1,\vec 0} = F$, and we are done.

If \emph{$\rho|_{I_\pp}$ has niveau~2}, the argument is completely analogous (as in Step~1). \end{proof}

\begin{proof}[Proof of lemma~\ref{lm:relate_decomp_bdj}] (i) First let us show the implication ``$\Rightarrow$''.
Without loss of generality,
\begin{equation*}
  \rho|_{I_\pp} \sim
  \begin{pmatrix}
    \prod \psi_i^{n_i} \\ &1
  \end{pmatrix}
\end{equation*}
for some $0 \le n_i \le p-1$. By~\cite{bib:Di}, prop.~1.1, the constituents of $\overline{V_\pp(\rho|_{I_\pp})}$ are the $F_{\vec c_J,\vec d_J}$ where
$J \subseteq \Z/f$ and
\begin{gather*}
  c_{J,i} =
  \begin{cases}
    n_i + \delta_J(i) -1 & \text{if $i \in J$}\\
    p-1-n_i-\delta_J(i) & \text{if $i \not\in J$}
  \end{cases}\\
  d_{J,i} =
  \begin{cases}
    0 & \text{if $i \in J$}\\
    n_i+\delta_J(i) & \text{if $i \not\in J$}
  \end{cases}
\end{gather*}
where~$\delta_J$ is the characteristic function of $\{i+1: i \in J\}$. Also, the convention is that $F_{\vec c_J,\vec d_J} = (0)$ if
$c_{J,i} = -1$ for some~$i$. Now note that
\begin{equation*}
  \rho|_{I_\pp} \sim
  \begin{pmatrix}
    \prod_{J^c} \psi_i^{n_i+\delta_J(i)} \\ & \prod_J \psi_i^{p-n_i-\delta_J(i)}
  \end{pmatrix}\prod_J \psi_i^{n_i+\delta_J(i)-1}\prod_{J^c} \psi_i^{p-1}.
\end{equation*}

Conversely, suppose without loss of generality that $\rho|_{I_\pp}$ is as in the statement of the lemma with $\vec b = 0$.
Note that whenever $m_i = p-1$ it is irrelevant whether $i \in J$ or not. Thus for all such~$i$ we can prescribe whether or not $i \in J$. There is
a unique way to alter~$J$ in this manner such that for all~$i$ with $m_i = p-1$, $i \in J \eq i-1 \in J$ (the latter is equivalent to $\delta_J(i) = 1$).
Note that
\begin{align*}
  V_\pp(\rho|_{I_\pp}) &\cong I(\pprod_{J^c}\lambda_i^{p-1-m_i}\pprod_J \lambda_i^{m_i+1-p},1) \tens \pprod \lambda_i^{m_i} \pprod_J \lambda_i^{p-1-m_i} \\
  &\cong I(\pprod_{J^c} \lambda_i^{p-1-m_i-\delta_J(i)}\pprod_J \lambda_i^{m_i+1-\delta_J(i)},1) \tens \pprod \lambda_i^{m_i} \pprod_J
  \lambda_i^{p-1-m_i}.
\end{align*}
By our choice of~$J$, all exponents of the first character in the induction are contained in $\{0,1,\dots,p-1\}$.  It follows
from~\cite{bib:Di}, prop.~1.1 (using the same subset~$J$) that $F_{\vec m,\vec 0}$ is a constituent of $\overline{V_\pp(\rho|_{I_\pp})}$, as
required.

(ii) This works completely analogously, it is only more cumbersome to write out. Note that we can assume $\vec m \ne \vec p -\vec 1$
as on the one hand
\begin{equation*}
  \dim F_{\vec p-\vec 1, \vec b} = p^f > p^f-1 = \dim V_\pp(\rho|_{I_\pp})  
\end{equation*}
and on the other hand $\rho|_{I_\pp}$ cannot be unramified up to twist (being of niveau~2).
\end{proof}

\begin{proof}[Proof of lemma~\ref{lm:combi_bij}] This is straightforward.
\end{proof}

\begin{proof}[Proof of lemma~\ref{lm:succ_of_ints}] 
Note that the first two statements are equivalent, by the definition of $\SS(F)$, to

\begin{enumerate}
\item [(i$'$)] For all $s\in S$,
   \begin{enumerate}
   \item $\alpha(s) \ne p-1$.
   \item There is an $i \in \Z/f$ such that $\ll i,s-1\rr \cap S = \varnothing$ and
     $\alpha(i) = 0$, $\alpha(i+1) = \dots = \alpha(s-1) = 1$.
   \end{enumerate}
\end{enumerate}
We will now show that $\text{(i$'$)} \eq \text{(iii)}$.

First suppose that $(\alpha,\I) \in \L_{[0,p-1]}$ and let $S$ be the set of successors of positive
intervals. Then by property~A4, $\alpha(s) \ne p-1$ if $s\in S$. Moreover, $\alpha(s-1) \in \{0,1\}$
and $s-1 \not\in S$ (as $s-1$ is in an interval).
If it is 1, by property~A2 the preceding entry lies in a different (negative) interval and iterating this process gives the desired interval
$\ll i,s-1\rr$. Note that the process has to stop (\ie, eventually we hit a~0) because $s \in S$ cannot itself lie in an interval (by A4).

Conversely, suppose given $S$ satisfying $\text{(i$'$)}$. Here is a way to define $\I$ having $S$ as set of
successors of positive intervals and such that $(\alpha,\I) \in \L_{[0,p-1]}$ (in fact it is the unique way).
It is easier to define $\bigcup \I$ first: we let $i \in \bigcup \I$ if and only if there is a~$j$ such that
$\ll j,i\rr \subseteq S^c$ and $\alpha(j) = 0$, $\alpha(j+1) = \dots = \alpha(i) = 1$ (in particular, this whole interval
will be contained in $\bigcup \I$). We let $i \in \bigcup \I$ be the start point of an \I-interval if and only if
$i-1 \not\in \bigcup \I$ or $i-1 \in \bigcup \I$ and $\alpha(i) = 1$. We let an \I-interval be positive if and only
if its successor is in~$S$.
It is straightforward to see that $(\alpha,\I) \in \L_{[0,p-1]}$; by definition $S$ is the set of successors of positive intervals.
\end{proof}

\appendix
\section{Generalisation of Jantzen's formula}\label{app:gener-jantz-form}

The purpose of this appendix is to explain how Jantzen's theorem on the decomposition of the reduction modulo~$p$ of
Deligne--Lusztig representations generalises to the case of reductive groups whose derived subgroup is simply connected. The
case of simply connected almost simple groups is treated in Jantzen's original paper~\cite{bib:Jan-DL}, and the case of split
reductive groups with simply connected derived subgroup was explained to the author by Jantzen in an informal yet very
carefully written manuscript. Below we take the ``fibre product'' of Jantzen's paper and his subsequent manuscript to give a
proof of the result in the general case. This doesn't require any new ideas, but is presented here for the sake of
completeness.

The argument follows that of~\cite{bib:Jan-DL}, and we will simply explain what changes need to be made to that argument.
As much as possible we will keep with the notation of that paper, including the numbering of sections and references.  Since
we are only interested in the decomposition result~\cite{bib:Jan-DL}, thm.~3.4, we will not comment on section~4 and a couple of
aside remarks like the one at the end of (2.5).

\emph{Acknowledgements.} I am very grateful to Jens Carsten Jantzen for explaining his proof and for allowing me to write it up
in this appendix. All results in this write-up are due to Jantzen; the author takes responsibility for all errors.

\medskip

1.1. Let~$G$ be a connected reductive algebraic group defined and split over~$\fp$ and such that its derived subgroup $G'$ is
\emph{simply connected}. Then $T_1 = T \cap G'$ is a split maximal torus in~$G'$ (its connectedness follows by comparing the
Bruhat decompositions of~$G$ and~$G'$). The restriction map $X(T) \onto X(T_1)$, which identifies the roots and the Weyl
groups of~$G$ and~$G'$, will be denoted by $\mu \mapsto \overline\mu$ and its kernel by $X^0(T)$.  Note that $X^0(T) = \{\mu
\in X(T) : \langle \mu, \alpha\dual \rangle = 0\ \forall \alpha \in R\} = X(T)^W$.

Let $R^+$ denote the set of positive roots. Since $G'$ is simply connected, for any simple root $\alpha \in B$ 
there exists $\omega'_\alpha \in X(T)$ 
satisfying $\langle \omega'_\alpha, \beta\dual \rangle = \delta_{\alpha\beta}$ for all $\beta \in B$. Equivalently,
$\omega'_\alpha$ is a choice of lifting of the fundamental weight $\omega_\alpha$ of~$G'$. Let $\rho' = \sum_{\alpha \in B}
\omega'_\alpha$. In particular, $\rho' - \frac 12 \sum_{R^+} \alpha \in X^0(T)\otimes\R$, and
for $w \in \widetilde W_p$
and $\lambda \in X(T)$, $w\cdot \lambda = w(\lambda+\rho')-\rho'$ is well defined.
\emph{Any occurrence of~$\rho$ in the text should be read as~$\rho'$.}

Define $\alpha_0\dual \in X(T)\dual$ to be the sum of the longest coroots of all irreducible components of~$R$. It is thus
generally not a coroot. If $\lambda \le \mu$ in $X(T)$ then $\langle \lambda,\alpha_0\dual\rangle \le \langle
\mu,\alpha_0\dual\rangle$ (the strict inequality in~\cite{bib:Jan-DL} is a typo), and for $\lambda \in X(T)^+$, $\langle
\lambda,\alpha_0\dual\rangle \ge 0$ with equality if and only if $\lambda \in X^0(T)$.

\medskip

1.2. Note that for $\mu \in X^0(T)$, $L(\mu) = V(\mu)$ is a one-dimensional $G$-module with formal character $e(\mu)$
(see the proof of prop.~1.3 below); denote it by~$\mu$ if no confusion arises. It follows from the definitions that
$V(\lambda+ \mu) \cong V(\lambda) \otimes \mu$, $L(\lambda+ \mu) \cong L(\lambda) \otimes \mu$
for any $\lambda \in X(T)$.

\medskip

1.3. Now $\pi$ is a finite order automorphism of the based root datum of~$G$; note that it preserves $\alpha_0\dual$ and
$X^0(T)$.  We may lift~$\pi$ to an automorphism~$\pi$ of $(G,B^+,T)$ that is of the same order and that  is defined over~$\fp$ (where
$B^+$ is the Borel subgroup determined by~$R^+$).  This follows from \cite[16.3.2]{bib:Springer_LAG} (or \cite[II.1.15]{bib:Jan-reps})
by using one fixed realisation $(u_\alpha)_\alpha$ 
for~$G$ in the proof, so that the lifted automorphism fixes a pinning. Note that this procedure
induces a bijection between conjugacy classes of finite order automorphisms of the based root datum
and isomorphism classes of $\fpn n$-forms of~$G$. Also note that $\pi$ induces $\fpn n$-structures on~$G'$, $G/G'$, etc. We
let $\Gamma'_n = (G')^F \le \Gamma_n$.

We have the following classification of simple $K\Gamma_n$-modules. For lack of a reference we explain how it follows from
the semisimple case \cite[2.11]{bib:Hum}.

\begin{propJ} \
  \begin{enumerate}
  \item For all $\lambda \in X_n(T)$, the simple $G$-module $L(\lambda)$ restricts to a simple $K\Gamma_n$-module.
    Each simple $K\Gamma_n$-module is isomorphic to such a restricted $L(\lambda)$.
  \item Let $\lambda$, $\lambda' \in X_n(T)$. Then $L(\lambda)$ and $L(\lambda')$ are isomorphic as $K\Gamma_n$-modules
    if and only if $\lambda-\lambda' \in (p^n-\pi)X^0(T)$.
  \end{enumerate}
\end{propJ}

\begin{proof}
  Any $L(\lambda)$ with $\lambda \in X(T)^+$ restricts to the simple $G'$-module $L(\overline\lambda)$, as $G = Z(G) \cdot G'$.
  If $\lambda \in X_n(T)$, then $L(\overline\lambda)$ is simple as $K\Gamma'_n$-module and so $L(\lambda)$ is simple
  as $K\Gamma_n$-module. The result in the semisimple case implies furthermore that for any~$\lambda$, $\lambda' \in X_n(T)$,
  $L(\lambda) \cong L(\lambda')$ as $K\Gamma'_n$-modules if and only if $\lambda-\lambda' \in X^0(T)$.

  Let $U^+$ denote the unipotent radical of the Borel subgroup $B^+$. As $U^+ \subset
  G'$, it is known that $L(\lambda)^{(U^+)^F} = L(\lambda)^\lambda$ \cite[2.11]{bib:Hum}. Thus $T^F$ acts on this space via the
  restriction of~$\lambda$ to~$T^F$; so if $\lambda$, $\lambda' \in X_n(T)$ and $L(\lambda) \cong L(\lambda')$ as
  $K\Gamma_n$-modules, then $\lambda-\lambda'$ is trivial on~$T^F$. By Lang's theorem there is a short exact
  sequence of diagonalisable groups, $1 \to T^F \to T \xrightarrow{F-1} T \to 1$, and by taking character groups it follows
  that $\lambda-\lambda' \in (p^n-\pi) X(T)$ (as $\Frn = p^n$ on~$T$). Let us write $\lambda-\lambda' = (p^n-\pi)\mu$; by the
  above this weight also lies in $X^0(T)$. If $d \ge 1$ denotes the order of~$\pi$, it follows that $(p^{nd}-1)\mu \in
  X^0(T)$ and thus $\lambda-\lambda' \in (p^n-\pi)X^0(T)$. This proves the ``only if'' direction of (ii).
  
  For the converse, since $L(\lambda+\mu) \cong L(\lambda)\otimes \mu$ for $\mu \in X^0(T)$, it suffices to show that
  $L((p^n-\pi)\mu)$ is trivial on~$\Gamma_n$ for $\mu \in X^0(T)$. Let~$\overline T$ denote the torus $G/G'$. By considering
  the short exact sequence of tori, $1 \to T_1 \to T \to \overline T \to 1$, it follows that $X(\overline T) = X^0(T)$.
  Moreover $G \onto \overline T \xrightarrow{\mu} \G_m$ has to be the irreducible $G$-module~$L(\mu)$.
  As above, $(p^n-\pi)\mu \in (p^n-\pi)X(\overline T)$ is trivial on $\overline T^F$, hence $L((p^n-\pi)\mu)$ is trivial
  on $G^F = \Gamma_n$. (Note that $\pi$ acts compatibly on~$T$ and~$\overline T$.)
  This proves the ``if'' direction of (ii).

  The argument so far shows that each simple $K\Gamma'_n$-module $L(\overline\lambda)$ ($\lambda \in X_n(T)$) has at least
  $\#(X^0(T)/(p^n-\pi)X^0(T))$ non-isomorphic extensions to a simple $K\Gamma_n$-module. Each extension is a quotient of
  $\Ind_{\Gamma'_n}^{\Gamma_n} L(\overline\lambda)$. By Lang's theorem we have the short exact sequences
  \begin{gather*}
      1 \to (G')^F \to G^F \to \overline T^F \to 1, \\
      1 \to \overline T^F \to \overline T \xrightarrow{F-1} \overline T \to 1,
  \end{gather*}
  and by applying character groups to the second sequence we obtain $[\Gamma_n : \Gamma'_n] = \#(X^0(T)/(p^n-\pi)X^0(T))$.
  For dimension reasons it follows that $\Ind_{\Gamma'_n}^{\Gamma_n} L(\overline\lambda)$ is a direct sum of all $L(\lambda +
  \mu)$ with $\mu$ running over representatives of $X^0(T)/(p^n-\pi)X^0(T)$. Since each simple $K\Gamma_n$-module is a
  homomorphic image of a module induced from a simple $K\Gamma'_n$-module, this proves (i).
\end{proof}

In the inside sum of (2), $\lambda$ should run over a system of representatives~$Z$ of $X_n(T)/p^n X^0(T)$;
then every dominant weight can be expressed uniquely as $p^n \nu + \lambda$ with $\nu \in X(T)^+$ and $\lambda \in Z$.

\medskip

1.4. Both sums in the lemma involve only finitely many non-zero terms (see the comment in the proof of lemma~2.3 below).

Fix a system of representatives~$Z$ as at the end of the last paragraph. By shifting the index~$\mu$ and by adjusting
$\chi_2$ we may assume without loss of generality that $\lambda \in Z$. Then the proof goes through, provided that
$\lambda'$ runs through $Z$, rather than $X_n(T)$.

\medskip

1.5. Denote by $\St'_n$ the simple $G$-module $L((p^n-1)\rho') = V((p^n-1)\rho')$ and by $\St_{n,\pi}$ the simple
$G$-module $L((p^n-\pi)\rho') = V((p^n-\pi)\rho')$. Thus $\St'_n \cong \St_{n,\pi} \otimes (\pi-1)\rho'$ since $(\pi-1)\rho'
\in X^0(T)$. The first will be useful in the context of $G$-modules, the second when dealing with $K\Gamma_n$-modules. As
$K\Gamma_n$-modules they are simple 
by prop.~1.3. Note that as $K\Gamma_n$-module,
$\St'_n$ may depend on the choice of the $\omega'_\alpha$, whereas $\St_{n,\pi}$ is independent of it. 
Observe that $\St'_n \cong \St_{n,\pi}$ automatically in the split case ($\pi = 1$) or if $G$ is
semisimple (as $X^0(T) = 0$). \emph{Any occurrence of the $G$-module $\St_n$ in the text should be read as $\St'_n$
in sections~1 and~2 and as $\St_{n,\pi}$ in section~3.}

For the proof of the theorem, the first case is now $\nu \in X^0(T)$. Using $(p^n-1)\rho' + p^n \mu \le \pi(\mu) + p^n\nu + \lambda$
and $\langle \lambda,\alpha_0\dual\rangle \le \langle (p^n-1)\rho',\alpha_0\dual\rangle$ it follows that $\mu \in X^0(T)$.
Since $\chi\chi_p(\pi(\mu)) = \chi_p(p^n\nu+\lambda+\pi(\mu))$, either side of the claimed equation equals~1 if
$p^n\nu + \lambda$ is congruent to $(p^n-1)\rho'$ modulo $(p^n-\pi)X^0(T)$, and 0 otherwise.
The remaining case follows as is written, once ``$\nu \ne 0$'' is replaced with ``$\nu \not\in X^0(T)$.''

\medskip

2.1. References [9] and [10] have mostly been superseded by Jantzen's book~\cite{bib:Jan-reps}, II.9 and II.11. To keep with the book,
we will use ``$G_nT$-module'' instead of ``$\text{$\mathbf{u}_n$-$T$}$-module,'' and the notation $\widehat L_n(\lambda)$,
$\widehat Z_n(\lambda)$, $\widehat Q_n(\lambda)$.

Note that for $\mu \in X^0(T)$, $\widehat L_n(\mu)$ is one dimensional and has character $e(\mu)$. Denote it by~$\mu$. Then
for all $\lambda \in X(T)$,
\begin{equation*}
  \widehat L_n(\lambda+\mu) \cong \widehat L_n(\lambda)\otimes \mu,\   \widehat Z_n(\lambda+\mu) \cong \widehat Z_n(\lambda)\otimes \mu,\ 
  \widehat Q_n(\lambda+\mu) \cong \widehat Q_n(\lambda)\otimes \mu.
\end{equation*}

\medskip

2.2. In equation~(1), the right-hand side should be replaced with
\begin{equation*}
  \begin{cases}
    \dim L(\nu)^{\nu'+\nu_0} & \text{if $\mu-\lambda = p^n\nu_0 \in p^n X^0(T)$} \\
    0 & \text{otherwise}
  \end{cases}
\end{equation*}

\medskip

2.3. Everything goes through except showing that only finitely many terms are non-zero.  Suppose $\mu$, $\nu$ are dominant
weights making the term in~(1) non-zero. Then $\nu \le \mu$ and $p^n \mu + \lambda \le \mu' + \pi(\nu)$ for some
weight~$\mu'$ of~$\chi$. Then $(p^n-\pi)\mu \le \mu' - \lambda$.  Note that $p^{dn}-1 = (\sum_{i=0}^{d-1}
p^{in}\pi^{d-1-i})(p^n-\pi)$ where $d \ge 1$ is the order of~$\pi$.  Thus $(p^{dn}-1)\mu$ is dominant and bounded for the
$\le$ partial order; so there are only finitely many choices for~$\mu$, a fortiori the same is true for~$\nu$.  Similarly one
shows that the term in~(2) is non-zero for only finitely many pairs $(\mu,\nu)$.

\medskip

2.4. On top of p.~460 the equation should be replaced by
\begin{equation*}
  \langle \ch \widehat Q_n(\lambda),\ch \widehat L_n(\mu)\rangle = 
  \begin{cases}
    e(\mu-\lambda) & \text{if $\mu-\lambda \in p^n X(T)$} \\
    0 & \text{otherwise},
  \end{cases}
\end{equation*}
for $\lambda$, $\mu \in X(T)$.

\medskip

2.5. To see that $\widehat Z_n((p^n-1)\rho') \cong \St'_n$, compare their formal characters using \cite{bib:Jan-reps}, II.5.10 and
II.9.2(3) and note that $A(p^n\rho') = A(\rho')^{\Frn}$ and $A(\rho') = e(\rho')\prod_{\alpha \in R^+} (1-e(-\alpha))$. 
Then $\ch \widehat Z_n(\mu) = e(\mu-(p^n-1)\rho')(\ch \St'_n)$ follows immediately from \cite[II.9.2(3)]{bib:Jan-reps}.
For the reciprocity law see \cite[II.11.4]{bib:Jan-reps}.  The result quoted from [10,\,3.2(1)] follows easily by adapting the proof
of \cite[II.9.16(a)]{bib:Jan-reps} using $\rho'$ instead of~$\rho$ and by noticing that the formula there is valid for all $\mu_0
\in X_r(T)$. \need{can in fact avoid covering groups referred to in II.9.16 by using different decomposition for $X(T)$;
see his manuscript, lemma 2.5}

\medskip

2.7. Let~$Y$ denote a set of representatives for $X_n(T)/(p^n-\pi)X^0(T)$. Then the sum in the first formula should
run over $\lambda \in Y$, and similarly the $\Psi L(\lambda)$ for $\lambda \in Y$ are linearly independent. For
the projectivity of $\St'_n$ as $K\Gamma_n$-module see the comments on (3.2) below. Also
\begin{equation*}
  \langle \Psi U(n,\lambda),\Psi L(\mu) \rangle = 
  \begin{cases}
    1 & \text{if $L(\lambda) \cong L(\mu)$ as $K\Gamma_n$-modules}, \\
    0 & \text{otherwise}.
  \end{cases}
\end{equation*}
One defines $[\widehat Q_n(\lambda) : U(n,\mu)]$ first for $\mu \in Y$ by using the same definition as in the text,
but with the sum running over $\mu \in Y$. Then one defines it in general by demanding that it depends on~$\mu$ only
modulo $(p^n-\pi)X^0(T)$. It is clearly independent of the choice of~$Y$.

\medskip

2.8 and 2.9. The sums over~$\mu$ should run over~$Y$ (rather than $X_n(T)$).

\medskip

2.10. In the corollary, ``$\lambda \ne \mu$'' should be replaced by ``$\lambda-\mu \not\in (p^n-\pi)X^0(T)$.''
In the proof, the terms for $\nu \in X^0(T)$ contribute 1 if $\mu-\lambda \in (p^n-\pi)X^0(T)$ and 0 otherwise.
The other case, now $\nu \not\in X^0(T)$, goes through as written.

\medskip

3.2. \emph{As pointed out in (1.5), from now on all occurrences of $\St_n$ in the text should be read as $\St_{n,\pi}$.
In this paragraph, any expression of the form ``$\widehat Z_n(\dots-\rho)$'' should be read as
``$\widehat Z_n(\dots-\pi\rho)$.''}

Jantzen establishes Humphreys' formula in great generality, following a suggestion of Lusztig. For the purpose of this
proposition only, $G$ denotes a connected reductive group defined over~$\fpn n$ and $T$ an arbitrary maximal torus of~$G$ 
that is defined over~$\fpn n$.  Let $F$ be the corresponding Frobenius endomorphism.
Note that to any $\chi \in \Z[X(T)]^W$ we can associate a Brauer character $\Psi\chi$ of~$G^F$ just as in (2.7). The point is
that any $G$-module can be restricted to a $KG^F$-module and that $\Psi$ is additive.

\begin{propJ}
  With the above notation,
  \begin{equation*}
    \sum_{w \in W} R_w(n,\mu) = (\# \Stab_W \mu) \Psi s(\mu) \St_G,
  \end{equation*}
  where $\St_G$ is the Steinberg character of~$G^F$ \cite[6.2]{bib:Carter}.
\end{propJ}

Note that $G^F = \Gamma_n$ and $\St_G = \Psi\! \St_{n,\pi}$ in the context above. This can be seen as follows. By \cite{bib:Carter},
6.2, 6.4.3, and 2.9, $\dim \St_G = \#((U^+)^F) = p^{n(\# R^+)}$, where $U^+$ is the unipotent radical of the Borel $B^+$.
Since $(U^+)^F$ is a Sylow $p$-subgroup in~$\Gamma_n$, the Brauer--Nesbitt theorem \cite[16.6]{bib:Hum} implies that
$\overline{\St_G}$, the reduction modulo~$p$ of $\St_G$, is irreducible and projective.  A short calculation with the Weyl
dimension formula shows that $\dim V(\lambda) \le p^{n(\# R^+)}$ for all $\lambda \in X_n(T)$ with equality if and only if
$\langle\lambda,\alpha\dual\rangle = p^n-1$ for all simple roots~$\alpha$. By prop.~1.3, $\overline{\St_G} \cong L(\lambda)$
for some such~$\lambda$. As $\St_G$ is trivial on~$T^F$ by definition, $\lambda \in (p^n-\pi)X(T)$ and the claim follows
easily. (This argument together with \cite[8.2]{bib:Hum} shows moreover that $L(\lambda)$ for $\lambda \in X_n(T)$ is projective
as $K\Gamma_n$-module if and only if $\langle\lambda,\alpha\dual\rangle = p^n-1$ for all simple roots~$\alpha$.)

\begin{proof}
  Let $(T_w^F)\dual$ denote that set of irreducible complex characters of~$T_w^F$ and let $\langle\;,\,\rangle_{T_w^F}$
  denote the usual inner product on the space of class functions. For any complex class function~$\chi$ on~$G^F$,
  \begin{equation}\tag{\text{$*$}}\label{eq:app}
    \chi\St_G = \frac 1{\# W} \sum_{w\in W} \sum_{\eta \in (T_w^F)\dual} \langle \chi,\eta\rangle_{T_w^F}\, 
    \varepsilon_G \varepsilon_{T_w} R^\eta_{T_w},
  \end{equation}
  where $\varepsilon_G = (-1)^{\text{$\fpn n$-$\mathrm{rank} (G)$}}$ and similarly for $\varepsilon_{T_w}$, and their
  product is the sign that makes $R^\eta_{T_w}$ positive at 1 \cite[7.5.1]{bib:Carter}. This is essentially the content of
  [4,\,7.12.2] and can be seen as follows. By [4,\,7.5], $\chi\St_G$ is a linear combination of Deligne--Lusztig characters.
  To determine the coefficients one uses the calculation of the inner product on top of p.~144 in~[4].

  To determine $\Psi s(\mu)$, note that for any $p$-regular $s \in G^F$ there exists a $t \in T$ that is conjugate to~$s$ in~$G$. Then
  \begin{equation*}
    \Psi s(\mu)(s) = \sum_{\nu \in W\mu} (\Theta \circ \nu)(t),
  \end{equation*}
  where $\Theta$ is the same embedding of the roots of unity in~$K$ into~$\C\s$ that was used implicitly in (2.7) and (3.1).
  To prove this, note that for any $G$-module~$V$, $s$ and~$t$ have the same set of eigenvalues~$\lambda$ on~$V$. Thus
  \begin{equation*}
    \Psi V(s) = \sum_{\lambda} \Theta(\lambda) = \sum _{\nu \in X(T)} (\Theta \circ \nu)(t) \dim V^\nu,
  \end{equation*}
  and the formula follows by taking linear combinations.

  In particular, $\langle \Psi s(\mu),\eta \rangle_{T_w^F} = \sum_{\nu \in W\mu} \langle \theta(\nu,w),\eta \rangle_{T_w^F}$.
  Applying~\eqref{eq:app} to $\chi = \Psi s(\mu)$ and using $\langle \theta(\nu,w),\eta \rangle_{T_w^F} = \delta_{\theta(\nu,w),\eta}$  yields
  \begin{equation*}
    \Psi s(\mu) \St_G = \frac 1 {\# W} \sum_{w\in W} \sum_{\nu \in W\mu} R_w(n,\nu).
  \end{equation*}
  The right-hand side may be rewritten as
  \begin{equation*}
    \frac {\#(W\mu)} {(\# W)^2} \sum_{w\in W}\sum_{w_1\in W} R_w(n,w_1\mu) =
    \frac {\#(W\mu)} {(\# W)^2} \sum_{w\in W}\sum_{w_1\in W} R_{w_1^{-1}w F(w_1)}(n,\mu),
  \end{equation*}
  where we used that a Deligne--Lusztig character $R^\theta_T$ depends only on the $G^F$-conjugacy class of $(T,\theta)$
  (see also (3.1)). The proposition now follows by interchanging the order of summation.
\end{proof}

Note that the formula just after (1) follows from (2.5(1), (2)) after shifting the index~$\mu$ in the sum by\
$(\pi-1)\rho' \in X^0(T)$.

Regarding the reference [10,\,3.2(1)], please see the remark in (2.5) above.

\medskip

3.3. \emph{In this paragraph, any expression of the form ``$\widehat Z_n(\dots-\rho)$'' should be read as
``$\widehat Z_n(\dots-\pi\rho)$.'' Similarly for ``$\chi(\dots-\rho)$,'' with the exception of the very first
formula.}

The weights $\rho'_w$ and $\varepsilon'_w$ are defined as in the text, but depend now on the choice of the $\omega'_\alpha$.
Also the definition of $\gamma'_{w,w'} \in \Z[X(T)]^W$ carries over for the following reason. A result
of Hulsurkar, recalled in [8,\,p.\,448],  implies that the matrix $(\chi(-\varepsilon_{w_0w} + \varepsilon_{w'} - \rho)\det(w'))_{w,w'}$
for the simply connected group $G'$ with entries in $\Z[X(T_1)]^W$
is upper triangular and unipotent for a suitable ordering of~$W$. Since for $\lambda \in X(T)$,
\begin{equation*}
  \chi(\lambda) = 0 \ \eq\  \langle \lambda+\rho',\alpha\dual\rangle = 0 \ \forall \alpha \in R \ \eq\  \chi(\overline\lambda) = 0
\end{equation*}
and $\overline{\chi(\lambda)} = \chi(\overline\lambda) = 1$ if and only if $\lambda \in X^0(T)$ (in which case $\chi(\lambda)
= e(\lambda)$), it follows that also the lifted matrix $(\chi(-\varepsilon'_{w_0w} + \varepsilon'_{w'} -
\rho')\det(w'))_{w,w'}$ is upper triangular with invertible diagonal entries, under the same ordering of~$W$. \emph{Any
  occurrence of~$\rho_w$, $\varepsilon_w$, $\gamma_{w,w'}$ in the text should be read as $\rho'_w$, $\varepsilon'_w$,
  $\gamma'_{w,w'}$.}

Here is how $\rho'_w$, $\varepsilon'_w$, and $\gamma'_{w,w'}$ depend on the choice of the $\omega'_\alpha$. For another choice
$\omega''_\alpha = \omega'_\alpha + \xi_\alpha$ ($\xi_\alpha \in X^0(T)$) let $\xi_w \in X^0(T)$ be the sum of $\xi_\alpha$ for all~$\alpha$
with $w^{-1}\alpha < 0$. Then $\rho''_w = \rho'_w + \xi_w$,  $\varepsilon''_w = \varepsilon'_w + \xi_w$, and
\begin{equation*}
  \gamma''_{w,w'} = \gamma'_{w,w'}\, e(\xi_{w_0w'} - \xi_w + \xi_{w_0}).
\end{equation*}

The statement and proof in [9,\,5.2] carry over word by word with $q = p^n$ (adding primes, as usual). Then~(1) follows by plugging in $\lambda = \mu-\pi\rho'$
and by using the character formula of (2.5) on the left-hand side.

We define for any $w \in W$ and $\mu \in X(T)$,
\begin{equation*}
  X'_w(n,\mu) = \sum_{w_1,w_2 \in W} \gamma^{\prime\,\Frn}_{w_1,w_2} \chi(w_1(\mu-w\pi\varepsilon'_{w_0w_2}) + p^n\rho'_{w_1}-\pi\rho'),
\end{equation*}
an element of $\Z[X(T)]^W$.
By the formulae just given, it is easy to see that $[X'_w(n,\mu):L]_{\Gamma_n}$ for a simple $K\Gamma_n$-module~$L$
is independent of all choices.

The proof of the lemma goes through. The formula of Brauer quoted from [6,\,p.\,38] is a simple exercise using the Weyl
character formula. \need{or \cite[II.5.8]{bib:Jan-reps}.}A slight simplification can actually be achieved in the middle of
p.~467 by choosing $w'$ so that $w'\nu$ is dominant, yielding right away that $b$ equals the sum of
\begin{equation*}
  \frac{\#(W\nu)}{\# W} 
  [\gamma^{\prime\,\Frn}_{w_1,w_2} \chi(w_1(\mu-w\pi\varepsilon'_{w_0w_2}) + w'\pi\nu + p^n\rho'_{w_1}-\pi\rho'):
  \widehat L_n(p^n\nu + \lambda)]
\end{equation*}
with $w_1$, $w_2$ and $w'$ running over~$W$ and~$\nu$ over $X(T)^+$, which together with Brauer's formula completes the proof.

\medskip

3.4. The sum in~(1) now runs over $\lambda \in Y$ (with $Y$ as in~(2.7)). To define $[\widetilde\zeta : L(\lambda)]_{\Gamma_n}$ for
general~$\lambda \in X_n(T)$, one demands that it depends on~$\lambda$ only modulo $(p^n-\pi)X^0(T)$. In this way the definition is seen
to be independent of all choices.

\begin{thmJ}
  For all $w \in W$ and all $\mu \in X(T)$, $\widetilde {R_w}(n,\mu) = \Psi(X'_w(n,\mu))$.
\end{thmJ}

In the proof of the theorem, one restricts $\lambda$, $\lambda_1 \in X_n(T)$ to be elements of~$Y$ everywhere. Note that by choosing
an ordering~$\le_Y$ of~$Y$ such that $\lambda \mathbin{\le_Y} \lambda_1$ implies
$\langle \lambda,\alpha_0\dual\rangle \le \langle \lambda_1,\alpha_0\dual\rangle$ it still follows that the matrix
of all $[\widehat Q_n(\lambda_1):U(n,\lambda)]$ is invertible, as it is unipotent by (2.10).

One slight simplification is possible. It is not necessary to introduce~$\zeta$ in the two formulas
at the top of p.~469; rather $\Psi\ch \widehat Q_n(\lambda_1)$ is the sum of 
\begin{equation*}
   [X'_{w'}(n,\mu')s(\pi\nu) : \widehat L_n(p^n\nu + \lambda_1)] \frac{R_{w'}(n,\mu')}{\langle R_{w'}(n,\mu'),R_{w'}(n,\mu')\rangle}
\end{equation*}
with $(w',\mu')$ and~$\nu$ running over the same index sets as in the text.

\bibliography{my}
\bibliographystyle{halpha}

\end{document}